\documentclass[11pt]{amsart}

\usepackage{amsmath,amsthm,amsfonts,amssymb,bm,graphicx }

\usepackage
{hyperref}
\hypersetup{colorlinks=true,citecolor=blue,linkcolor=blue,urlcolor=blue,
pdfstartview=FitH, pdfauthor=Valentin Blomer and Djordje Milicevic, pdftitle=The second moment of twisted modular \texorpdfstring{$L$}{L}-functions}

\theoremstyle{plain}
\newtheorem{theorem}{Theorem}
\newtheorem{lemma}{Lemma}

\newtheorem{cor}[lemma]{Corollary}
\newtheorem{prop}[theorem]{Proposition}
\numberwithin{equation}{section}

\theoremstyle{definition}

\renewcommand{\geq}{\geqslant}
\renewcommand{\leq}{\leqslant}

\def\exmid{\,\|\,}
\def\bseps{{\boldsymbol\epsilon}}
\def\sym{\mathop{\mathrm{sym}}}

\DeclareMathOperator{\SL}{SL}

\def\sumstar{\mathop{\sum\nolimits^{\ast}}}
\def\ord{\mathop{\textnormal{ord}}\nolimits}

\makeatletter
\DeclareRobustCommand\widecheck[1]{{\mathpalette\@widecheck{#1}}}
\def\@widecheck#1#2{%
    \setbox\z@\hbox{\m@th$#1#2$}%
    \setbox\tw@\hbox{\m@th$#1%
       \widehat{%
          \vrule\@width\z@\@height\ht\z@
          \vrule\@height\z@\@width\wd\z@}$}%
    \dp\tw@-\ht\z@
    \@tempdima\ht\z@ \advance\@tempdima2\ht\tw@ \divide\@tempdima\thr@@
    \setbox\tw@\hbox{%
       \raise\@tempdima\hbox{\scalebox{1}[-1]{\lower\@tempdima\box
\tw@}}}%
    {\ooalign{\box\tw@ \cr \box\z@}}}
\makeatother

\begin{document}

\author{Valentin Blomer}
\author{Djordje Mili\'cevi\'c}
\address{Mathematisches Institut, Bunsenstr.~3-5, D-37073 G\"ottingen, Germany} \email{blomer@uni-math.gwdg.de}
\address{Bryn Mawr College, Department of Mathematics, 101 North Merion Avenue, Bryn Mawr, PA 19010, U.S.A.}\email{dmilicevic@brynmawr.edu}

\title{The second moment of twisted modular $L$-functions}
 
\thanks{The first author acknowledges the support by the Volkswagen Foundation and a Starting Grant of the European Research Council. The second author acknowledges the support by the National Security Agency. Project is sponsored by the NSA under Grant Number H98230-14-1-0139. The United States Government is authorized to reproduce and distribute reprints notwithstanding any copyright notation herein.}


\begin{abstract}  We prove an asymptotic formula with a power saving error term for the (pure or mixed) second moment \[ \underset{\chi \bmod{q}}{\left. \sum \right.^{\ast}} L(1/2, f_1 \otimes \chi) \overline{L(1/2, f_2 \otimes \chi)} \] of central values of $L$-functions of any two (possibly equal) fixed   cusp forms $f_1, f_2$ twisted by all primitive characters modulo $q$,  valid for all sufficiently factorable $q$ including $99{.}9\%$ of all admissible moduli. The two key ingredients are a careful spectral analysis of a potentially highly unbalanced shifted convolution problem in Hecke eigenvalues and power-saving bounds for sums of products of Kloosterman sums where the length of the sum is below the square-root threshold of the modulus. Applications are given to simultaneous non-vanishing and lower bounds on higher moments of twisted $L$-functions.
\end{abstract}


\setcounter{tocdepth}{2}  \maketitle 

\section{Introduction} 

\subsection{The main result}

Most $L$-functions come in families, and often their moments encode some deep properties about the family. The complexity of an $L$-function is measured by its analytic conductor $\mathcal{C}$ (which is typically essentially constant within a family $\mathcal{F}$), and  a measure for the  complexity of a moment calculation is the ratio $r = \log \mathcal{C}/\log |\mathcal{F}|$  (the family may not be discrete in which case an obvious modification is necessary). The edge of current technology where one can hope to obtain an asymptotic formula with a \emph{power saving error term} is $r = 4$. The stock of asymptotic formulas of this kind, however, is very small, and experience has shown that quite often in the case $ r = 4$ the current methods of analytic number theory fail ``by an $\varepsilon$''; if they don't, then typically some very deep input is required.

\bigskip

The most classical example is the fourth moment of the Riemann zeta-function, where one has the asymptotic formula
\begin{equation}\label{zeta}
  \int_0^T |\zeta(1/2 + it)|^4 dt = T P_4(\log T) + {\rm O}(T^{2/3+\varepsilon})
\end{equation}
for a certain polynomial $P_4$ (see \cite{Za, IvMo, Mot}), which is one of the prime applications of the Kuznetsov formula. This formula can be seen as the second moment of the $L$-function attached to a (derivative of an) Eisenstein series, and the corresponding cuspidal analogue, proved by Good \cite{Go}, states that
\begin{equation}\label{good}
  \int_0^T |L(1/2 + it, f)|^2 dt = T P_1(\log T) + {\rm O}(T^{2/3+\varepsilon})
\end{equation}
for a certain polynomial $P_1$ depending on the holomorphic Hecke cusp form $f$. In addition to spectral ana\-ly\-sis of automorphic forms, this result also required an optimal bound for the decay rate of triple~products. 

Other results on  moments with power saving error terms in the case $r=4$ have been established by Kowalski-Michel-VanderKam \cite{KMV1}, Iwaniec-Sarnak \cite{IS}, Blomer \cite{Bl2}, and with a slightly broader interpretation of the notion of a ``moment'' by Li \cite{Li} and  Khan \cite{Kh}.

\bigskip

 
From an adelic point of view, it is natural to 
replace the archimedean twist by $|\det|^{it}$ with a non-archimedean twist by a Dirichlet character $\chi$, and to consider the moments 
 \begin{equation}
 \label{TwistedMoments}
   ({\rm A}) \quad   \sumstar_{\chi \bmod{q}}  |L(1/2,   \chi)|^4 \qquad \text{and} \qquad  ({\rm B}) \quad    \sumstar_{\chi \bmod{q}}  |L(1/2, f \otimes \chi)|^2,
 \end{equation}
 where the sum runs over all primitive Dirichlet characters $\chi$ modulo $q$ and $f$ is a fixed Hecke cusp form in the second sum. Equally interesting and related in spirit are the  moments over quadratic characters only:
  \begin{displaymath}
    ({\rm C}) \quad \underset{\substack{d \leq X\\ d \text{ squarefree}}}{\left.\sum\right.^{\ast}}   |L(1/2,   \chi_{d})|^4 \qquad \text{and} \qquad   ({\rm D}) \quad   \underset{\substack{d \leq X\\ d \text{ squarefree}}}{\left.\sum\right.^{\ast}}  |L(1/2, f \otimes \chi_{d})|^2.
 \end{displaymath}

It was a major breakthrough when M.\ Young \cite{Y} established an asymptotic formula with power saving for (A) for  prime numbers $q$:
\begin{equation}\label{mattyoung}
 \sumstar_{\chi \bmod{q}}  |L(1/2,   \chi)|^4 =  q \sum_{i=1}^4 c_i (\log q)^i+ {\rm O}\left(q^{1-\frac{1}{80} + \frac{\theta}{40} +  \varepsilon}\right),
\end{equation}
where    $c_i$ are effectively computable constants and $\theta \leq 7/64$ is an admissible exponent towards the Ramanujan-Petersson conjecture. 

 The  (harder) cases (B), (C$ $) and (D) have remained unsolved up until now. This is perhaps a bit surprising, but it is important to notice that all 4 moments (A)--(D) single out the point 1/2, and therefore carry some \emph{intrinsic arithmetic information}. This is in contrast to the true adelic analogues of \eqref{zeta} and  \eqref{good} (with a test function expanding in the non-archimedean direction), which are
\begin{equation}\label{analogues}
\int_{-\infty}^{\infty}  \sumstar_{\chi \bmod{q}}  |\Lambda(1/2 + it,   \chi)|^4dt \qquad \text{and} \qquad    \int_{-\infty}^{\infty}  \sumstar_{\chi \bmod{q}}  |\Lambda(1/2+it, f \otimes \chi)|^2dt,
\end{equation}
where $\Lambda$ denotes the completed $L$-function. It is an interesting phenomenon that, comparing \eqref{analogues} to \eqref{TwistedMoments}, an additional average of essentially bounded length  in the $t$-aspect makes the problem incomparably easier, and indeed good asymptotic formulas for both quantities in \eqref{analogues} are fairly routine.

\bigskip

In this paper we couple spectral theory of automorphic forms  with an algebro-arithmetic treatment of short sums of products of Kloosterman sums  to solve the case (B) for 99.9\% of all moduli $q$. Let 
\begin{equation}\label{psi}
  \psi(q) = \sum_{d \mid q} \phi(d)\mu\left(\frac{q}{d}\right)
  \end{equation}
denote the number of primitive characters modulo $q$. It is non-zero if and only if $q \not\equiv 2$ (mod 4), and in this case $\psi(q) = q^{1+o(1)}$. We call a modulus $q \not\equiv 2$ (mod 4) admissible.  

\begin{theorem}\label{mainthm}    For $j= 1, 2$, let $f_j$ be  (fixed)  holomorphic cuspidal newforms  of (even) weight $\kappa_j$   for the group ${\rm SL}_2(\mathbb{Z})$ with Hecke eigenvalues $\lambda_j(n)$, normalized as in \eqref{relation}. Assume that $\kappa_1 \equiv \kappa_2$ $(${\rm mod} $4)$. Let 
\begin{align}
\label{poly1}
  P(s) &= \left(\frac{L_q(s, \text{\rm sym}^2f_1)}{\zeta_q(2s)}\right)^{-1}  = \prod_{p \mid q} 
  \left(1 - \frac{\lambda_1(p^2)}{p^s} + \frac{\lambda_1(p^2)}{p^{2s}} - \frac{1}{p^{3s}}\right) \left(1 - \frac{1}{p^{2s}}\right)^{-1},\\
\label{poly2}
  Q(s) &= \left(\frac{L_q(s, f_1 \times f_2)}{\zeta_q(2s)}\right)^{-1}\\
  &= \prod_{p \mid q} 
  \left(1 - \frac{\lambda_1(p )\lambda_2(p )}{p^s} + \frac{\lambda_1(p^2 ) + \lambda_2(p^2 ) }{p^{2s}} - \frac{\lambda_1(p )\lambda_2(p )}{p^{3s}}  + \frac{1}{p^{4s}}\right) \left(1 - \frac{1}{p^{2s}}\right)^{-1}. \nonumber
\end{align}
Let $q \in \mathbb{N}$, and $q_1 \mid q$ be a divisor such that $(q, 6^{\infty}) \mid q_1$. Then,
    \begin{equation}\label{formula}
   \sumstar_{\chi \bmod{q}}  L(1/2, f_1 \otimes \chi)\overline{L(1/2, f_2 \otimes \chi)} = \frac{2}{\zeta(2)} \psi(q)\cdot M(f_1,f_2,q)+{\rm O}_{f_1, f_2 }\left(q^{1+\varepsilon} \Bigl(q_1^{-\frac{1}{22}} + (q/q_1^2)^{-\frac{1}{22}}\Bigr)\right), 
 \end{equation}
where
\[ M(f_1,f_2,q)=\begin{cases}
\displaystyle P(1) L(1, {\rm sym}^2 f_1) \left( \log q + c +\frac{P'(1)}{P(1)}   \right), &f_1=f_2,\\
Q(1) L(1, f_1 \times f_2), &f_1\neq f_2, \end{cases} \]
and $c$ is a constant depending only on $f_1$ (not on $q$) given explicitly as
 \begin{equation}\label{defc}
  c = \gamma - \frac{1}{2} \log(2\pi) + \frac{\Gamma'(\kappa_1/2)}{\Gamma(\kappa_1/2)} + \frac{L'(1, {\rm sym}^2 f_1)}{L(1, {\rm sym}^2 f_1)} - \frac{2\zeta'(2)}{\zeta(2)}.
\end{equation}   
\end{theorem}
 
Note that $P(1),Q(1)=(\log\log q)^{\text{O}(1)}$ and $P'(1)/P(1)=\text{O}(\log\log q)$, and that the leading coefficients $L(1,\textrm{sym}^2f_1)$ and $L(1,f_1\times f_2)$ do not vanish by the lower bounds of Hoffstein and Lockhart \cite{HL} and Ramakrishnan and Wang~\cite{RW} (see also \cite{Br}), so that the term $M(f_1,f_2,q)$ is not far from a linear polynomial in $\log q$ or a constant depending on $f_1$ and $f_2$.

The error term in Theorem~\ref{mainthm} saves a power of $q$ as soon as $q$ has a divisor $q_1$ in the range
\[ q^{\eta} \ll q_1 \ll q^{1/2 - \eta} \]
for some fixed $\eta > 0$ and if in addition $2^{100} \nmid q$ and $3^{100} \nmid q$ (say) holds. We thus obtain a power saving for  $99.9\%$   of all admissible moduli $q$. In fact, it is not hard to see that these conditions are satisfied for all $q$ except those that are highly divisible by 2 or 3 or are essentially a prime or the product of two primes of almost equal size, that is, those $q$ for which there is a prime $p\geqslant q^{1-\eta}$ with $p\mid q$ or primes $p_1,p_2\geqslant q^{1/2-\eta}$ with $p_1p_2\mid q$. We get the highest savings if $q$ has a divisor of size $q_1\asymp q^{1/3+o(1)}$, for example when $q=p^n$   is a high power of a fixed prime $p > 3$ or when $q$ is essentially a cube, in which case our error term is ${\rm O}(q^{65/66 + \varepsilon})$.

The condition that $(q,6^{\infty})\mid q_1$ is introduced for purely technical and notational reasons; it can be avoided without introducing any new ideas at the cost of increasing the length of the already rather long paper. In Theorem~\ref{mainthm} and all theorems below, the condition that $\kappa_1\equiv\kappa_2$ (mod 4) is necessary in the sense that otherwise the product of the central values vanishes for root number reasons.

\bigskip

Our method works for fixed Maa{\ss} forms $f_1$, $f_2$, assuming that they satisfy the Ramanujan conjecture, which we use crucially in the course of the argument. In Section \ref{mass}, we state the small modifications needed to prove the following result.

\begin{theorem}\label{variant}    For $j= 1, 2$, let $f_j$ be  (fixed)   cuspidal Maa{\ss} newforms  of the same parity   for the group ${\rm SL}_2(\mathbb{Z})$ with   Hecke eigenvalues $\lambda_j(n)$. If $f_1, f_2$ satisfy the Ramanujan conjecture, i.e.\ if $\lambda_j(n) \ll n^{\varepsilon}$ for all $n \in \mathbb{N}$, then \eqref{formula} holds. 
\end{theorem}

The first result in the direction of Theorems \ref{mainthm} and \ref{variant} in the case $f_1=f_2$ is due to Stefanicki \cite{St}, who proved an asymptotic formula for the second moment with an error term that saves a small power of $\log q$, provided $q$ has only few prime divisors. A formula with a $\log\log q$-saving was established by Gao-Khan-Ricotta \cite{GKR} for  almost all integers $q$. As either method saves less than a factor of $\log q$ in the error term, this type of argument cannot produce an asymptotic formula in the case $f_1 \not= f_2$, regardless of the factorization of $q$. An individual asymptotic formula with a power saving error term, and in case $f_1\not= f_2$ an asymptotic formula with \emph{any} saving in the error term, that would be valid for \textit{any} infinite subset of moduli $q$ has been a long-standing open problem until now. Theorems \ref{mainthm}  and \ref{variant} cover, in a weak sense, almost all moduli.

Theorems~\ref{mainthm} and \ref{variant} are concerned with the family of character twists to an individual modulus $q$. If an additional average over moduli $q$ is introduced, the problem becomes easier, and indeed such versions of (B) are available due to Akbary \cite{Ak} and, with a considerably shorter average, to Hoffstein and Lee \cite{HL}.
   
\subsection{Selected applications}   

In addition to providing statistics in families of $L$-functions, asymptotic formulas with a power saving are an essential prerequisite to the analytic techniques of amplification, mollification, and resonators in questions of arithmetic importance, including upper bounds, nonvanishing, and extreme values. The allowable length of the Dirichlet polynomial (such as the amplifier), and thus the quality of arithmetic implications, is related to the strength of the power saving in the summation formula. Several such applications of Theorems~\ref{mainthm} and \ref{variant} are featured here, beginning with the nonvanishing problem.

\bigskip

Combining Theorem~\ref{variant} with a mollifier, one can improve the work of Stefanicki \cite{St} 
to show that (for Maa{\ss} forms satisfying the Ramanujan conjecture) a  \emph{positive proportion} of $L$-functions with twists by primitive Dirichlet characters modulo $q$ does not vanish at the central point, provided that $q$ has a divisor in a suitable range. A non-vanishing result of positive proportion strength had been out of reach so far in this family.

We highlight a different application of Theorem \ref{variant} to \emph{simultaneous} non-vanishing of twisted $L$-functions, as follows:
\begin{theorem}\label{non-vanish} Let $f_1, f_2$ be two  (fixed) cuspidal  Maa{\ss} newforms of the same parity for ${\rm SL}_2(\mathbb{Z})$ that satisfy the Ramanujan conjecture, and let $\eta>0$. Then, for every sufficiently large modulus $q\geqslant C=C(f_1,f_2,\eta)$ such that $q  \not\equiv 2\pmod 4$ and $q$ has a divisor $q_1 \in [q^{\eta}, q^{1/2 - \eta}]$ such that $(q, 6^{\infty}) \mid q_1$, there exist  primitive Dirichlet characters $\chi$ modulo $q$ such that 
\[ L(1/2, f_1 \otimes \chi) \overline{L(1/2, f_2 \otimes \chi)} \not= 0, \]
and, in fact, the number of such characters is at least $q^{1/4 - \varepsilon}$. 
\end{theorem}
  
Nonvanishing results for central values of $L$-functions of character twists have a long history, in particular in connection with cusp forms associated to elliptic curves, but also for general automorphic forms (on fairly general reductive groups). We cannot quote here all the relevant literature, but we would like to emphasize that the focus in Theorem \ref{non-vanish} is on the Maa{\ss} case, because in the holomorphic case one can establish extremely strong non-vanishing results by Galois-theoretic methods \cite{Ro, Ch}. In the Maa{\ss} case, however,   
Theorem \ref{non-vanish} is,  at least under the assumption of the Ramanujan conjecture,  the first instance  of any simultaneous non-vanishing result for general twists of  automorphic $L$-functions.  The quantitative version comes from the best-known subconvexity results for twisted $L$-functions \cite{BH1}.

\bigskip
  
As another application of the asymptotic formula in Theorem \ref{mainthm} --- and here the power saving is absolutely crucial --- one obtains a lower bound of the correct order of magnitude for $k^{\text{th}}$ moments of mixed products
\[ \sumstar_{\chi \bmod{q}}  \Bigl(L(1/2, f_1 \otimes \chi)  \overline{L(1/2, f_2 \otimes \chi)}\Bigr)^{k}, \]
following the method of Rudnick and Soundararajan \cite{RS, RS1}.  As an illustration we provide complete details for the following result.
\begin{theorem}\label{theorem-rudnicksound}
Let $p > 3$ be a fixed prime, and let $q = p^{\kappa}$ be large. Let $f_1, f_2$ be two fixed holomorphic cuspidal Hecke eigenforms of level 1 and respective weights $\kappa_1$, $\kappa_2$ with $\kappa_1 \equiv \kappa_2$ (mod 4). Then
\[  \sumstar_{\chi \bmod{q}}  \Bigl(L(1/2, f_1 \otimes \chi) \overline{L(1/2, f_2 \otimes \chi)}\Bigr)^2 \gg q (\log q)^{2}. \]
\end{theorem}
We remark that with slightly more technical effort one can show by the same method the general lower bound
\[ \sumstar_{\chi \bmod{q}}  \Bigl(L(1/2, f_1 \otimes \chi)  \overline{L(1/2, f_2 \otimes \chi)}\Bigr)^{k} \gg q(\log q)^{k^2/2} \]
for any even integer $k \geq 2$, as well as similar results (up to a factor of $(\log q)^{-\varepsilon}$) for more general $q$, as in Theorem~\ref{mainthm}. Note that $L(1/2, f_1 \otimes \chi) \overline{L(1/2, f_2 \otimes \chi)}$ is real, cf.\ \eqref{afe1} below.  The proof of Theorem \ref{theorem-rudnicksound} will be given at the end of the paper.
  

\subsection{The methods} In this section, we sketch the method of proof of Theorem \ref{mainthm} and highlight some auxiliary results of independent interest, in particular Lemma~\ref{basis} and Theorems  \ref{klooster-short},   \ref{avoid}, and \ref{CentralHBLemma}.  

A natural starting point is an approximate functional equation, and there are two options: one can either take an approximate functional equation for    $L(s, f_1 \otimes \chi)\overline{L(s, f_2 \otimes \chi)}$ with root number independent of $\chi$, or the product of two separate approximate functional equations for $L(s, f_1 \otimes \chi)$ and $L(s, f_2\otimes \chi)$, each of which has a root number depending on $\chi$. Summing over $\chi$, one obtains either way an expression roughly of the shape
\[ \sum_{\substack{nm \leq q^2\\ n \equiv m \bmod{q}}}\lambda_1(m)\lambda_2(n), \]
where $\lambda_j(n)$ denotes the normalized $n$-th Hecke eigenvalue of $f_j$. 
We need to beat the trivial bound ${\rm O}(q^{1+\varepsilon})$ for the contribution of the off-diagonal terms $n \not= m$ by a small, but fixed power of $q$. 
There are two ways to interpret this double sum: either as a shifted convolution problem, or as a problem of summing Hecke eigenvalues in arithmetic progressions. The former point of view is useful if $n$ and $m$ are not too far apart, the latter if one variable is sufficiently small compared to the other variable. For clarity, let us restrict $n \asymp N$ and $m \asymp M$ to dyadic intervals and assume $N \geq M$ by symmetry and (for the sake of argument) $NM = q^2$, which is supposedly the hardest range. 
On the one hand, we can apply Voronoi summation to the inner sum in 
\begin{displaymath}
  \sum_{m \asymp M} \lambda_1(m) \sum_{\substack{n \asymp N\\ n \equiv m \bmod{q}}} \lambda_2(n),
\end{displaymath}
getting roughly
\begin{equation}\label{kloosterm}
 \frac{N}{q^2}  \sum_{m \asymp M}\sum_{ n \asymp q^2/N} \lambda_1(m)  \lambda_2(n) S(n, m, q). 
\end{equation}
The trivial bound at this point, using Deligne (or Rankin-Selberg) and Weil bounds, is $Mq^{1/2}$, which is admissible if $M \leq q^{1/2-\delta}$ (or equivalently $N \geq q^{3/2+\delta}$).

\bigskip

 Alternatively, we can consider the average of shifted convolution problems
\begin{equation}\label{shifted}
 \sum_{r \asymp N/q} \sum_{\substack{n \asymp N, m \asymp M \\ n - m = rq}}  \lambda_1(m) \lambda_2(n).\end{equation}
There is by now a well-developed toolbox of methods for handling shifted convolution problems. The first step is always to detect the linear condition $n-m=rq$ by additive characters. The corresponding (horocycle) integral can then be decomposed by a variant of the circle method. Voronoi summation in the $n, m$-variables leads to sums of Kloosterman sums which can be analyzed spectrally through the Kuznetsov formula. Alternatively (and quite similarly in spirit), one can apply Mellin inversion and the unfolding trick to express the horocycle integral directly as a triple product involving Poincar\'e series which can again be decomposed spectrally. This is the strategy followed by Good \cite{Go} and Sarnak \cite{Sa}. Finally, as a third option, one can use carefully chosen vectors in the representation space of the automorphic representations generated by $f_1$ and $f_2$ to spectrally decompose the horocycle integral directly \cite{BH}. In all approaches, the $n, m$-sum can be spectrally expanded, and the resulting expansion can then be summed over $r$. 

In this paper, we follow \cite{Bl, BHM1} and start with a very flexible variant of the circle method due to Jutila. To speed up the performance, we observe that, although $n\asymp N$, the $n$-sum is in reality relatively short, namely $n = rq + {\rm O}(M)$. One of the main devices in the argument is the well-known trick of attaching a redundant weight function that localizes $n$ at $rq+\mathrm{O}(M)$ (which is of course automatic in  \eqref{shifted}, but gets ``forgotten'' in the course of the manifold transformations unless we remember it explicitly by an additional weight function). The price for this manoeuvre is a very subtle and delicate analysis with Bessel functions, for which we prepare in Section \ref{secbessel}. As a first order approximation, we end up with an expression roughly of the form
\begin{equation}\label{firstorder}
 \frac{M^{3}}{C^3 N^{3/2}} \sum_{t_j \leq (N/M)^{1/2}}  \lambda_j(q) \sum_{r \asymp N/q} \lambda_j( r)  \sum_{m \asymp C^2/M} \sum_{n \asymp C^2N/M^2}  \lambda_1(m) \lambda_2(n) \lambda_j(n-m),
 \end{equation}
where $C = N^{1000}$ is a very large parameter and the outermost spectral sum runs over a basis of level 1 Maa{\ss} forms with spectral parameter $t_j\leq (N/M)^{1/2}$. We caution that \eqref{firstorder} is a much oversimplified expression that reflects reality only in a very vague sense; in particular, some extra cost has to be paid to separate variables, there is also a continuous spectrum contribution, and the level is not always 1, but sometimes a bit larger. Note that the two innermost sums resemble the triple products that would arise in a direct spectral analysis. 

One can now apply the Cauchy-Schwarz inequality and the spectral large sieve of Deshouillers-Iwaniec, thus obtaining the final bound $ Nq^{\theta -1/2}$ plus some more terms that are smaller in typical ranges. (Here, as usual, $\theta$ denotes an admissible exponent towards the Ramanujan-Petersson conjecture.) We point out the interesting feature of Jutila's method that the auxiliary parameter $C$ is only a catalyst that does not enter the final bound and conclude that this analysis is admissible if $N \leq q^{3/2 - \theta - \delta}$.

\bigskip

Obviously, the ranges $N\geq q^{3/2 + \delta}$ and $N \leq q^{3/2 - \theta - \delta}$ do not overlap, not even assuming the Ramanujan conjecture ($\theta = 0$). 
The overall strategy up to this point is  similar to the analysis in \cite{Y}, and both here and there the main problem is to overcome the small gap in the two ranges. Young uses the fact that one can decompose the divisor function in order to get more variables with which one can apply Poisson summation. The corresponding saving is strong enough to close the gap. This route is not available in the present situation.

\bigskip

As the first step, we remove the dependence on the Ramanujan conjecture by applying H\"older's inequality to \eqref{firstorder} with exponents $1/4$, $1/4$, $1/2$, getting
\begin{displaymath}
\begin{split}
 \frac{M^{3}}{C^3 N^{3/2}}&\Bigl( \sum_{t_j \leq (\frac{N}{M})^{1/2}}  |\lambda_j(q)|^4\Bigr)^{\frac{1}{4}}\Bigl( \sum_{t_j \leq (\frac{N}{M})^{1/2}} \Bigl| \sum_{r \asymp \frac{N}{q}} \lambda_j( r)\Bigr|^4 \Bigr)^{\frac{1}{4}}   \Bigl( \sum_{t_j \leq (\frac{N}{M})^{1/2}} \Bigl| \sum_{h \asymp \frac{C^2N}{M^2}} \lambda_j(h) \underset{n-m=h}{\sum_{ m \asymp \frac{C^2}{M}} \sum_{n \asymp \frac{C^2N}{M^2}}}  \lambda_1(m) \lambda_2(n) \Bigr|^2\Bigr)^{\frac{1}{2}}. 
\end{split} 
 \end{displaymath}
After expanding one of the squares inside the fourth powers using multiplicativity,  we apply the Kuznetsov formula for the first factor, and the large sieve (which is, of course, also based on the Kuznetsov formula) for the other two factors, getting a  bound roughly of the strength  $Nq^{-1/2}$ without dependence on the Ramanujan conjecture. A precise version can be found in Proposition \ref{prop3} below. The crucial input is Theorem \ref{avoid}, which presents a flexible variant of the spectral large sieve that allows for additional divisibility conditions (that are, in turn, essential to the success of our method) without being wasteful. This requires, among other things, an orthonormalization of the collection of 
 Maa{\ss} forms $\{f(dz) : d \mid \ell\}$ for a newform $f$ and some integer $\ell$, where $\ell$ is not necessarily squarefree; see Lemma  
 \ref{basis}. 
 
\bigskip

This procedure works in great generality. As observed by Fouvry, Kowalski and Michel,  the methods we employ improve Young's result \eqref{mattyoung} on the fourth moment of Dirichlet $L$-functions:
\[  \sumstar_{\chi \bmod{q}}  |L(1/2,   \chi)|^4 =  q \sum_{i=1}^4 c_i (\log q)^i+ {\rm O}\left(q^{1-\frac{1}{82} +   \varepsilon}\right) \]
for primes $p$.   At the current state of knowledge, this is better than \eqref{mattyoung} and -- more importantly -- independent of bounds towards the Ramanujan-Petersson conjecture. Inserting more algebraic geometry, the error term in \eqref{mattyoung} can in fact be improved to ${\rm O}(q^{1-1/32 + \varepsilon})$ with no recourse on bound  towards the Ramanujan-Petersson conjecture. This  result   is contained, among other things, in the companion paper \cite{BFKMM}, which imports the spectral analysis discussed in this subsection.

\bigskip

Returning to the situation of Theorem \ref{mainthm}, it now remains to close the ``small" gap where $N = q^{3/2+o(1)}$ and $M = q^{1/2+o(1)}$, for which an essentially new idea is necessary. We  use  the Cauchy-Schwarz inequality to bound \eqref{kloosterm}  by
\begin{equation}\label{expressionby}
  \frac{NM^{1/2}}{q^2}\Bigl( \sum_{n_1, n_2 \asymp q^2/N} \Bigl| \sum_{m \asymp M} S(m, n_1, q) S(m, n_2, q)\Bigr|\Bigr)^{1/2}.
\end{equation}
Weil's individual bound for Kloosterman sums yields an upper bound of $Mq^{1/2}$, and we win if we can prove some extra cancellation for generic pairs $(n_1, n_2)$ in the short $m$-sum. Note that at this point all automorphic information is gone, and we are left with a problem of bounding exponential sums, namely short sums of products of two Kloosterman sums.  Generically, the length of the $m$-sum is roughly the square-root of the modulus of the two Kloosterman sums, so this seems to be a hard problem in general. 

\subsection{Short sums of products of Kloosterman sums} The crucial new arithmetic input of this paper is a non-trivial estimation of the inner double sum in \eqref{expressionby} if $q$ is sufficiently factorable. In fact, we can estimate the individual $m$-sums  with pleasing success generically and only use the sum over $n_1$, $n_2$ to control the frequency of ``nearly diagonal'' pairs $(n_1,n_2)$.  
 Our analysis is somewhat inspired by Heath-Brown's paper on hybrid bounds for Dirichlet $L$-functions \cite{HB}. Our situation is more involved, since the function $b(m) = S(m, n_1, q) S(m, n_2, q)$ is not multiplicative in $m$ (not even in some twisted sense), unlike a Dirichlet character $\chi(m)$ modulo $q$. Moreover, for higher prime powers $q=p^s$, the Kloosterman sum resembles the exponential of a $p$-adic square-root, and therefore much more genuine $p$-adic methods naturally enter the analysis of the corresponding multiple exponential sums.
 
Nevertheless, provided that we can factorize $q = r_1r_2$ with $(r_1,r_2) = 1$, a careful application of Weyl differencing with respect to $r_2$ (presented in Lemma~\ref{DifferencingLemma}), followed by an application of Poisson summation to effect the technique of ``completion'', yields   a bound roughly of the form
\begin{displaymath}
 \Bigl|\sum_{m \asymp M} S(m, n_1, q) S(m, n_2, q)\Bigr|^2 \ll M^2q^2\Bigl(\frac{r_2}{M} +\frac{r_2^2}{M^2} +  \frac{\widehat{S}}{r_1^2M} +  \frac{\widehat{S}}{r_1^3}\Bigr),
\end{displaymath}
where $\widehat{S}$ is the average of \emph{complete} sums of the type 
\begin{equation}
\label{CompleteSumToBeEstimated}
  \sum_{m \bmod{r_1}} S(m, n_1, r_1)  S(m, n_2, r_1) S(m+h, n_1, r_1) S(m+h, n_2, r_1) e\left(\frac{km}{r_1}\right)
  \end{equation}
for various values of $k$ and $h$. A general version of the underlying idea is presented in Theorem~\ref{CentralHBLemma} in Section~\ref{sec7}, which may be of use in other situations.

If $r_1$ is squarefree, one can use the independence of Kloosterman sheafs \cite{Ka} to obtain square-root cancelation (in generic situations) in the multiple exponential sum \eqref{CompleteSumToBeEstimated}. For the squareful parts, we obtain a bound of generically similar strength via an unexpectedly involved $p$-adic stationary phase argument that features, among other things, singular critical points; the latter are necessary to obtain results for the class of moduli of the stated generality and (as will be evident from our treatment) provably contribute to the correct order of magnitude. It is common belief that exponential sums to squareful moduli are easy to handle; while it is true that their treatment is \emph{elementary} (in the sense that in most ranges no algebraic geometry is needed), the analysis is often extremely complicated, and the treatment of degenerate cases can turn out to be quite involved (see \cite{DF} for an example of ${\rm GL(3)}$ Kloosterman sums). The upshot of the above discussion is the following result:

\begin{theorem}\label{klooster-short} Let $r, q, n_1, n_2 \in \mathbb{N}$ with $r \mid q$, let $A \in \mathbb{R}$, $M > 1$. Then, for any $s\mid r$ satisfying $(r, 6^{\infty})\mid s$ we have
\[ \sum_{\substack{A < m \leq A+M\\ (m, q) = 1}} S(m, n_1, r) S(m, n_2, r) \ll
r^{\varepsilon}\Bigl(M^{1/2}r s^{1/2}+\frac{M^{1/2}r^{5/4 }}{s^{1/4}}+Mr^{3/4 }(r,n_1-n_2)^{1/4}s^{1/4}+r s+\sigma\Bigr), \]
where the term $\sigma$ defined in \eqref{Definitionsigma} satisfies $\sigma=0$ if $r/(r,s^{\infty})$ is cube-free, and $\sigma\ll r^{11/8 }s^{1/8}$ in all cases.
\end{theorem}

As in Theorem \ref{mainthm}, with a bit more work the condition $(r, 6^{\infty}) \mid s$ could be removed in Theorem~\ref{klooster-short}; it affects only moduli $r$ divisible by extremely high powers of 2 or 3. 

Comparing with the ``trivial'' bound $Mr^{1+o(1)}$ on the left-hand side, and assuming for simplicity that $(r,n_1n_2(n_1-n_2))=1$, we obtain a power saving as long as
\[ \frac{r}{M^2}(rM)^{\eta}\ll s\ll\min\left(M,\frac{M^8}{r^3}\right)(rM)^{-\eta}, \]
where the   term involving $M^8/r^3$ can simply be omitted (and the ranges of application in $s$ extended) if $r/(r,s^{\infty})$ is cube-free. In the important range $M\asymp r^{1/2}$, this gives a power saving as long as $r$ has a divisor $s$ in the (essentially full) range
\[ r^{\eta}\ll s\ll r^{1/2-\eta} \]
(with the above constraint on high powers of 2 and 3). This holds for $99{.}9\%$ of all $r$.

\bigskip

As an application, let us consider the most interesting range, the ``square-root threshold''  $M\asymp r^{1/2}$. If  $(r,n_1n_2(n_1-n_2))=1$ and $r$ has a divisor in the range $s\asymp r^{1/3}$, we obtain the bound $r^{17/12+\varepsilon}$, an improvement of $r^{1/12}$ over the ``trivial'' bound $r^{3/2}$. For a sum such as that featured in Theorem~\ref{klooster-short}, with $\asymp r^{1/2}$ terms of arithmetic nature to modulus $r$ of size $\asymp r^{1+o(1)}$, it may be reasonable to speculate that the best possible bound (and the true order of magnitude) is $\asymp r^{5/4+o(1)}$. Our bound thus reaches $\frac13$ of the way from the trivial to the best possible result and may be seen as the analogue of the ``Weyl exponent'' in this case.

In the case when $r=p^s$ is a sufficiently high prime power and $(r,n_1n_2(n_1-n_2))=1$, Theorem~\ref{klooster-short} is concerned with a short sum of exponentials with a $p$-adically analytic phase that may be directly estimated by \cite[Theorem~2]{Mi}. In fact, in the crucial range $M\asymp r^{1/2}$ this yields a bound of sub-Weyl strength $r^{17/12-\delta}$ in the situation of Theorem~\ref{klooster-short} and consequently a stronger error term of the form $\text{O}(q^{65/66-\delta'})$ in Theorem~\ref{mainthm} in the case of a prime power modulus $q$, with some small but fixed $\delta,\delta'>0$. The corresponding route does not appear to be as readily available for more general $r$ (not even at high prime power divisors of $r$), since, absent additional arithmetic conditions on the divisor $s$, degenerate critical points genuinely must be considered.

\bigskip

Using Theorem \ref{klooster-short} in \eqref{expressionby}, we obtain Proposition \ref{bound1} below, which enables us to complete the proof of Theorem \ref{mainthm}. We finally remark that the pleasing generality of the moduli considered in this paper requires a lot of technical overhead (in both the automorphic and the algebro-arithmetic treatment) that contributes to the length of the paper. 

\bigskip

\textbf{Acknowledgements.} We would like to take the opportunity to thank \'Etienne Fouvry, Emmanuel Kowalski, Philippe Michel, Lillian Pierce and Guillaume Ricotta for helpful remarks and discussions. This paper grew out of the conversations we had while the second author visited the Max Planck Institute for Mathematics in Bonn; it is a pleasure to acknowledge the support and excellent research infrastructure at MPIM.

\section{Automorphic Preliminaries I}\label{Hecketheory}



We follow the notation of \cite{BHM1}. We write the Fourier expansion of a holomorphic modular form $f$ of level $\ell$ and weight $k$ as
\begin{displaymath}
  f(z) = \sum_{n \geq 1} \rho_f(n) (4\pi n)^{k/2} e(nz), 
\end{displaymath}
and similarly we write for a Maa{\ss} form $f$ of level $\ell$ and spectral parameter $t = t_f \in  \mathbb{R} \cup [-i \theta, i\theta]$ (where currently $\theta = 7/64$ is known)
\begin{equation}\label{four}
  f(z) = \sum_{n \not= 0} \rho_f(n) W_{0, it}(4\pi |n|y) e(nx)
\end{equation}
where $W_{0, it}(y) = (y/\pi)^{1/2} K_{it}(y/2)$ is a Whittaker function. The inner product of two Maa{\ss} forms $f$ and $g$ of level $\ell$ is given by
\begin{equation}\label{innerprod}
\langle f, g\rangle := \int_{\Gamma_0(\ell)\backslash \mathbb{H}} f(z) \overline{g(z)} \frac{dx\, dy}{y^2}. 
\end{equation}
 For each cusp $\mathfrak{a}$ of $\Gamma_0(\ell)$ there is an Eisenstein $E_{\mathfrak{a}}(z, s)$   series whose Fourier expansion at $s = 1/2 + it$ we write as 
\begin{displaymath}
  E_{\mathfrak{a}}(z, 1/2 + it) = \delta_{\mathfrak{a} = \infty} y^{1/2+it} + \varphi_{\mathfrak{a}}(1/2 + it) y^{1/2-it} + \sum_{n \not= 0} \rho_{\mathfrak{a}}(n, t) W_{0, it}(4\pi |n|y) e(nx). 
\end{displaymath}
\bigskip

If $f$ is a cuspidal newform (and in particular an eigenform of all Hecke operators), we denote its  normalized Hecke eigenvalues by $\lambda_f(n)$ and record the relation
\begin{equation}\label{relation}
\lambda_f(n) \rho_f(1) = \sqrt{n} \rho_f(n)
\end{equation}
for $n \geq 1$, and $\rho_f(-n) = \pm \rho_f(n)$ in the Maa{\ss} case (since $f$ is an eigenform of the involution $z \mapsto -\bar{z}$).  For future reference we  state the well-known bounds (e.g.\ \cite[(30)]{HM})
\begin{equation}\label{rho1}
 | \rho_{f}(1) |^2 = \frac{\cosh(\pi t_f)}{\ell} (\ell (1+|t_f|))^{o(1)}
\end{equation}
for a \emph{newform} $f$ of level $\ell$ which are essentially due to Hoffstein-Lockhart (upper bound) and Iwaniec (lower bound). We will frequently use the Hecke relation
\begin{equation}\label{hecke}
  \lambda_f(nm) = \sum_{d \mid (n, m)} \mu(d)\chi_0(d) \lambda_f\left(\frac{n}{d}\right)\lambda_f\left(\frac{m}{d}\right), \quad n, m \in \mathbb{N},  
\end{equation}
where $\chi_0$ is the trivial character modulo $\ell$, and the Rankin-Selberg bound
\begin{equation}\label{RS}
  \sum_{n \leq x} |\lambda_f(n)|^2 \ll_f x. 
\end{equation}
If $f$ is in addition holomorphic, then we have Deligne's bound \cite{De}
\begin{equation}\label{deligne}
  \lambda_f(n) \ll n^{\varepsilon}. 
\end{equation}
This is expected to hold for Maa{\ss} newforms (of arbitrary level) as well, but in general we only know
\begin{equation}\label{KS}
    \lambda_f(n) \ll n^{\theta +\varepsilon}, 
\end{equation}
where $\theta$ is an admissible exponent for the Ramanujan-Petersson conjecture. Currently $\theta = 7/64$ is known \cite{KS}. Wilton's bound gives
\begin{equation}\label{wilton}
  \sum_{n \leq x} \lambda_f(n) e(\alpha n) \ll_f x^{1/2+\varepsilon},
\end{equation}
uniformly in $\alpha \in \mathbb{R}$.

\bigskip

For a smooth, compactly supported function $V : (0, \infty) \rightarrow \mathbb{C}$ and fixed $\kappa \in \mathbb{N}$ define the Hankel-type  transform
\begin{equation}\label{hankel}
  \mathring{V}(y) = 2\pi i^{\kappa} \int_{0}^{\infty} V(x) J_{\kappa-1}(4\pi \sqrt{xy}) dx.
\end{equation}
 It depends on $\kappa$, but this is not displayed in the notation.  It is easy to see that $\mathring{V}$ is a Schwartz class function; indeed, by \cite[Section 2.6]{BM} we have
\begin{equation}\label{byparts}
 \int_{0}^{\infty} V(x) J_{\kappa-1}(4\pi \sqrt{xy}) dx = \left(-\frac{1}{2\pi \sqrt{y}}\right)^j \int_0^{\infty} \frac{\partial^j}{\partial x^j} \left(V(x) x^{- \frac{\kappa-1}{2}}\right) x^{\frac{\kappa-1+j}{2}} J_{\kappa-1+j}(4\pi \sqrt{xy})dx
 \end{equation}
for any $j \in \mathbb{N}_0$, and now one can differentiate under the integral sign using \cite[8.471.2]{GR}. 

The Mellin transform of a function $f$ will always be denoted by $\widehat{f}$. More integral transforms will be introduced in the context of the Kuznetsov formula. The following formula is standard (e.g.\ \cite[Proposition 1]{HM}). 

\begin{lemma}\label{vor} [Voronoi summation] Let $c \in \mathbb{N}$, $b \in \mathbb{Z}$, and assume $(b, c) = 1$. Let $V$ be a  smooth compactly supported function, and let $N > 0$. Let $\lambda(n)$ denote the normalized Hecke eigenvalues of a holomorphic cuspidal newform of weight $\kappa$ for $\SL_2(\mathbb{Z})$.  Then
\begin{displaymath}
  \sum_n \lambda(n) e\left(\frac{bn}{c}\right) V\left(\frac{n}{N}\right) = \frac{N}{c}  \sum_n \lambda(n) e\left(-\frac{\bar{b}n}{c}\right) \mathring{V}\left(\frac{n}{c^2/N}\right). 
\end{displaymath} 
\end{lemma}

\section{The core argument}
 
 \subsection{The main term}
In this section, we present the backbone of the   proof of the Theorem \ref{mainthm}. By a standard approximate functional equation (\cite[Theorem 5.3]{IK}) we have for  each primitive character $\chi$ modulo $q$ that
\begin{equation}\label{afe1}
  L(1/2, f_1 \otimes \chi)  \overline{L(1/2, f_2 \otimes \chi)}   =  \sum_{n, m} \frac{(\lambda_1(m)\lambda_2(n) +\lambda_2(m)\lambda_1(n))  \chi(m) \bar{\chi}(n)}{(nm)^{1/2}} W\left(\frac{nm}{q^2}\right)
 \end{equation}
where
\begin{equation}\label{afe}
  W(x) = \frac{1}{2\pi i} \int_{(2)} \frac{\Gamma(\kappa_1/2 + s)\Gamma(\kappa_2/2 + s)}{(2\pi)^{2s}\Gamma(\kappa_1/2)\Gamma(\kappa_2/2)} x^{-s} \frac{ds}{s}
\end{equation} 
 satisfies 
  $W^{(j)}(x) \ll_{A, j} (1+x)^{-A}$ for all $A, j \geq 0$.  Note that by \cite[Proposition 14.20]{IK}  the $L$-function $L(s, f_1 \otimes \chi)  \overline{L(s, f_2 \otimes \chi)} $ has root number 1 if $\kappa_1 \equiv \kappa_2$ (mod 4). Summing over all primitive characters $\chi$ and using the elementary identity
 \begin{displaymath}
   \underset{\chi \bmod{q}}{\left. \sum \right.^{\ast}} \chi(n) = \sum_{d \mid (n-1, q)} \phi(d) \mu(q/d), 
 \end{displaymath}
for $(n, q) = 1$, we obtain
\begin{equation}\label{start}
   \underset{\chi \bmod{q})}{\left. \sum \right.^{\ast}}  L(1/2, f_1 \otimes \chi)  \overline{L(1/2, f_2 \otimes \chi)}   = 2 \sum_{d \mid q} \phi(d) \mu\left(\frac{q}{d}\right)\sum_{\substack{n \equiv m \bmod{d}\\ (nm, q) = 1}} \frac{\lambda_1(m)\lambda_2(n)  }{(nm)^{1/2}} W\left(\frac{nm}{q^2}\right). 
\end{equation}

The diagonal term $n=m$ contributes
\begin{displaymath}
 \Delta(q) =  2 \psi(q) \sum_{(n, q) = 1} \frac{\lambda_1(n)\lambda_2(n)}{n} W\left(\frac{n^2}{q^2}\right) = 
   \frac{2 \psi(q) }{2\pi i} \int_{(1)} \frac{L^{(q)}(1 + 2s, f_1 \times f_2)}{\zeta^{(q)}(2(1+2s))} q^{2s} \widehat{W}(s) ds
\end{displaymath}
where the superscript $(q)$ denotes omission of the Euler factors at primes dividing $q$. Define 
   $P$, $Q$ and $c$ as in \eqref{poly1} --  \eqref{defc} so that in particular
   \[ \frac{L^{(q)}(1+2s, f_1 \times f_2) }{\zeta^{(q)}(2(1+2s))}=  \frac{L(1+2s, f_1 \times f_2)}{\zeta(2(1+2s))} Q(s). \]  
Shifting  the contour to $\Re s= -1/4 + \varepsilon$, we obtain 
 \begin{displaymath}
 \Delta(q) =    2 \psi(q) \frac{P(1) L(1, {\rm sym}^2 f_1)}{\zeta(2)}\left( \log q + c +\frac{P'(1)}{P(1)}  + O\left(q^{-\frac{1}{2}+\varepsilon}\right)  \right), \quad f_1 = f_2, 
   \end{displaymath}
and
\begin{displaymath}
 \Delta(q) =    2 \psi(q) \frac{Q(1) L(1, f_1 \times f_2)}{\zeta(2)}\left(1 + O\left(q^{-\frac{1}{2}+\varepsilon}\right)  \right), \quad f_1 \not= f_2. 
   \end{displaymath}

\subsection{The off-diagonal term}\label{off}
We proceed to treat the off-diagonal contribution $n \not = m$ in \eqref{start}. We attach a smooth partition of unity to the $n$- and $m$-sum, and localize the variables at $N \leq n \leq 2N$ and $M \leq m \leq 2M$ with weight functions $v_1, v_2$, where $N, M \geq 1$ and $NM \leq q^{2+\varepsilon}$ (at the cost of a negligible error). By Mellin inversion we are left with bounding 
\[ \sum_{d \mid q} d \Bigl|\int_{(\varepsilon)} \widehat{W}(s) \sum_{\substack{n \equiv m \bmod{d}\\ (nm, q) = 1\\n \not= m}} \frac{\lambda_1(m)\lambda_2(n)  }{(nm)^{1/2}} v_1\left(\frac{n}{N}\right) v_2\left(\frac{m}{M}\right)  \left(\frac{nm}{q^2}\right)^{-s} \frac{ds}{2\pi i} \Bigr|. \]
 By Stirling's formula, $\widehat{W}$ is exponentially decreasing on vertical lines, so that we can truncate the integral at $|\Im s| \leq (\log 5q)^2$ at a negligible cost.  It therefore suffices to bound
\begin{equation}\label{snmd}
  S_{N, M, d, q} := \frac{d}{(NM)^{1/2}} \sum_{\substack{n \equiv m \bmod{d}\\ (nm, q) = 1\\ n \not= m}}  \lambda_1(m)\lambda_2(n)  V_1\left(\frac{m}{M}\right) V_2\left(\frac{n}{N}\right). 
  \end{equation}
for $d \mid q$ and $N \geq M$ (by symmetry) for  functions $V_{1, 2}$ with compact support in $[1, 2]$ and derivatives bounded by 
\begin{equation}\label{derivative}
  V_{1, 2}^{(j)}(x) \ll (\log 5q)^{2j}\ll_j q^{\varepsilon} .
\end{equation}  
   Using Deligne's bound\footnote{This is the only point in the argument where Deligne's bound seems unavoidable.} \eqref{deligne}, we obtain immediately a trivial bound
\begin{equation}\label{verytrivial}
  S_{N, M, d, q} \ll  \frac{d}{(NM)^{1/2-\varepsilon}} \sum_{M \leq m \leq 2M} \sum_{\substack{N \leq n \leq 2N\\ n \equiv m \bmod{d}\\ n \not= m}} 1 \ll  (NM)^{1/2+\varepsilon} . 
\end{equation}

In the next section we will show
\begin{prop}\label{bound1} Let $q_1 \mid q$ be a divisior satisfying $(q, 6^{\infty}) \mid q_1$. Then
 \begin{displaymath}
    S_{N, M, d, q} \ll    q^{\varepsilon}  \frac{q}{N^{1/2}} \Biggl(M^{1/4} q^{1/2} q_1^{1/4} + \frac{M^{1/4} q^{5/8}}{q_1^{1/8}} + M^{1/2} q^{3/8} q_1^{1/8} + (qq_1)^{1/2} + q^{11/16} q_1^{1/16} \Biggr)
     \end{displaymath}
    for any $d \mid q$ whenever $N \geq 20M$. 
 \end{prop}   
    
To see when this result will be useful for us, we assume that $NM = q^2$. If $q_1 \leq q^{1/2}$, Proposition \ref{bound1} covers the range $N \geq q^{3/2 + \delta}$, for any fixed $\delta > 0$. (In fact, the trivial estimate on the upper bound on $S_{N,M,d,q}$ reached by \eqref{separatedelta} suffices in this range of parameters, and this requires no special divisibility properties of $q$.) However, if we can find $q_1$ such that $q^{\eta} \leq q_1 \leq q^{1/2 - \eta}$, then we can extend the range for $N$ slightly beyond $q^{3/2}$, so that it will overlap with the admissible range in Proposition \ref{prop3} below.

\bigskip

Now let $\ell_1, \ell_2 \in \mathbb{N}$, $h \in \mathbb{N}$, and define 
 \begin{equation}\label{defD}
   \mathcal{D}(\ell_1, \ell_2, h, N , M ) = \sum_{\ell_1 n - \ell_2 m = h} \lambda_1(m)\lambda_2(n) V_1\left(\frac{\ell_2 m}{M}\right) V_2\left(\frac{\ell_1 n}{N}\right)
\end{equation}
and
\begin{equation}\label{average}
  \mathcal{S}(\ell_1, \ell_2, d, N, M) = \sum_{r } \mathcal{D}(\ell_1, \ell_2, rd, N, M)
\end{equation}
where $d$ is a positive integer. Note that the support of $V_2$ restricts $r \leq 2N/d$. 
From \cite[Theorem 3]{Bl} we quote the individual uniform bound
\begin{equation}\label{indivbound}
   \mathcal{D}(\ell_1, \ell_2, h, N , M ) \ll (N+M)^{1/2 + \theta}(NMq)^{\varepsilon}. 
\end{equation}
From \eqref{snmd} we obtain by M\"obius inversion, \eqref{hecke} and \eqref{deligne}  that
\begin{equation}\label{s2}
\begin{split}
 & S_{N, M, d, q}  =  \frac{d}{(NM)^{1/2}}\sum_{r = 1}^{2N/d} \sum_{\substack{n-m = rd\\ (nm, q) = 1}} \lambda_1(m) \lambda_2(n) V_1\left(\frac{m}{M}\right)V_2\left(\frac{n}{N}\right)\\
  & = \frac{d}{(NM)^{1/2}}\sum_{g_1\mid f_1 \mid q}\mu(g_1)\mu(f_1)\lambda_2\left(\frac{f_1}{g_1}\right) \sum_{g_2 \mid f_2 \mid q}\mu(g_2)\mu(f_2)\lambda_1\left(\frac{f_2}{g_2}\right)  \sum_{r = 1}^{2N/d} \mathcal{D}(f_1g_1, f_2g_2, rd, N , M )\\
 &\ll \frac{d }{(NM)^{1/2}} \sum_{\substack{g_1\mid f_1 \mid q\\ g_2\mid f_2 \mid q}}   (f_1f_2)^{\varepsilon}\Bigl|   \mathcal{S}(f_1g_1, f_2g_2, d, N, M)\Bigr|.
 \end{split}
\end{equation}
Sections \ref{SpectralDecompositionSection} and \ref{ShiftedSumsAverageSection} are devoted to the proof of 
 \begin{prop}\label{prop3} Let $\ell_1, \ell_2, d \in \mathbb{N}$, $N, M \geq 1$ and define $\mathcal{S}(\ell_1, \ell_2, d, N, M)$ as in \eqref{defD} -- \eqref{average}. Assume that $N \geq 20M$. 
 Then
\begin{displaymath}
\begin{split}
 & \mathcal{S}(\ell_1, \ell_2, d, N, M) \ll  (dN)^{\varepsilon}   \left(\frac{N}{d^{1/2}} + \frac{N^{5/4}M^{1/4} }{d } + \frac{N^{3/4}M^{1/4} }{d^{1/4}} + \frac{NM^{1/2} }{d^{3/4} }\right). 
\end{split}
\end{displaymath}
The implicit constant depends on $\varepsilon$ alone. 
\end{prop}

This implies
\begin{equation}\label{shiftedbound}
  S_{N, M, d, q}  \ll \left( \frac{(Nq)^{1/2}}{M^{1/2}} + \frac{N^{3/4}}{M^{1/4}} + \frac{N^{1/4}q^{3/4}}{M^{1/4}} + N^{1/2} q^{1/4}    \right)(qN)^{\varepsilon}
\end{equation}
for $N \geq 20M$, while from \eqref{indivbound} and \eqref{s2} we conclude by trivial estimates  
\begin{equation}\label{auxbound}
 S_{N, M, d, q} \ll \frac{d}{(NM)^{1/2}} q^{ \varepsilon}  \frac{N}{d} N^{1/2 +\theta} =   \frac{q^{\varepsilon} N^{1+\theta} }{M^{1/2}}
\end{equation}
in the slightly larger range $N \geq M$.


\subsection{An optimization problem}\label{opti}

 We are now prepared to prove Theorem \ref{mainthm}.  First we observe that \eqref{auxbound} in connection with our general assumption $NM \leq q^{2+\varepsilon}$ suffices to prove Theorem \ref{mainthm} whenever $N \asymp 20M$. 
 Hence from now on we assume 
 $N \geq M$ so that Proposition \ref{bound1} and \eqref{shiftedbound} are available. 
 In preparation for later estimates, we observe that \eqref{shiftedbound} implies
\begin{equation}\label{easier}
  S_{N, M, d, q} \ll \frac{q^{3/4 + \varepsilon} N^{1/4}}{M^{1/4}}, \quad \text{if} \quad NM \leq q^{2+\varepsilon}, N \leq Mq.
\end{equation} 
  We distinguish two cases.
 
 \emph{Case I: $q_1 \leq q^{1/3}$.} In this case we need to show $S_{N, M, d, q} \ll q^{1+\varepsilon} q_1^{-1/22}$.   The bound \eqref{verytrivial} is admissible unless 
\begin{equation}\label{firstcond}
   q^{2 }q_1^{-1/11} \leq NM \leq q^{2+\varepsilon}. 
\end{equation}
In this range, \eqref{easier} is admissible unless
 \begin{equation}\label{seccond}
   N/M \geq q q_1^{-2/11}.
 \end{equation}
 If both \eqref{firstcond} and \eqref{seccond} hold, then Proposition \ref{bound1} implies that $S_{N, M, d, q}$ is, up to a factor $q^{\varepsilon}$, at most
\begin{displaymath}
\begin{split}
& \frac{q^{3/2}q_1^{1/4}}{(N/M)^{3/8}(NM)^{1/8}} + \frac{q^{13/8}q_1^{-1/8}}{ (N/M)^{3/8}(NM)^{1/8}} + \frac{q^{11/8} q_1^{1/8}}{(N/M)^{1/2}} + \frac{q^{3/2}q_1^{1/2}}{(N/M)^{1/4} (NM)^{1/4}} + \frac{q^{27/16}q_1^{1/16}}{(N/M)^{1/4} (NM)^{1/4}}  \\
 \ll &  q^{7/8} q_1^{29/88} + q q_1^{-1/22} + q^{7/8}q_1^{19/88} + q^{3/4} q_1^{25/44} + q^{15/16}q_1^{23/176} \ll  q q_1^{-1/22} 
\end{split}
\end{displaymath} 
for $q_1 \leq q^{1/3}$.   

 \emph{Case II: $q^{1/3} \leq q_1 \leq q^{1/2}$.} In this case we need to show $S_{N, M, d, q} \ll q^{21/22+\varepsilon} q_1^{1/11}$.   The bound \eqref{verytrivial} is admissible unless 
\begin{equation}\label{firstconda}
   q^{21/11 }q_1^{2/11} \leq NM \leq q^{2+\varepsilon}. 
\end{equation}
In this range, \eqref{easier} is admissible unless
 \begin{equation}\label{secconda}
   N/M \geq q^{9/11} q_1^{4/11}.
 \end{equation}
 If both \eqref{firstconda} and \eqref{secconda} hold, then Proposition \ref{bound1} implies that $S_{N, M, d, q}$ is, up to a factor $q^{\varepsilon}$, at most
 \begin{displaymath}
\begin{split}
& \frac{q^{3/2}q_1^{1/4}}{(N/M)^{3/8}(NM)^{1/8}} + \frac{q^{13/8}q_1^{-1/8}}{ (N/M)^{3/8}(NM)^{1/8}} + \frac{q^{11/8} q_1^{1/8}}{(N/M)^{1/2}} + \frac{q^{3/2}q_1^{1/2}}{(N/M)^{1/4} (NM)^{1/4}} + \frac{q^{27/16}q_1^{1/16}}{(N/M)^{1/4} (NM)^{1/4}}  \\
 \ll &  q^{21/22} q_1^{1/11}+ q^{95/88}q_1^{-25/88} + q^{85/88}q_1^{-5/88} + q^{9/11} q_1^{4/11}+ q^{177/176} q_1^{-13/176} \ll q^{21/22+\varepsilon} q_1^{1/11}
\end{split}
\end{displaymath} 
for $q^{1/3} \leq q_1 \leq q^{1/2}$. 


\section{Hecke eigenvalues in residue classes}
\label{HeckeResidueClasses}

In this section we prove Proposition \ref{bound1}, assuming the validity of Theorem \ref{klooster-short} whose proof we postpone to the end of the paper. The method presented here is strong if $N$ is much larger than $M$. Initially we only assume $N \geq 20M$, so that the condition $n \not= m$ is moot. We write
\begin{displaymath}
  S_{N, M, d, q}  = \frac{d}{(NM)^{1/2}} \sum_{(m, q) = 1} \lambda_1(m) V_1\left(\frac{m}{M}\right) \sum_{\substack{n \equiv m \bmod{d}\\ (n, q) = 1}} \lambda_2(n) V_2\left(\frac{n}{N}\right). 
\end{displaymath}
Let us write $q = q_d q'$ where $q_d = (q, d^{\infty})$ and hence $(q', d) = 1$. Since $(m, q) = 1$ and $n \equiv m$ (mod $d$), the conditions $(n, q) = 1$ and $(n, q') = 1$ are equivalent. We remove the latter condition by M\"obius inversion and \eqref{hecke}, getting
\begin{equation}
\label{HeckeSectionSNMd}
\begin{split}
  S_{N, M, d, q} & = \frac{d}{(NM)^{1/2}}\sum_{f \mid q'} \mu(f) \sum_{(m, q) = 1} \lambda_1(m) V_1\left(\frac{m}{M}\right) \sum_{ n \equiv \bar{f}m \bmod{d}} \lambda_2(fn) V_2\left(\frac{fn}{N}\right)\\
  & = \frac{d}{(NM)^{1/2}}\sum_{g \mid f \mid q'} \mu(f)\mu(g) \lambda_2\left(\frac{f}{g}\right) \sum_{(m, q) = 1} \lambda_1(m) V_1\left(\frac{m}{M}\right) \sum_{ n \equiv \overline{fg}m \bmod{d}} \lambda_2(n) V_2\left(\frac{fgn}{N}\right). 
  \end{split}
\end{equation}

The innermost sum in \eqref{HeckeSectionSNMd} equals
\begin{displaymath}
 \frac{1}{d} \sum_{ r \mid  d} \underset{b \bmod{r}}{\left.\sum \right.^{\ast}} e\left(\frac{\overline{fg}mb}{r}\right) \sum_{n} \lambda_2(n) e\left(-\frac{bn}{r}\right) V_2\left(\frac{fgn}{N }\right).
\end{displaymath}
Applying the Voronoi summation formula (Lemma \ref{vor}) to the $n$-sum, this is further equal to
\begin{displaymath}
\frac{1}{d} \sum_{r \mid  d} \frac{N}{fgr}  \sum_{n}S\bigl(\overline{fg}m, n, r\bigr) \lambda_2(n)   \mathring{V}_2\left(\frac{nN}{fg r^2}\right). 
 \end{displaymath}

Inserting this transformed sum back into \eqref{HeckeSectionSNMd}, applying the Cauchy-Schwarz inequality to the $m$-sum, and using \eqref{RS}, we obtain
 \begin{displaymath}
 \begin{split}
   S_{N, M, d, q}  &\ll \frac{1}{N^{1/2}} \sum_{g \mid f \mid q'}\mu^2(f) \Bigl| \lambda_2\left(\frac{f}{g}\right)\Bigr| \sum_{r \mid  d}  \frac{N}{fgr}\Bigl( \sum_{\substack{m \asymp M\\(m, q) = 1}}   \Bigl| \sum_{n}S\bigl(\overline{fg}m, n, r\bigr) \lambda_2(n)   \mathring{V}_2\left(\frac{nN}{fgr^2 }\right)\Bigr|^2\Bigr)^{1/2}\\
   & \ll \frac{1}{N^{1/2}} \sum_{g \mid f \mid q'}\mu^2(f) \Bigl| \lambda_2\left(\frac{f}{g}\right)\Bigr| \sum_{r \mid  d}  \frac{N}{f gr}\Bigl( \sum_{n_1, n_2 \ll fgr^2 q^{\varepsilon} /N} |\lambda_2(n_1)\lambda_2(n_2) \mathcal{S}_M(\overline{fg}n_1, \overline{fg}n_2, r)|  \Bigr)^{1/2} +  q^{-10}
  \end{split} 
 \end{displaymath}
by the rapid decay of $\mathring{V}_2$ (recall \eqref{derivative}), where
 \begin{equation}
\label{DefMathcalSM}
\mathcal{S}_M( n_1,  n_2, r) =   \sum_{\substack{m \asymp M\\ (m, q) = 1}}  S( m, n_1, r) S( m, n_2, r).
\end{equation}
(This depends also on $q$, but this is not displayed in the notation.) 
Applying \eqref{deligne}, we obtain our basic estimate
\begin{equation}\label{separatedelta}
 S_{N, M, d, q} \ll \frac{q^{\varepsilon}}{N^{1/2}} \sum_{g \mid f \mid q'}  \sum_{r \mid  d}  \frac{N }{fg r}\Bigl( \sum_{n_1, n_2 \ll fgr^2 q^{\varepsilon}/ N } | \mathcal{S}_M(\overline{fg}n_1, \overline{fg}n_2, r)|  \Bigr)^{1/2}.
 \end{equation}
 Now let $q_1$ be a divisor of $q$ with $(q, 6^{\infty}) \mid q_1$ and write $s = (r, q_1)$. Then in particular $(r, 6^{\infty}) \mid s$. Applying Theorem \ref{klooster-short}, we can bound  $S_{N, M, d, q}$ by 
\begin{displaymath}
\begin{split}
\ll  & q^{\varepsilon} N^{\frac{1}{2}}\! \sum_{g \mid f \mid q'}  \sum_{r \mid  d} \frac{1}{fgr} \Bigl( \sum_{n_1, n_2 \ll fgr^2 q^{\varepsilon}/ N } \!\!\!\! M^{1/2} q q_1^{1/2} + \frac{M^{1/2} q^{5/4}}{(r, q_1)^{1/4}} + M q^{3/4} (q, n_1-n_2)^{1/4}q_1^{1/4} + qq_1 + q^{11/8} q_1^{1/8}\Bigr)^{1/2}\!\!,
  \end{split}
 \end{displaymath}
 and Proposition \ref{bound1} follows. 

\section{Automorphic Preliminaries II}\label{Hecketheory1}


Unfortunately not all cusp forms are newforms. An $L^2$-basis $\mathcal{B}_k(\ell)$ for the finite-dimensional vector space $S_k(\ell)$,  the space of holomorphic cusp forms of weight $k$ and level $\ell$, and an $L^2$-basis $\mathcal{B}(\ell, t)$ for  $\mathcal{A}(\ell, t)$, the space of  Maa{\ss} forms of level $\ell$ and spectral parameter $t$, will in general also include oldforms. We describe the procedure in detail for Maa{\ss} forms, the holomorphic case requires only small notational changes. For $\ell_1 \mid \ell$ let  $\mathcal{B}^{\ast}(\ell_1, \ell, t) \subseteq \mathcal{B}(\ell, t)$ denote the set of all $L^2(\Gamma_0(\ell)\backslash \mathbb{H})$-normalized newforms of level $\ell_1$ and spectral parameter $t$ and write $f|_d(z) := f(dz)$. Then by newform theory we have
  \begin{equation}\label{newform}
 \mathcal{A}(\ell, t) = \underset{\substack{\vspace{1mm} \\\ell_1 \mid \ell}}{\text{\LARGE $\bigcirc\!\!\!\!\!\!\!\perp$}} \,\, \underset{\substack{\vspace{1mm}\\ f \in \mathcal{B}^{\ast}(\ell_1, \ell, t)}}{ \text{\LARGE $\bigcirc\!\!\!\!\!\!\!\perp$}}  \,\,  \bigoplus_{d \mid \frac{\ell}{\ell_1}}  f|_d \cdot \mathbb{C}.
 \end{equation}
The first two sums are orthogonal; the last one is, in general, not orthogonal and needs to be orthogonalized by Gram-Schmidt. In this way we get an orthogonal basis $\mathcal{B}(\ell, t)$ of $\mathcal{A}(\ell, t)$, and we collect all spectral parameters to obtain $\mathcal{B}(\ell) := \coprod_t \mathcal{B}(\ell, t)$, and correspondingly
\[ \mathcal{B}^{\ast}(\ell_1, \ell) := \coprod_t \mathcal{B}^{\ast}(\ell_1, \ell, t). \]

The Fourier coefficients of the forms in the bases $\mathcal{B}_k(\ell)$ and $\mathcal{B}(\ell)$ are not exactly multiplicative, but almost so. More precisely \cite[p.\ 74]{BHM1}, if $m = qm' \in \mathbb{N}$ with $(m', q) = 1$, then
\begin{equation}\label{mult}
  \sqrt{m} \rho_f(m) = \sum_{d \mid (\ell, q/(q, \ell))} \mu(d)\chi_0(d) \lambda_{f^{\ast}}\left(\frac{q}{d(q, \ell)}\right) \left(\frac{(\ell, q)m'}{d}\right)^{1/2} \rho_f\left(\frac{(\ell, q)m'}{d}\right)
  \end{equation}
where $f^{\ast}$ is the underlying newform. In particular, if $(q,\ell)=1$, then
\begin{equation}\label{extra}
   \sqrt{m}\rho_f(m)=\lambda_{f^{\ast}}(q)\sqrt{m'}\rho_f(m'). 
   \end{equation}
Moreover, if  $f^{\ast}$ satisfies the Ramanujan conjecture and $a_{m}$ is any finite sequence of complex numbers supported on integers $m = qm'$ with $(m', q) = 1$, then
\begin{equation}\label{multRama}
  \Bigl|\sum_m a_m \sqrt{m} \rho_{f}(m)\Bigr|^2 \leq \tau(q)^2 \sum_{d \mid (q, \ell)}  \Bigl|\sum_{m'} a_{qm'} \sqrt{dm'} \rho_{f}(dm')\Bigr|^2.
\end{equation} 
A somewhat involved explicit calculation shows a similar result \cite[p.\ 80]{BHM1} for the coefficients $\rho_{\mathfrak{a}}(m, t)$ of Eisenstein series: if $q \in \mathbb{N}$ and $a_{m}$ is any finite sequence of complex numbers supported on integers $m = qm'$ with $(m', q) = 1$, then
\begin{equation}\label{multEis}
  \sum_{\mathfrak{a}} \Bigl|\sum_m a_m \sqrt{m} \rho_{\mathfrak{a}}(m, t)\Bigr|^2 \leq 9\tau(\ell)^3\tau(q)^4 \sum_{d \mid (q, \ell)} \sum_{\mathfrak{a}} \Bigl|\sum_{m'} a_{qm'} \sqrt{dm'} \rho_{\mathfrak{a}}(dm', t)\Bigr|^2
\end{equation} 
for all $t \in \mathbb{R}$.  

The relation \eqref{mult} is very useful, but not sufficient for all our purposes. We proceed to make the orthogonalization process in \eqref{newform} explicit. 
For a newform $f\in \mathcal{B}^{\ast}(\ell_1, \ell)$  we define the following arithmetic functions:
\begin{displaymath}
\begin{split}
&r_f(c ) := \sum_{b \mid c} \frac{\mu(b)  \lambda_f(b)^2}{b \cdot \sigma_{-1}(b)^2}, \quad \alpha(c ) := \sum_{b \mid c} \frac{\mu(b)}{b^2}, \quad \beta(c ) = \sum_{b \mid c} \frac{\mu^2(b)}{b},\\
& \mu_f(c ) \text{ given by } L(f, s)^{-1} = \sum_{c} \frac{\mu_f(c )}{c^s}, \text{ so } \mu_f(p ) =  - \lambda_f( p), \, \mu_f(p^2) = \chi_0(p ), \, \mu_f(p^{\nu}) = 0, \, \nu > 2, 
\end{split}
\end{displaymath}
where $\sigma_{-1}(b)$ is the sum of the reciprocal divisors of $b$ and $\chi_0$ is the trivial character modulo $\ell_1$. For $d \mid g$ define 
\begin{displaymath}
  \xi'_g(d) := \frac{\mu(g/d) \lambda_f(g/d)  }{r_f(g)^{1/2} (g/d)^{1/2}\beta(g/d)}, \quad \xi''_g(d) = \frac{\mu_f(g/d)}{(g/d)^{1/2} (r_f(g)\alpha(g))^{1/2}}. 
\end{displaymath}
Write uniquely $g = g_1g_2$ where $g_1$ is squarefree, $g_2$ is squarefull, and $(g_1,g_2) = 1$. Then for $d \mid g$ we define 
\begin{equation}\label{defxi}
  \xi_g(d) = \xi'_{g_1}((g_1, d)) \xi''_{g_2}((g_2, d)) \ll g^{\varepsilon} (g/d)^{\theta - 1/2} . 
\end{equation}
The following lemma is an extension of \cite[Section 2]{ILS} to non-squarefree levels. It is essentially contained \cite[Proposition 5]{Ro}. As this result is crucial for us, and the assumptions are a little different from \cite{Ro}, we  provide a complete proof.

\begin{lemma}\label{basis} Let $\ell_1 \mid \ell$, and let $f^{\ast} \in \mathcal{B}^{\ast}(\ell_1, \ell) \subseteq \mathcal{B}(\ell)$ be an  $L^2(\Gamma_0(\ell)\backslash \mathbb{H})$-normalized newform of level $\ell_1$.  Then the set of functions
\[ \Bigl\{ f^{(g)} := \sum_{d \mid g} \xi_g(d) f^{\ast}|_d \,: \, g \mid \frac{\ell}{\ell_1}\Bigr\} \]
is an orthonormal basis of the space $\bigoplus_{d \mid \frac{\ell}{\ell_1}}  f^{\ast}|_d \cdot \mathbb{C}$. 

If $f$ is any member in this basis, then its Fourier coefficients satisfy the bound
\begin{equation}\label{boundrhogeneral}
\sqrt{n}  \rho_f(n) \ll (n\ell)^{\varepsilon} n^{\theta} (\ell, n)^{1/2-\theta} |\rho_{f^{\ast}}(1)|. 
\end{equation}
\end{lemma}

\bigskip

\textbf{Remark:} We stress that $f^{\ast}$ of level $\ell_1$ is normalized as in \eqref{innerprod}, i.e.\ with respect to the group $\Gamma_0(\ell)$. The map $\mathcal{B}^{\ast}(\ell_1, \ell) \rightarrow \mathcal{B}^{\ast}(\ell_1, \ell_1) \subseteq \mathcal{B}(\ell_1)$ is not an isometry, but reduces the norm by a factor $ [\Gamma_0(\ell_1) : \Gamma_0(\ell)]^{-1/2}$.

\bigskip

Although we do not need it in the present paper, we remark that with the definition $f|_d (z) := d^{k/2}f(dz)$ the same construction (and the same proof) works for holomorphic cusp forms of weight $k$, and in particular the bound \eqref{boundrhogeneral} remains true with $\theta = 0$ for holomorphic cusp forms. Moreover, the trivial character $\chi_0$ modulo $\ell_1$ plays no special role, the same construction and the same proof work for any Dirichlet character $\chi$ modulo $\ell_1$.

\begin{proof}  We write $\tilde{\ell} := \ell/\ell_1$. As a first step we need to compute the Gram matrix 
$(\langle f^{\ast}|_{d_1}, f^{\ast}|_{d_2}\rangle)_{d_1, d_2 \mid \tilde{\ell}}$ where all inner products are as in \eqref{innerprod}.  
Write $d'_1 = d_1/(d_1, d_2)$, $d'_2 = d_2/(d_1, d_2)$. As in \cite{ILS} we apply  Rankin-Selberg theory. First we observe that $\langle f^{\ast}|_{d_1}, f^{\ast}|_{d_2}\rangle = \langle f^{\ast}|_{d_1'}, f^{\ast}|_{d_2'}\rangle$ since multiplication by a scalar $(d_1, d_2)$ is an isometry. Let $E(z, s)$ be the standard   non-holomorphic Eisenstein series of level $\ell$. Then we unfold and use \eqref{four} and \eqref{relation} to obtain
\begin{displaymath}
  \langle E(\cdot, s) f^{\ast}|_{d_1'}, f^{\ast}|_{d_2'}\rangle = \int_{0}^{\infty} \int_0^1 y^s f^{\ast} (d_1'z)\bar{f}^{\ast}(d'_2z) \frac{dx\, dy}{y^2} = 2\sum_{n=1}^{\infty} \frac{\lambda_{f^{\ast}}(d_2'n)\lambda_{f^{\ast}}(d_1'n)}{(d_1'd_2')^{s-1/2}n^s} \int_0^{\infty} y^s |W_{0, it}(y)|^2 \frac{dy}{y^2}.
  \end{displaymath} 
We use \eqref{hecke} to evaluate  the Dirichlet series
\[ \sum_{n} \lambda_{f^{\ast}}(d_1'n) \lambda_{f^{\ast}}(d_2'n) n^{-s} = \sum_{(n, d_1'd_2') = 1}  \lambda_{f^{\ast}}(n)^2 n^{-s} \prod_{p^{e_p}  \parallel d_1'd_2'} \sum_{\nu=0}^{\infty} \lambda_{f^{\ast}}(p^{\nu+e_p})\lambda_{f^{\ast}}(p^{\nu}) p^{-\nu s} \]
and compare residues on both sides at $s=1$. In this way we obtain
 \begin{equation*}\label{recur}
  \langle f^{\ast}|_{d_1}, f^{\ast}|_{d_2} \rangle = \langle f^{\ast}|_{d'_1}, f^{\ast}|_{d'_2} \rangle =A(d_1'd_2')\langle f^{\ast}, f^{\ast}\rangle =   A\left(\frac{{\rm lcm}(d_1, d_2)}{{\rm gcd}(d_1, d_2)} \right) \langle f^{\ast}, f^{\ast}\rangle,
\end{equation*}
where $A$ is the multiplicative function given by
\begin{equation}\label{rec}
 A( p ) = \frac{\lambda_{f^{\ast}}(p )}{\sqrt{p}(1 + 1/p)}, \quad A(p^{\nu+1}) = \frac{\lambda_{f^{\ast}}(p )}{\sqrt{p}}A(p^\nu) - \frac{\chi_0(p )}{p} A(p^{\nu-1}).
 \end{equation}
(Here again $\chi_0$ is the trivial character modulo $\ell_1$.) We need to verify that
\begin{displaymath}
  \sum_{d_1 \mid g_1} \sum_{d_2 \mid g_2} \xi_{g_1} (d_1) \xi_{g_2}(d_2) A\left(\frac{{\rm lcm}(d_1, d_2)}{{\rm gcd}(d_1, d_2)} \right) = \delta_{g_1= g_2}.
\end{displaymath}
By multiplicativity and symmetry it is enough to consider the case $g_1 = p^{\alpha}$, $g_2 = p^{\beta}$ for a prime $p$ and $\beta \geq \alpha\geq 0$, so that it suffices to verify
\[ I(\alpha, \beta) :=  \sum_{\delta_1 \leq \alpha} \sum_{\delta_2 \leq \beta} \xi_{p^{\alpha}}(p^{\delta_1}) \xi_{p^{\beta}}(p^{\delta_2}) A(p^{|\delta_1 - \delta_2|}) = \delta_{\alpha = \beta}. \]
For prime powers, the arithmetic function $\xi_g(d)$ simplifies as follows:
\begin{displaymath}
\begin{split}
&  \xi_1(1) = 1, \quad \quad\quad  \xi_p(p ) = r_{f^{\ast}}(p )^{-1/2}, \quad \xi_p(1) = \frac{-\lambda_{f^{\ast}}(p )}{\sqrt{p}(1 + 1/p) } \xi_{p}(p),\\
&  \xi_{p^{\nu}}(p^{\nu}) = \left(r_{f^{\ast}}( p)(1-p^{-2})\right)^{-1/2}, \quad  \xi_{p^{\nu}}(p^{\nu-1}) =  \frac{-\lambda_{f^{\ast}}(p )}{\sqrt{p}} \xi_{p^{\nu}}(p^{\nu}), \quad \xi_{p^{\nu}}(p^{\nu-2})  = \frac{\chi_0(p )}{p} \xi_{p^{\nu}}(p^{\nu}), \quad \nu \geq 2,
\end{split}
\end{displaymath}
and $\xi_{p^{a}}(p^b) = 0$ in all other cases. In particular, for $\nu \geq 2$ and $c \leq \nu$, the value $\xi_{p^{\nu}}(p^{\nu-c})$ depends only on $p$ and $c$, but not on $\nu$. Hence 
\begin{equation}\label{rec1}
I(\alpha, \beta) = I(\alpha+c, \beta+c)
\end{equation}
 for any $c \in \mathbb{N}$ and any $2 \leq \alpha \leq \beta$, and by the recurrence relation in \eqref{rec} we also have 
 \begin{equation}\label{rec2}
 I(\alpha, \beta+1) = \frac{\lambda_{f^{\ast}}(p )}{\sqrt{p}} I(\alpha, \beta) - \frac{\chi_0(p )}{p} I(\alpha, \beta - 1)
 \end{equation}
if $\beta \geq \alpha + 3$ (this condition is needed to ensure that the summation indices $\delta_1, \delta_2$ satisfy $\delta_2 - \delta_1 \geq 0$ in all arising sums). By \eqref{rec1}, it suffices to assume $\alpha \leq 2$, and by \eqref{rec2} it suffices to assume $\beta - \alpha \leq 3$; the rest follows by induction. This leaves us with the 12 cases $0 \leq \alpha \leq \beta \leq \alpha + 3 \leq 5$, which are straightforward to verify. 

The bound \eqref{boundrhogeneral} now follows from
\begin{displaymath}
  \rho_{f^{(g)}}(n) = \sum_{d \mid g} \xi_g(d) \rho_{f^{\ast}}(n/d)
\end{displaymath}
(with the convention $\rho(x) = 0$ for $x \not \in \mathbb{Z}$), \eqref{relation},  \eqref{KS}, and \eqref{defxi}. \end{proof}

 We define the following integral transforms for a smooth function $\phi : [0, \infty) \rightarrow \mathbb{C}$ satisfying $\phi(0) = \phi'(0) = 0$, $\phi^{(j)}(x) \ll (1+x)^{-3}$ for $0 \leq j \leq 3$:
\begin{equation}
\label{IntegralTransforms}
\begin{split}
 & \dot{\phi}(k) = 4i^k \int_0^{\infty} \phi(x) J_{k-1}(x) \frac{dx}{x},\\
 & \tilde{\phi}(t) = 2\pi i \int_0^{\infty} \phi(x) \frac{J_{2it}(x) - J_{-2it}(x)}{\sinh(\pi t)} \frac{dx}{x},\\
 & \check{\phi}(t) = 8 \int_0^{\infty} \phi(x) \cosh(\pi t) K_{2it}(x) \frac{dx}{x}. 
\end{split}
\end{equation}
With the already established notation, the following spectral sum formula holds (see e.g. \cite[Theorem 2]{BHM1}).

\begin{lemma}\label{kuznetsov} [Kuznetsov formula] Let $\phi$ be as in the previous paragraph, and let $a, b,\ell> 0$ be integers. Then,
\begin{displaymath}
\begin{split}
  \sum_{\ell \mid c} \frac{1}{c}S(a, b, c) \phi\left(\frac{4\pi \sqrt{ab}}{c}\right) = & \sum_{\substack{k \geq 2\\ k \text{ even}}} \sum_{f \in \mathcal{B}_k(\ell)} \dot{\phi}(k) \Gamma(k) \sqrt{ab}  {\rho_f(a)} \rho_f(b)\\
  &+ \sum_{f \in \mathcal{B}(\ell)} \tilde{\phi}(t_f) \frac{  \sqrt{ab}}{\cosh(\pi t_f)}  {\rho_f(a)} \rho_f(b)\\
  & + \frac{1}{4\pi } \sum_{\mathfrak{a}} \int_{-\infty}^{\infty}\tilde{\phi}(t) \frac{  \sqrt{ab}}{\cosh(\pi t)}  {\rho_{\mathfrak{a}}(a, t)} \rho_{\mathfrak{a}}(b, t) dt 
 \end{split} 
\end{displaymath}
and
\begin{displaymath}
\begin{split}
  \sum_{\ell \mid c} \frac{1}{c}S(a, -b, c) \phi\left(\frac{4\pi \sqrt{ab}}{c}\right) = &  \sum_{f \in \mathcal{B}(\ell)} \check{\phi}(t_f) \frac{ \sqrt{ab}}{\cosh(\pi t_f)}  {\rho_f(a)} \rho_f(-b)\\
  & + \frac{1}{4\pi} \sum_{\mathfrak{a}} \int_{-\infty}^{\infty}\check{\phi}(t) \frac{  \sqrt{ab}}{\cosh(\pi t)}  {\rho_{\mathfrak{a}}(a, t)} \rho_{\mathfrak{a}}(-b, t) dt.  
 \end{split} 
\end{displaymath}
\end{lemma}

Often the Kuznetsov formula is used hand in hand with the large sieve inequalities of Deshouillers-Iwaniec \cite{DI}. 
\begin{lemma}\label{largesieve} [Spectral large sieve] Let $T, M \geq 1$, $\ell \in \mathbb{N}$, and let $(a_m)$, $M \leq m \leq 2M$, be a sequence  of complex numbers. Then all three quantities
\begin{displaymath}
\begin{split}
 & \sum_{\substack{2 \leq k \leq T\\ k \text{ even}}}\Gamma(k)  \sum_{f \in \mathcal{B}_k(\ell)}\Bigl| \sum_m a_m \sqrt{m} \rho_f(m)\Bigr|^2, \quad \sum_{\substack{f \in \mathcal{B}(\ell)\\ | t_f| \leq T} }\frac{1}{\cosh(\pi t_f)} \Bigl| \sum_m a_m \sqrt{m} \rho_f(\pm m)\Bigr|^2, \\
  & \sum_{\mathfrak{a}} \int_{-T}^T \frac{1}{\cosh(\pi t)} \Bigl| \sum_m a_m \sqrt{m} \rho_{\mathfrak{a}}(\pm m, t)\Bigr|^2 dt
  \end{split}
\end{displaymath}
are bounded by
\begin{displaymath}
  M^{\varepsilon} \left(T^2 + \frac{M}{\ell}\right) \sum_{m} |a_m|^2. 
\end{displaymath}
\end{lemma}

Another application of the Kuznetsov formula is the following bound.

\begin{lemma}\label{moto} Let $T \geq 1$, $m, \ell \in \mathbb{N}$. Then
\begin{displaymath}
  \sum_{\substack{|t_f| \leq T\\ f \in \mathcal{B}(\ell)}}\frac{1}{\cosh(\pi t_f)} |\sqrt{m}\rho_f(m)|^2 \ll \left(T^2 + \frac{(\ell, m)^{1/2}m^{1/2 }}{\ell}\right)(T m)^{\varepsilon}
\end{displaymath}
with an implied constant depending only on $\varepsilon$. 
\end{lemma}

\begin{proof} This is \cite[Lemma 2.4]{Mot} for $\ell = 1$, and the proof in the more general case is verbatim the same, except that in \cite[(2.3.7), (2.3.10)]{Mot} an additional divisibility condition   is added in the sum over Kloosterman sums that leads to an obvious modification of the last two displays in the proof. 
\end{proof}


The following important result  will be used to avoid the Ramanujan conjecture. 

\begin{theorem}\label{avoid} Let $\ell , s \in \mathbb{N}$, $R,  T\geq 1$, and let $\alpha( r)$,  $R \leq r \leq 2R$, be any sequence of complex numbers with $|\alpha(r ) | \leq 1$. Then
\[ \sum_{\substack{|t_f|  \leq T\\ f \in \mathcal{B}(\ell )  }} \frac{1}{\cosh(\pi t_f)}      \Bigl|\sum_{\substack{R \leq r \leq 2R \\ (r, s\ell) = 1}} \alpha(r )\sqrt{r s} \rho_f(r s) \Bigr|^2 \ll (\ell s T R)^{\varepsilon} (\ell , s) \left(T+\frac{s^{1/2}}{\ell^{1/2}}\right)\left(T + \frac{R}{\ell^{1/2}}\right) R. \]
\end{theorem}

\begin{proof}
We call the left hand side $\Xi$. Fix an $f \in \mathcal{B}(\ell)$ and denote by $f^{\ast} \in \mathcal{B}^{\ast}(\ell_1, \ell)$ the underlying newform of level $\ell_1$, say. An  application of \eqref{mult} and \eqref{extra} shows for $(r, s\ell) = 1$ that
\begin{displaymath}
\begin{split}
\sqrt{rs}   \rho_f(rs) &= \sum_{\delta \mid (\ell , \frac{s}{(s, \ell )})} \mu(\delta) \chi_0(\delta) \lambda_{f^{\ast}}\left(\frac{s}{\delta(\ell , s)}\right) \left(\frac{(\ell , s)r}{\delta}\right)^{1/2} \rho_f\left(\frac{(\ell , s)r}{\delta}\right)\\
  & = \sum_{\delta \mid (\ell, \frac{s}{(s, \ell )})} \mu(\delta) \chi_0(\delta)  \lambda_{f^{\ast}}\left(\frac{s}{\delta(\ell , s)}\right) \left(\frac{(\ell , s)}{\delta}\right)^{1/2} \rho_f\left(\frac{(\ell , s)}{\delta}\right)\lambda_{f^{\ast}}(r ). 
  \end{split}
\end{displaymath}
We apply the Cauchy-Schwarz inequality  first to the sum over $\delta$ and then to the sum over $f \in \mathcal{B}(\ell_1\ell_2)$  to obtain
\begin{displaymath}
 \Xi  \leq \tau(s)^{1/2}  \Theta_2^{1/2} \sum_{  \delta \mid (\ell, \frac{s}{(s, \ell )})} \Theta_1^{1/2}
\end{displaymath}
where
\begin{displaymath}
  \Theta_1 =  \sum_{\substack{f \in \mathcal{B}(\ell )\\ |t_f| \leq T}} \frac{1}{\cosh(\pi t_f)^2} \Big|\lambda_{f^{\ast}}\left(\frac{s}{\delta(\ell , s)}\right) \left(\frac{(\ell , s)}{\delta}\right)^{1/2} \rho_f\left(\frac{(\ell , s)}{\delta}\right)\Big|^4 
\end{displaymath}
and
\begin{displaymath}
  \Theta_2 = \sum_{\substack{f \in \mathcal{B}(\ell )\\ |t_f| \leq T}}  \Bigl|\sum_{\substack{R \leq r \leq 2R  \\ (r, s\ell) = 1}} \alpha(r ) \lambda_{f^{\ast}}(r )\Bigr|^4.
\end{displaymath}
The main idea is to transform the sums
$\Theta_1$ and $\Theta_2$ into sums to which Lemma \ref{moto} and Lemma \ref{largesieve}, respectively, may
be applied. By a crude application of \eqref{boundrhogeneral}, the M\"obius inverse of \eqref{hecke} and \eqref{relation}   we have
\begin{displaymath}
\begin{split}
  \Theta_1& \ll   (\ell , s)^{2}\ell^{\varepsilon} \sum_{\substack{f \in \mathcal{B}(\ell )\\ |t_f| \leq T}} \frac{1}{\cosh(\pi t_f)^2} \Big|\lambda_{f^{\ast}}\left(\frac{s}{\delta(\ell , s)}\right)\rho_{f^{\ast}}(1)\Bigr|^4\\
  & \leq   \tau(s)    (\ell , s)^{2}\ell^{\varepsilon} \sum_{g \mid \frac{s} {\delta(\ell , s)}}\sum_{\substack{f \in \mathcal{B}(\ell )\\ |t_f| \leq T}} \frac{|\rho_{f^{\ast}}(1)|^4}{\cosh(\pi t_f)^2} \Big|\lambda_{f^{\ast}}\left(\frac{s^2}{(g\delta(\ell , s))^2}\right) \Bigr|^2\\
   & =   \tau(s)  (\ell , s)^{2}\ell^{\varepsilon} \sum_{g \mid \frac{s} {\delta(\ell , s)}}\sum_{\substack{f \in \mathcal{B} (\ell )\\ |t_f| \leq T}} \frac{|\rho_{f^{\ast}}(1)|^2}{\cosh(\pi t_f)^2} \Big| \frac{s}{g\delta(\ell , s)} \rho_{f^{\ast}}\left(\frac{s^2}{(g\delta(\ell , s))^2}\right) \Bigr|^2.
\end{split}  
\end{displaymath}
The newform $f^{\ast} \in \mathcal{B}^{\ast}(\ell_1, \ell) $ is counted   $\tau(\ell /\ell_1)$ times in the sum over $\mathcal{B}(\ell)$, and we sum now over $L^2(\Gamma_0(\ell_1)\backslash\mathbb{H})$-normalized newforms $f  \in \mathcal{B}^{\ast}(\ell_1, \ell_1)\subseteq \mathcal{B}(\ell_1)$ which by the remark following Lemma \ref{basis} leads to a renormalizing factor $(\ell /\ell_1)^{-2+ o(1)}$. Hence by \eqref{rho1} we conclude 
 \begin{displaymath}
\begin{split}
  \Theta_1&  \ll   \tau(s)  (\ell , s)^{2}\ell^{\varepsilon} \sum_{g \mid \frac{s} {\delta(\ell , s)}}\sum_{\ell_1 \mid \ell }\frac{\tau(\ell /\ell_1)}{(\ell /\ell_1)^2}   \sum_{\substack{f \in \mathcal{B}^{\ast} (\ell_1, \ell_1)\\ |t_f| \leq T}} \frac{|\rho_{f}(1)|^2}{\cosh(\pi t_f)^2} \Big| \frac{s}{g\delta(\ell , s)} \rho_{f}\left(\frac{s^2}{(g\delta(\ell , s))^2}\right) \Bigr|^2\\
  & \ll  \frac{ (\ell , s)^{2 }}{\ell }(\ell s T)^{\varepsilon} \sum_{g \mid \frac{s} {\delta(\ell , s)}}\sum_{\ell_1 \mid \ell } \frac{1}{\ell /\ell_1}  \sum_{\substack{f \in \mathcal{B}^{\ast}(\ell_1, \ell_1)\\ |t_f| \leq T}} \frac{1}{\cosh(\pi t_f)} \Big| \frac{s}{g\delta(\ell , s)} \rho_{f}\left(\frac{s^2}{(g\delta(\ell , s))^2}\right) \Bigr|^2.
\end{split}  
\end{displaymath}
By positivity we can extend the innermost sum to all of $\mathcal{B}(\ell_1)$. 
By Lemma \ref{moto} we finally obtain
\begin{displaymath}
  \Theta_1 \ll   (\ell s TR)^{\varepsilon}  \frac{ (\ell , s)^{2 }}{\ell }\sum_{\ell_1 \mid \ell } \frac{1}{\ell /\ell_1}   \left(T^2 +\frac{s\big(\ell_1,  s^2/(s, \ell )^2\big)^{1/2}}{(\ell , s)\ell_1} \right) \leq (\ell s TR)^{\varepsilon}  \frac{ (\ell , s)^{2 }}{\ell } \left(T^2 +\frac{s}{\ell } \right) . 
\end{displaymath}
 
 Next we turn to the estimation of $\Theta_2$. By a similar argument we have
\begin{displaymath}
\begin{split}
  \Theta_2 & = \sum_{\ell_1 \mid \ell } \tau(\ell /\ell_1)  \sum_{\substack{f \in \mathcal{B}^{\ast}(\ell_1, \ell_1)\\ |t_f| \leq T}}  \Bigl|\sum_{\substack{R \leq r \leq 2R \\ (r, s\ell) = 1}} \alpha(r )  \lambda_{f}(r )\Bigr|^4 \\
  & =  \sum_{\ell_1 \mid \ell } \tau(\ell /\ell_1) \sum_{\substack{f \in \mathcal{B}^{\ast}(\ell_1, \ell_1)\\ |t_f| \leq T}}  \Bigl|\sum_{\substack{R \leq r, r' \leq 2R \\ (rr', s\ell) = 1}} \alpha(r )\alpha(r' )  \sum_{g \mid (r, r')} \lambda_{f}\left(\frac{rr'}{g^2}\right)\Bigr|^2\\
  & \ll  (\ell T)^{\varepsilon} \sum_{\ell_1 \mid \ell } \ell_1   \sum_{\substack{f \in \mathcal{B}^{\ast}(\ell_1, \ell_1 )\\ |t_f| \leq T}}  \frac{1}{\cosh(\pi t_f)}  \Bigl|\sum_{\substack{R \leq r , r '  \leq 2R \\ (r r ', s\ell) = 1}} \alpha(r )\alpha(r')  \sum_{g \mid (r, r')}\left(\frac{rr'}{g^2}\right)^{1/2} \rho_{f^{\ast}}\left(\frac{rr'}{g^2}\right)\Bigr|^2\\
  & = (\ell T)^{\varepsilon}  \sum_{\ell_1 \mid \ell } \ell_1  \sum_{\substack{f \in \mathcal{B}^{\ast}(\ell_1, \ell_1)\\ |t_f| \leq T}}  \frac{1}{\cosh(\pi t_f)}  \Bigl|\sum_{ r\ll R^2 }   \sqrt{r} \rho_f(r )\beta(r )\Bigr|^2
  \end{split}  
\end{displaymath}
 where
 \[ \beta( r)  = \sum_{\substack{R \leq r_1 , r_2  \leq 2R\\ (r_1r_2, s\ell) = 1}} \alpha(r_1 )\alpha(r_2)  \sum_{\substack{g \mid (r_1, r_2)\\ r_1r_2 = g^2 r}} 1  \ll  \sum_{g \ll R/\sqrt{r}} \tau(r ) \ll \frac{R^{1+\varepsilon}}{\sqrt{r}}. \]
 Again we complete the sum over $f$ to all of $\mathcal{B}(\ell_1)$. The large sieve (Lemma \ref{largesieve}) shows
 \begin{displaymath}
   \Theta_2 \ll (\ell TR)^{\varepsilon} \sum_{\ell_1 \mid \ell } \ell_1  \left(T^2 + \frac{R^2}{  \ell_1}\right) R^2,
 \end{displaymath}
 and the lemma follows.
  \end{proof}
  
 \textbf{Remark.} The important step in the proof in the application of the Cauchy-Schwarz inequality. A simpler strategy would  apply \eqref{extra} with $r = q$, $s = m'$ directly, estimate $\sqrt{s}\rho_f(s)$ by \eqref{boundrhogeneral} and apply the large sieve to obtain
\begin{equation}\label{simpler}
\sum_{\substack{|t_f|  \leq T\\ f \in \mathcal{B}(\ell )  }} \frac{1}{\cosh(\pi t_f)}      \Bigl|\sum_{\substack{R \leq r \leq 2R \\ (r, s\ell) = 1}} \alpha(r )\sqrt{r s} \rho_f(r s) \Bigr|^2 \ll (\ell s T R)^{\varepsilon} s^{2\theta} (\ell , s)^{1-2\theta} \left(T^2+  \frac{R}{\ell}\right) R.
\end{equation}
 

\section{Bessel functions}\label{secbessel}

We collect here some useful formulas for future reference. In view of the integral transform appearing in the Kuznetsov formula we write
\begin{equation}\label{defJ+}
\begin{split}
  \mathcal{J}^+_{2it}(x) & := \pi i \frac{J_{2 i t}(x) - J_{-2it}(x)}{\sinh(\pi t)}, \\
  \mathcal{J}^-_{2it}(x) & := 4 \cosh(\pi t) K_{2it}(x).
\end{split}  
\end{equation}
We start with    the power series expansion  \cite[8.402]{GR}
\begin{equation}\label{power}
  J_{\nu}(x) = \frac{x^{\nu}}{2^{\nu} } \sum_{k=0}^{\infty} (-1)^k \frac{x^{2k}}{2^{2k} k! \Gamma(\nu + k + 1)}
\end{equation}
valid for $x > 0$ and $\nu \in \mathbb{C}$. Next, we record the uniform asymptotic expansion \cite[7.13(17)]{Er}
\begin{equation}\label{asymp2}
  \frac{J_{ i t}(x)}{\sinh(\pi t/2)} =   \exp\left(i\sqrt{t^2+x^2} - it \, {\rm arcsinh}(t/x)\right) \mathcal{J}_M(t, x)  + {\rm O}\left((x + t)^{-M}\right) 
\end{equation}
for $t  > 1$ and any fixed $M \in \mathbb{N}$, where $\mathcal{J}_M(t, x)$ satisfies 
\begin{displaymath}
  x^j  \frac{\partial^j}{\partial x^j} \mathcal{J}_M(t, x) \ll_{M,  j} (t+x)^{-1/2}
\end{displaymath}
for any $ j \in \mathbb{N}_0$. The original error term in \cite{Er} is only ${\rm O}(x^{-M})$ in place of ${\rm O}((x + t)^{-M})$, but the stronger error term follows from the power series expansion \eqref{power}  for $x < t^{1/3}$. A similar expansion holds for $J_{-it}(x) = \overline{J_{it}(x)}$. By \cite[8.411.1]{GR} we have
\begin{equation}\label{int1}
  J_{k-1}(x) = \frac{1}{\pi} \int_0^{\pi} \cos((k-1)\xi - x \sin\xi) d\xi. 
\end{equation}
for $k \in \mathbb{N}$, and by \cite[6.561.16]{GR}  we have
\begin{displaymath}
\widehat{\mathcal{J}}_{2it}^-(s) = \cosh(\pi t) 2^{s-2}\Gamma\left(\frac{s}{2} + it\right) \Gamma\left(\frac{s}{2} - it\right), \quad \Re s > 2 |\Im t|. 
\end{displaymath}
In particular, for $\Re s = 1$, we have the  bound
\begin{equation}\label{supnorm}
  \widehat{\mathcal{J}}_{2it}^-(1 + i\tau)  \ll   e^{-\pi \max(0, \frac{|\tau|}{2} - |t|)}. 
\end{equation}

\begin{lemma}\label{boundbessel} Let $k \in \mathbb{N}$, $t \in \mathbb{R} \cup (-i/4, i/4)$, $x > 0$. Then 
\begin{equation}\label{uniform}
\begin{split}
  \mathcal{J}^{+}_{2it}(x)& \ll x^{-1/2},\\
    J_{k-1}(x) & \ll x^{-1/2}, \quad  x > 100k,  
  \end{split}
\end{equation}
with absolute implied constants. Moreover, for fixed $\nu \in \mathbb{C}$ and $j \in \mathbb{N}_0$, we have
\begin{equation}\label{fixed}
\frac{d^j}{dx^j} J_{\nu}(x) \ll_{\nu, j} \begin{cases} x^{\Re \nu - j}, & x \leq 1,\\
x^{-1/2}, & x \geq 1.
\end{cases} 
\end{equation}
\end{lemma}

\begin{proof} The bound for $J_{k-1}(x)$ follows from \cite[Lemma 4.2, 4.3]{Ra} for $k \geq 16$, while for $k < 16$ the bound is a trivial consequence of the asymptotic formula \cite[8.451.1]{GR}.  The bound for $\mathcal{J}^+_{2it}(x)$ for   $x \geq 1$ follows from \eqref{asymp2} and for $x < 1$ from the power series expansion \eqref{power}. This proves \eqref{uniform} 
The bound \eqref{fixed} follows similarly from \eqref{power} and \cite[8.451.1]{GR}. 
\end{proof}


\begin{lemma}\label{bessel-decomp} Let $\nu \in \mathbb{C}$ with $\Re \nu \geq 0$ be fixed. There exist smooth  functions $F^{\pm}_{\nu}(x)$ such that
\begin{equation}\label{1}
x^j (F^{\pm}_{\nu})^{(j)}(x) \ll_{\nu, j}   \min(x^{\Re \nu},  x^{-1/2})
\end{equation}
for all $j \in \mathbb{N}_0$ and
\begin{equation}\label{2}
J_{\nu}(x) = F_{\nu}^+(x) e^{ix} + F_{\nu}^-(x) e^{-ix}.
\end{equation}
\end{lemma}

\begin{proof} The idea is to use the asymptotic formula for $x \geq 1$ and a trivial decomposition for $x < 1$ and then to glue these decompositions together. To make this precise, we define $H_{\nu}^{(1)}(x) = J_{\nu}(x) + i Y_{\nu}(x)$ and $H_{\nu}^{(2)}(x) = J_{\nu}(x) - i Y_{\nu}(x)$ as in \cite[(8.405)]{GR} and write 
\begin{displaymath}
  H_{\nu}^{+}(x)  = H_{\nu}^{(1)}(x)e^{-ix}, \quad   H_{\nu}^{-}(x)  = H_{\nu}^{(2)}(x)e^{ix}.
\end{displaymath}
By \cite[8.476.10]{GR} we have $\overline{H_{\nu}^+(x)} = H_{\bar{\nu}}^-(x)$ for $x \in \mathbb{R}$.  Then, 
\begin{displaymath}
  J_{\nu}(x) =  \frac{1}{2}\left(H_{\nu}^{+}(x) e^{ix} +  H_{\nu}^{-}(x) e^{-ix}\right)
\end{displaymath}
by \cite[8.481]{GR}.  Finally, we choose a smooth function $V$ with support in $[1, \infty)$ and $V(x) = 1$ on $[2, \infty)$ and define
\[ F^+_{\nu}(x) := \frac{1}{2}H^+_{\nu}(x) V(x) + e^{-ix} J_{\nu}(x)(1 - V(x)), \quad F^-_{\nu}(x) := \frac{1}{2}H^-_{\nu}(x) V(x), \]
so that \eqref{2} holds.

We compute the derivatives of $H^{\pm}_{\nu}(x)$ for $x \geq 1$ using the integral representation (\cite[8.421.9]{GR})
$$H_{\nu}^+(x) = \left(\frac{2}{\pi x}\right)^{1/2} \frac{e(-\frac{2\nu+1}{8})}{\Gamma(\nu + 1/2)} \int_0^{\infty} \left(1 + \frac{it}{2x}\right)^{\nu - 1/2} t^{\nu-1/2} e^{-t} dt$$
and the derivatives of $e^{-ix} J_{\nu}(x)$ for $x \leq 2$ using \eqref{fixed}. This implies \eqref{1}. 
\end{proof}

The next lemma shows when the integral transforms of the Kuznetsov formula are negligibly small.

\begin{lemma}\label{boundtrafo} Let $Z \geq 1$, $X, P, \alpha  > 0$, and let $C\geq Z+X+P+\alpha$ be a large parameter. Let $\Omega$ be a smooth weight function of fixed compact support satisfying $ \Omega^{(j)}(x) \ll PZ^{j}$ for all $j \in\mathbb{N}_0$. Then the following bounds hold for any fixed $A > 0$. 
\begin{alignat}{3}\label{11}
& \int_0^{\infty} \Omega\left(\frac{x}{X}\right) e^{\pm i \alpha x}
\mathcal{J}_{2it}^+(x) \frac{dx}{x} \ll |t|^{-A},
 && \qquad\text{if}\quad  t \geq C^{\varepsilon}Z(X\sqrt{\alpha^2 - 1} + X^{1/2} +
1), \quad \alpha \geq 1;\\ \label{22}
&  \int_0^{\infty} \Omega\left(\frac{x}{X}\right) e^{\pm i \alpha x}
\mathcal{J}_{2it}^-(x) \frac{dx}{x} \ll |t|^{-A},
&&\qquad\text{if}\quad t \geq C^{\varepsilon}Z (X + \alpha X + 1);\\ \label{33}
&  \int_0^{\infty} \Omega\left(\frac{x}{X}\right) e^{\pm i \alpha x} J_{k-1}(x)
\frac{dx}{x} \ll k^{-A},
&&\qquad\text{if}\quad  k \geq C^{\varepsilon}Z(  X^{1/2} + 1), \quad \alpha \geq 1.
\end{alignat}
\end{lemma}

\begin{proof} This is essentially \cite[Lemma 3, Remark 1 \& 2]{J3}. We give a variant of the proof in \cite{J3}. 

By \cite[(2.14)]{BHM1}, all three bounds \eqref{11}, \eqref{22}, \eqref{33} hold if 
\begin{equation}\label{BHMcond}
 t, \, k \geq C^{\varepsilon} (X + Z+\alpha X).
\end{equation} 
In particular \eqref{22} is proved, and also   \eqref{11} if $\alpha \geq 2$, so in order to complete the proof of \eqref{11} we may assume $\alpha = 1 + \beta$ with $0 \leq \beta \leq 1$, and then we may also assume $X \geq Z^2$, for otherwise the size condition in  \eqref{11} implies \eqref{BHMcond}. Hence the range for $t$ not yet covered by \eqref{BHMcond} and the condition in \eqref{11} is contained in $C^{\varepsilon} \leq t \ll C^{\varepsilon}X$. We insert the uniform asymptotic formula \eqref{asymp2} getting (up to an admissible error)
\begin{displaymath}
  \int_0^{\infty} \Omega\left(\frac{x}{X}\right) \mathcal{J}_M(t, x) e^{i f(x)}  \frac{dx}{x}. 
\end{displaymath}
 where 
 \begin{displaymath}
 \begin{split}
 & f(x ) =  \pm  \alpha x \pm \Bigl(\sqrt{(2t)^2+x^2} - 2t \, {\rm arcsinh}(2t/x)\Bigr),\\
  &   f'(x) = \pm \alpha \pm \frac{\sqrt{(2t)^2 + x^2}}{x}, \quad f^{(j)}(x) \asymp \frac{t^2}{x^j \sqrt{t + x}} \quad (j \geq 2). 
\end{split}  
\end{displaymath}\\
Under the present size assumptions an integration by parts argument  as in \cite[Lemma 8.1]{BKY} with $U = X/Z$, $Q = X$, $ Y = t^2/\sqrt{t + X}$ now shows the bound in \eqref{11} provided
\begin{displaymath}
  \frac{\sqrt{t^2 + X^2}}{X}  - 1   \geq \beta + \left(\frac{Z}{X} + \frac{t}{X^{3/2}}\right)C^{\varepsilon}
\end{displaymath}
which is implied by the assumption (observe that the left hand side is of order $t^2/X^2$). 

Finally we prove \eqref{33}. Since $J_{k}(x) \ll e^{k/5}$ for $x \leq k/2$ (see \cite[Lemma 4.2]{Ra}), we may assume $k \ll X$. In combination with our current assumption this implies   $X \gg (C^{\varepsilon}Z)^2$ and $X^{1/2} \leq k \ll X$. 
We insert \eqref{int1} getting
\begin{displaymath}
  \int_{0}^{\infty} \Omega(x/X) e^{\pm i \alpha x} \int_{-\pi}^{\pi} \cos((k-1)\xi - x \sin \xi) d\xi \, \frac{dx}{x}.
\end{displaymath}
Repeated integrating by parts in the $x$-integral shows \eqref{33} if $ \alpha \geq 1 + C^{\varepsilon}Z/X$ (in particular if $\alpha \geq  2$). More precisely, we may extract  smoothly the range $\sin\xi = \pm 1 + {\rm O}(C^{\varepsilon} Z/X)$ from the $\xi$-integral at the cost of an admissible error. In the remaining $\xi$-integral we integrate by parts sufficiently to complete the proof of \eqref{33}. 
 \end{proof}

\begin{lemma}\label{analyzeW} Let $W$ be a  fixed smooth  function with support in $[1/2, 3]$ satisfying $W^{(j)}(x) \ll_j 1$ for all $j$. Let $\nu \in \mathbb{C}$   be a fixed number with $\Re \nu \geq 0$. For $z, w > 0$ define 
\begin{displaymath}
  W^{\ast}(z, w) =   \int_0^{\infty} W(y) J_{\nu}(4\pi\sqrt{yw + z}) dy. 
\end{displaymath}
 Fix $C \geq 1$ and  $A, \varepsilon > 0$. Then for $z \gg w$ we have 
\begin{equation}\label{decompW}
  W^{\ast}(z, w) = W_+(z, w) e(2 \sqrt{z}) + W_-(z, w) e(-2 \sqrt{z})  + {\rm O}_A(C^{-A})
\end{equation}
for suitable functions $W_{\pm}$ (depending on $\nu$) satisfying 
\begin{equation}\label{boundWpm}
   z^i w^j \frac{\partial^i}{\partial z^i}    \frac{\partial^j}{\partial w^j} W_{\pm}(z, w) \begin{cases}
   = 0, & \sqrt{z}/w \leq C^{-\varepsilon},\\
   \ll C^{\varepsilon(i+j)}   \min(z^{-1/4}, 1), & \text{otherwise.}\end{cases}
 \end{equation}
 for any   $i, j \in \mathbb{N}_0$. The implied constants depend on $i, j$ and $\nu$. 
\end{lemma}
 
\begin{proof} Integration by parts  in connection with \cite[8.472.3]{GR} (cf.\ \eqref{byparts}) yields
\begin{displaymath}
  \int_0^{\infty} W(y) J_{\nu}(4 \pi \sqrt{yw + z}) dy = \int_0^{\infty} \left(\frac{-\nu}{4\pi \sqrt{yw+z}} W(y) + \frac{\sqrt{yw + z}}{2\pi w} W'(y)\right) J_{\nu+1}(4 \pi \sqrt{yw + z}) dy, 
\end{displaymath}
for $z, w > 0$.  Repeated application   together with  \eqref{fixed} shows  \begin{displaymath}
  W^{\ast}(z, w) \ll_{A} \Bigl(
   \frac{\sqrt{z}}{w}\Bigr)^A 
 \end{displaymath}
for   $A \in \mathbb{N}_0$.  For $\sqrt{z}/w \leq C^{-\varepsilon}$ we obtain an admissible decomposition satisfying \eqref{decompW} and \eqref{boundWpm} by putting  $W_+(z, w) = W_-(z, w) = 0$. Let us now assume $\sqrt{z}/w \geq \frac{1}{2} C^{-\varepsilon}$. We insert the decomposition from Lemma \ref{bessel-decomp}   into the definition of $W^{\ast}(z, w)$. 
In this way we obtain a decomposition satisfying \eqref{decompW} by putting
\[ W_{\pm}(z, w) :=   \int_0^{\infty} W(y) F_{\nu}^{\pm}\bigl(4\pi \sqrt{yw+z}\bigr) \exp\bigl(\pm 4\pi i(\sqrt{yw + z} - \sqrt{z})\bigr)dy. \]
Now the second line of \eqref{boundWpm} is easily verified. As in Lemma \ref{bessel-decomp}, we glue these decompositions together to complete the proof of the lemma.
  \end{proof}

\begin{cor}\label{cor8}
  The double Mellin transform
\begin{displaymath}
\widehat{{W}}_{\pm}(s, t) = \int_{0}^{\infty} \int_0^{\infty} {W}_{\pm}(z, w) z^{s} w^t \frac{dz\, dw}{zw}
\end{displaymath}
is absolutely convergent in the tube domain defined by $\Re t > 0$, $0 < \Re s + \Re t/2 < 1/4$, and satisfies 
\begin{equation}\label{intbyparts}
\widehat{{W}}_{\pm}(s, t) \ll_{A, B, \varepsilon, \Re s, \Re t} C^{\varepsilon}  |s|^{-A}|t|^{-B}
\end{equation}
 in this region.   Moreover,  the Mellin inversion formula
\[ { W}_{\pm}(z, w) = \int_{(c_2)} \int_{(c_1)} \widehat{{W}}_{\pm}(s, t)  z^{-s} w^{-t} \frac{ds}{2\pi i} \frac{dt}{2\pi i}  \]
holds whenever $c_1, c_2 > 0$, $ c_1 + c_2/2< 1/4$. 
\end{cor}

\begin{proof} Repeated integration by parts gives
\[ \widehat{{W}}_{\pm}(s, t) \ll_{i, j}  |s|^{-i} |t|^{-j} \int_0^{\infty} \int_0^{\infty} z^i y^j {W}_{\pm}^{(i, j)}(z, w) z^{s-1} w^{t-1} dz \, dw. \]
Inserting \eqref{boundWpm} proves \eqref{intbyparts} in the desired range, and the Mellin inversion formula follows easily (for instance by applying first the one-dimensional inversion formula in $w$ and then in $z$). 
 \end{proof}

\textbf{Remark:} Lemma \ref{analyzeW} and Corollary \ref{cor8} play an important role in the analysis of shifted convolution sums for holomorphic cusp forms. In the Maa{\ss} case we need a small, but somewhat technical extension of these results. It is convenient to state it already at this point:

\begin{enumerate}
\item Lemma \ref{analyzeW} holds true for negative $w$ as long as $4|w| \leq z$ (with $|w|$ in place of $w$ in \eqref{boundWpm}). In this case  the support condition of $W$ implies $yw + z > 0$ and in fact $yw +z \asymp z$.  

\item In order to encode the condition $4|w| \leq z$ into Corollary \ref{cor8}, we proceed as follows: let $0 < z_0 < 1$ and let $W_0(z, w)$ be a smooth function on $[0, \infty) \times \mathbb{R}$ such that
\begin{itemize}
\item  $W_0(z, w) = 1$ if $5|w| \leq z$ and $z \geq z_0$, 
\item  $W_0(z, w) = 0$ if $4|w| \geq z$ or $z \leq \frac{1}{2} z_0$,
\item  $z^i|w|^jW^{(i, j)}_0(z, w) \ll_{i, j} 1$ for all $i, j \in \mathbb{N}_0$, uniformly in $z_0$. 
\end{itemize} 
 Define ${\tt W}_{\pm}(z, w) := W_0(z, w) W_{\pm}(z, w)$ with $W_{\pm}$ as in Lemma \ref{analyzeW}, and define $$\widehat{{\tt W}}_{\pm, \pm}(s, t) = \int_{0}^{\infty} \int_0^{\infty} {\tt W}_{\pm}(z, \pm w) z^{s} w^t \frac{dz\, dw}{zw}.$$
Then Corollary \ref{cor8} holds with $\widehat{{\tt W}}_{\pm, \pm}(s, t)$ in place of $\widehat{{W}}_{\pm}(s, t)$, and \eqref{intbyparts} is uniform in $z_0$. 
\end{enumerate}

\section{Spectral decomposition of shifted convolution sums}
\label{SpectralDecompositionSection}
This section is devoted to the spectral decomposition of the shifted convolution sum $\mathcal{D}(\ell_1, \ell_2, h, N , M )$, defined in \eqref{defD}. 
 We choose a large parameter
\begin{equation}\label{sizeQ} C := N^{1000}\end{equation}
and make the general   assumption
 \begin{equation}\label{sizeh}
  h \asymp  N \geq 20M.
\end{equation}
We can also assume without loss of generality that
\[ \ell_1, \ell_2 \leq 2N, \]
for otherwise $\mathcal{D}(\ell_1, \ell_2, h, N , M )$ vanishes trivially. 
Slightly more generally than in \eqref{derivative} we only assume that 
\begin{equation}\label{assonV}
 V_{1, 2} \text{ are supported in $[1, 2]$ and satisfy } V_{1, 2}^{(j)} \ll C^{j \varepsilon}.
 \end{equation}  

The weight function $V_2$ localizes $\ell_1 n$ in a dyadic interval of size $N$, but the summation condition $\ell_1n - \ell_2 m = h$ suggests that $\ell_1 n$ can, for a given $h$, vary only in an interval of length $M$. Therefore we attach a redundant weight function $W(\frac{\ell_1n - h}{M})$ to the sum where $W$ is smooth with bounded derivatives, constantly 1 on $[1, 2]$,  and  supported on $[1/2, 3]$.    With this notation, we can re-write 
\begin{displaymath}
\begin{split}
\mathcal{D}(\ell_1, \ell_2, h, N , M )& = \sum_{\ell_1n - \ell_2m = h} \lambda_1(m) \lambda_2(n)V_1\left(\frac{\ell_2m}{M}\right) V_2\left(\frac{\ell_2m+h}{N}\right) W\left(\frac{\ell_1n - h}{M}\right)\\
& = \int_{-\infty}^{\infty}V_2^{\dagger}(z) e\left(\frac{zh}{N}\right) \mathcal{D}_z(\ell_1, \ell_2, h, N , M ) dz,
\end{split}
\end{displaymath}
where $V_2^{\dagger}$ is the Fourier transform of $V_2$ and 
\[ \mathcal{D}_z(\ell_1, \ell_2, h, N , M ) = \sum_{\ell_1n - \ell_2m = h} \lambda_1(m) \lambda_2(n)V_z\left(\frac{\ell_2m}{M}\right)   W\left(\frac{\ell_1n - h}{M}\right), \]
with $V_z(x) = V_1(x) e(z x M/N)$. We can truncate the $z$-integral at $|z| \leq C^{\varepsilon}$ at the cost of an error ${\rm O}(C^{-100})$. 


\subsection{The circle method}\label{CircleMethodSubsection} 
The following lemma is Jutila's variant of the circle method \cite{J1, J2}. 

\begin{lemma}\label{jutila} [Jutila's circle method] Let $Q \geq 1$ and $Q^{-2}\leq \delta \leq Q^{-1}$ be two parameters. Let $w$ be a nonnegative function with support in $[Q, 2Q]$ satisfying $\| w \|_{\infty} \leq 1$ and $\sum_c w(c ) > 0$. For $r \in \mathbb{Q}$ write $I_{r}(\alpha)$ for the characteristic function of the interval $[r- \delta, r+\delta]$ and define
\begin{equation}\label{defLambda}
  \Lambda := \sum_c w(c ) \phi( c),  \quad \tilde{I}(\alpha) = \frac{1}{2 \delta \Lambda} \sum_c w(c ) \left.\sum_{d \bmod{c}}\right. ^{\ast} I_{d/c}(\alpha).
\end{equation}
Then $\tilde{I}(\alpha)$ is a good approximation to the characteristic function on $[0, 1]$ in the sense that
\begin{displaymath}
  \int_0^1 (1 - \tilde{I}(\alpha))^2 d\alpha \ll_{\varepsilon} \frac{Q^{2+\varepsilon}}{\delta \Lambda^2}
\end{displaymath}
for any $\varepsilon > 0$.
\end{lemma}

We apply this lemma with $Q = C$ and  $\delta = C^{-1}$. Let $w_0$ be a fixed smooth function with support in $[1, 2]$, and let 
\begin{equation}\label{defww}
w(c ) = \begin{cases}w_0(c/C), & \ell_1\ell_2 \mid c,\\
0, & \text{else.}\end{cases}
\end{equation}
  With the notation as in Lemma \ref{jutila}, we have
\begin{equation}\label{sizeLambda}
   \Lambda \asymp C^2(\ell_1\ell_2)^{-1}
\end{equation}
and 
\begin{displaymath}
\begin{split}
  \mathcal{D}_z(\ell_1, \ell_2, h, N, M)  & = \int_0^{1} \sum_{n, m} \lambda_1(m)\lambda_2(n) W\left(\frac{\ell_1n-h}{M}\right)V_z\left(\frac{\ell_2 m}{M}\right) e(\alpha(\ell_1 n - \ell_2 m - h)) d\alpha \\
  & =  \frac{1}{2\delta} \int_{-\delta}^{\delta}   \mathcal{D}_{z, \eta}(\ell_1, \ell_2, h, N, M ) d\eta + E,
\end{split}  
\end{displaymath}
where   
\begin{equation}\label{deta}
\begin{split}
& \mathcal{D}_{z, \eta}(\ell_1, \ell_2, h, N, M) \\
 &= \frac{1}{\Lambda} \sum_{\ell_1\ell_2 \mid c} w_0\left(\frac{c}{C}\right) \underset{d  \bmod{c}}{\left. \sum \right.^{\ast}} \sum_{n, m} \lambda_1(m)\lambda_2(n) e\left(\frac{d}{c}(\ell_1n - \ell_2m - h)\right)     W_{\eta  M}\left(\frac{\ell_1 n -h}{ M}\right)V_{z, \eta M}\left(\frac{\ell_2 m}{M}\right) 
\end{split} 
 \end{equation}
with $V_{z, \eta}(x) = V_z(x) e(-\eta x) =  V_1(x) e(x(zM/N - \eta))$, $W_{\eta}(x) = W(x) e(\eta x)$, and
\begin{displaymath}
\begin{split}
  E &=  \int_0^{1} \sum_{n, m} \lambda_1(m)\lambda_2(n) W\left(\frac{\ell_1n-h}{M}\right)V_z\left(\frac{\ell_2m}{M}\right) e(\alpha(\ell_1 n - \ell_2 m - h)) (1- \tilde{I}(\alpha)) d\alpha \\
  &  \ll \frac{C^{1+\varepsilon}}{\delta^{1/2} \Lambda}\Bigl(\sum_{m \ll M/\ell_2} |\lambda_1(m)|\Bigr)\Bigl(\sum_{n \ll N/\ell_1} |\lambda_2(n)|\Bigr) \ll \frac{C^{1+\varepsilon}}{\delta^{1/2} \Lambda}\frac{NM}{\ell_1\ell_2} \ll \frac{  NM}{C^{1/2-\varepsilon}} \ll C^{-2/5}
  \end{split}
\end{displaymath}
by the Cauchy-Schwarz inequality and \eqref{RS}. Since $|\eta| \leq C^{-1} = N^{-1000}$ is very small (in particular $\eta \ll M^{-1}$), the functions $V_{z, \eta M}$ and $W_{\eta M}$ have again nice properties, in particular $W_{\eta M}^{(j)} \ll 1$ and $V^{(j)}_{z, \eta M}  \ll C^{j\varepsilon}$, uniformly in $|z| \ll C^{\varepsilon}$, and $V_{z, \eta M}$, $W_{\eta M}$ have support in $[1, 2]$ resp.\ $[1/2, 3]$. 

\subsection{Voronoi summation}
\label{VoronoiSummationSubsection}
In the main term \eqref{deta}, we apply Lemma \ref{vor} to the $n, m$-sum, getting
\begin{equation}\label{firstvor}
  \sum_{m} \lambda_1(m) e\left(-\frac{dm}{c/\ell_2}\right) V_{z, \eta M}\left(\frac{\ell_2 m}{M}\right)  = \frac{M}{c}\sum_{m} \lambda_1(m) e\left(\frac{\bar{d}\ell_2m}{c}\right) \mathring{V}_{z, \eta M} \left(\frac{\ell_2 mM}{c^2}\right)
\end{equation}
and
\begin{equation}\label{secondvor}
\begin{split}
   \sum_{n} \lambda_2(n) & e\left(\frac{dn}{c/\ell_1}\right)   W_{\eta  M}\left(\frac{\ell_1 n -h}{ M}\right)\\
   & =  
   \frac{\ell_1}{c} \sum_n \lambda_2(n) e\left(-\frac{\bar{d}\ell_1n}{c}\right)  2\pi i^{\kappa_2} \int_0^{\infty} W_{\eta M} \left(\frac{\ell_1 x - h}{M}\right) J_{\kappa_2-1}\left(4\pi \frac{\sqrt{xn}}{c/\ell_1}\right) dx \\
   &
=    \frac{M}{c} \sum_{n} \lambda_2(n) e\left(-\frac{\bar{d}\ell_1n}{c}\right) W_{\eta M}^{\ast}\left(\frac{h\ell_1 n}{c^2}, \frac{M\ell_1n}{c^2}\right), 
   \end{split}
   \end{equation}
where
\begin{equation}\label{defWast}
  W^{\ast}_{\eta M}(z, w) = 2\pi i ^{\kappa_2} \int_0^{\infty} W_{\eta M}(y) J_{\kappa_2-1}(4\pi\sqrt{yw + z}) dy
\end{equation}
was analyzed in Lemma \ref{analyzeW}.  
Substituting \eqref{firstvor} and \eqref{secondvor} back into \eqref{deta} and using \eqref{decompW}, we obtain 
\begin{displaymath}
\begin{split}
 \mathcal{D}_{z, \eta}(\ell_1, \ell_2, h, N, M) = \frac{M^2}{\Lambda C}& \sum_{\ell_1\ell_2 \mid c} w_1\left(\frac{c}{C}\right)\frac{1}{c}  \sum_{n, m} \lambda_1(m)\lambda_2(n)  S(\ell_1 n - \ell_2 m, h, c)   \\
 & \times W_{\pm}\left( \frac{h\ell_1 n}{c^2}, \frac{M\ell_1n}{c^2}\right)e\left(\pm 2  \frac{\sqrt{h\ell_1 n}}{c}\right) \mathring{V}_{z, \eta M}\left(\frac{\ell_2 m}{c^2/M}\right) + {\rm O}(C^{-A})
 \end{split}
 \end{displaymath}
where
\[ w_1(x) = w_0(x)/x. \]
By \eqref{boundWpm} and the fact that $ \mathring{V}_{z, \eta M}$ is a Schwartz class function (cf.\ \eqref{byparts})  we can restrict the $n, m$-sums to
\begin{equation}\label{sizes}
  \ell_1 n \leq \mathcal{N}_0 := \frac{C^{2+\varepsilon}N}{M^2}, \quad \ell_2 m \leq \mathcal{M}_0 := \frac{C^{2+\varepsilon}}{M}
\end{equation}
  at the cost of a negligible error. It is convenient to restrict the $n$ and $m$-variable to  dyadic intervals. We use the notation $x \asymp X$ to mean $X \leq x \leq 2X$, and for $\mathcal{N} \leq \mathcal{N}_0$, $\mathcal{M} \leq \mathcal{M}_0$  we split  $\mathcal{D}_{\eta}(\ell_1, \ell_2, h, N, M)$ into subsums $n \asymp \mathcal{N}$, $m \asymp \mathcal{M}$.  It is also convenient to restrict to $  |\ell_1n - \ell_2m| \asymp \mathcal{K}$.  We split the arising subsums into three pieces ${\sum}_+$, ${\sum}_0$, and ${\sum}_-$, according to $\ell_1 n > \ell_2m$, $\ell_1 n = \ell_2 m$, and $\ell_1 n < \ell_2m$. Each of ${\sum}_+$, ${\sum}_0$ and ${\sum}_-$ depends on $\ell_1, \ell_2, h, N, M, \mathcal{N}, \mathcal{M}$ and $\mathcal{K}$. 
We first treat the terms with $\ell_1n= \ell_2m$. A trivial estimate shows that their contribution is at most
\begin{displaymath}
\begin{split}
{\sum}_0 & \ll  \frac{M^2}{\Lambda C^{1-\varepsilon}} \sum_{ C \leq c \leq 2C} \frac{(h, c)}{c}\sum_{\substack{\ell_1 n \asymp \mathcal{N}, \ell_2 m \asymp \mathcal{M}\\ \ell_1n = \ell_2m}}|\lambda_1(m)\lambda_2(n)| \\
&  \ll \frac{ M^2\tau(h)}{\Lambda C^{1-\varepsilon}}\Bigl( \sum_{m \ll \mathcal{M}} |\lambda_1(m)|^2\Bigr)^{1/2}\!\Bigl( \sum_{n \ll \mathcal{N}} |\lambda_2(n)|^2\Bigr)^{1/2}\!\!
\ll  \frac{  M^2 \tau(h) (\mathcal{N}_0\mathcal{M}_0)^{1/2}}{\Lambda C^{1-\varepsilon}} \ll \frac{C^{\varepsilon}(NM)^{1/2} \ell_1\ell_2}{C^{1-\varepsilon}} \ll C^{-1/2}. 
\end{split} 
\end{displaymath}

\subsection{Spectral analysis of \texorpdfstring{$\sum_+$}{sum-plus}} Next, we consider
\begin{equation}
\label{SumPlus}
{\sum}_{+} =   \frac{M^2}{\Lambda C}   \sum_{\substack{b > 0\\ |b| \asymp \mathcal{K}}} 
 \sum_{  \substack{\ell_1 n - \ell_2m = b\\ \ell_1 n \asymp \mathcal{N}, \ell_2m \asymp \mathcal{M}}} 
 \lambda_1(m)\lambda_2(n)  \sum_{\ell_1\ell_2 \mid c} \frac{S(b, h, c)}{c}  \Phi\left(4\pi \frac{\sqrt{|b|h}}{c}\right),
\end{equation}
where
\begin{displaymath}
\begin{split}
\Phi(x) = &w_1\left(\frac{4 \pi \sqrt{|b|h}}{xC}\right) W_{\pm}\left( \frac{\ell_1 n x^2}{(4\pi)^2|b| }, \frac{M\ell_1n x^2}{(4\pi)^2 |b|h}\right) e\left(\pm \frac{x\sqrt{\ell_1 n} }{2\pi\sqrt{|b|} }\right)   \mathring{V}_{z, \eta M}\left(\frac{\ell_2 m M  x^2}{(4\pi)^2|b|h}\right)
\end{split}
\end{displaymath}
and the inner sum over $c$ in \eqref{SumPlus} is ready for an application of the Kuznetsov trace formula (Lemma~\ref{kuznetsov}). (We are writing here $|b|$ instead of $b$ for notational consistency with next subsection.) The relevant Bessel transforms of $\Phi$ are   given by
\begin{displaymath}
\begin{split}
&  \tilde{\Phi}(t) = 2   \int_0^{\infty}  \Omega\Bigl(\frac{xC}{4\pi \sqrt{|b| h}}\Bigr)  \exp\Bigl(\pm i  x \sqrt{\ell_1 n/|b|}\Bigr)  \mathcal{J}^+_{2it}(x)   \frac{dx}{x},\\
  & \dot{\Phi}(k) = 4i^k    \int_0^{\infty}  \Omega\Bigl(\frac{x C}{4\pi \sqrt{|b| h}}\Bigr)  \exp\Bigl(\pm i  x\sqrt{\ell_1 n/|b|}\Bigr) J_{k-1}(x) \frac{dx}{x},
  \end{split}  
\end{displaymath} 
(cf.\ \eqref{IntegralTransforms} and \eqref{defJ+}), where
\begin{displaymath}
\begin{split}
&  \Omega(x ) 
:=  w_1\left(\frac{1}{x}\right) \mathring{V}_{z, \eta M} \left( \frac{x^2M\ell_2m }{C^2}\right ) W_{\pm} \left(\frac{x^2 h\ell_1n }{C^2}, \frac{x^2 M\ell_1n }{C^2}\right). 
\end{split}  
\end{displaymath}
Note that 
$\Omega$ has support on a fixed compact interval (inherited from $w_1$) and is almost non-oscillating, more precisely 
\[ \Omega^{(j)}(x ) \ll_{j} 
C^{j\varepsilon} \min\left(\Bigl(\frac{\mathcal{N}N}{C^2}\Bigr)^{-1/4}, \Bigl(\frac{\mathcal{N}N}{C^2}\Bigr)^{-\varepsilon}\right) \]
by \eqref{boundWpm}. By Lemma \ref{boundtrafo} with
\[ X = \frac{4\pi \sqrt{|b| h}}{C}, \quad Z = C^{\varepsilon}, \quad \alpha = \left(\frac{\ell_1 n}{|b|}\right)^{1/2} \geq 1, \]
the transforms $\tilde{\Phi}(t)$  and $\dot{\Phi}(k)$ are negligible   unless
\begin{equation}\label{defT}
\begin{split}
 & |t|  \ll  \mathcal{T}_+ :=    C^{\varepsilon} \left(1+\Bigl(\frac{\mathcal{K}N}{C^2}\Bigr)^{1/4} +  \Bigl(\frac{\mathcal{M}N}{C^2}\Bigr)^{1/2}\right),\\
   & k  \ll  \mathcal{T}_h :=    C^{\varepsilon} \left(1+\Bigl(\frac{\mathcal{K}N}{C^2}\Bigr)^{1/4}  \right). 
  \end{split} 
  \end{equation}
  
By the Kuznetsov formula (Lemma \ref{kuznetsov}), $\sum_{+}  = \mathcal{H}_{+}(h) + \mathcal{M}_+(h) + \mathcal{E}_+(h)  + {\rm O}(C^{-A})$ can be decomposed as the sum of three main terms, corresponding to the holomorphic, Maa{\ss} and Eisenstein spectrum, where
\begin{equation}\label{defmathcalH}
\begin{split} 
  \mathcal{H}_+(h) = \frac{M^2}{\Lambda C} \int_0^{\infty} \sum_{\substack{2 \leq k \leq \mathcal{T}_h\\ k \text{ even}}} & \sum_{f \in \mathcal{B}_k(\ell_1\ell_2)}  4i^k \Gamma(k) J_{k-1}(x)    \sqrt{h} \rho_f(h) 
\sum_{\substack{b > 0\\ |b| \asymp \mathcal{K}}} 
w_1\Bigl(\frac{4\pi \sqrt{|b|h}}{Cx}\Bigr) \sqrt{|b|} \rho_f(b ) \gamma_+(b, h, x) \frac{dx}{x}
\end{split}
\end{equation}
with
\begin{displaymath}
\begin{split}
 \gamma_+(b, h, x) = \sum_{\substack{\ell_1n - \ell_2m = b\\ \ell_1n \asymp \mathcal{N}, \ell_2m\asymp \mathcal{M}}} &
  \lambda_1(m)\lambda_2(n)     \mathring{V}_{z, \eta M}\left(\frac{x^2\ell_2m M}{(4\pi)^2|b|h}\right) W_{\pm}\left(\frac{x^2 \ell_1 n}{(4\pi)^2 |b|}, \frac{x^2 \ell_1 n M}{(4\pi)^2|b|h}\right) \vartheta_x\left(\frac{\ell_2m}{|b|}\right)\\
\end{split} 
\end{displaymath}
and
\begin{displaymath}
  \vartheta_x(y) = \exp\left(\pm i x \sqrt{1+y}\right)  v\left(\frac{y}{\mathcal{M}/\mathcal{K}}\right),
\end{displaymath}
where $v$ is  an artificially added, redundant smooth weight function of compact support $[1/4, 3]$ that is constantly 1 on $[1/2, 2]$.  We note that 
\begin{equation}\label{boundvarphi}
  y^j \frac{d^j}{dy^j}\vartheta_x(y) \ll \Bigl(1+\frac{x\mathcal{M}}{\sqrt{\mathcal{K}\mathcal{N}}  }\Bigr)^j. 
\end{equation}

Analogous expressions hold for $\mathcal{M}_{+}(h)$ and $\mathcal{E}_+(h)$:
\begin{equation}\label{defM+}
\begin{split} 
&  \mathcal{M}_+(h) = \frac{2  M^2}{\Lambda C} \int_0^{\infty}   \sum_{\substack{f \in \mathcal{B}(\ell_1\ell_2)\\ |t_f| \leq \mathcal{T}_+}}   \frac{\mathcal{J}^+_{2it_f}(x)  }{\cosh(\pi t_f)}  \sqrt{h} \rho_f(h) \sum_{\substack{b > 0\\ |b| \asymp \mathcal{K}}} 
w_1\Bigl(\frac{4\pi \sqrt{|b|h}}{Cx}\Bigr) \sqrt{|b|} \rho_f(b ) \gamma_+(b, h, x) \frac{dx}{x},\\
&  \mathcal{E}_+(h) = \frac{2  M^2}{\Lambda C} \int_0^{\infty}  \frac{1}{4\pi}  \sum_{\mathfrak{a}} \int_{-\mathcal{T}_+}^{\mathcal{T}_+}   \frac{\mathcal{J}^+_{2it}(x)  }{\cosh(\pi t)}  \sqrt{h} \rho_{\mathfrak{a}}(h, t) \sum_{\substack{b > 0\\|b| \asymp\mathcal{K}}} 
w_1\Bigl(\frac{4\pi \sqrt{|b|h}}{Cx}\Bigr) \sqrt{|b|} \rho_{\mathfrak{a}}(b, t ) dt\,\gamma_+(b, h, x) \frac{dx}{x}. 
 \end{split}
\end{equation}

\subsection{Spectral analysis of \texorpdfstring{$\sum_{-}$}{sum-minus}} The treatment of 
\begin{displaymath}
{\sum}_{-} =   \frac{M^2}{\Lambda C}  \sum_{\substack{b < 0 \\ |b| \asymp \mathcal{K}}} 
\sum_{\substack{  \ell_1 n - \ell_2m = b\\ \ell_1n \asymp \mathcal{N}, \ell_2m \asymp \mathcal{M}}} 
\lambda_1(m)\lambda_2(n)  \sum_{\ell_1\ell_2 \mid c} \frac{S(b, h, c)}{c}  \Phi\left(4\pi \frac{\sqrt{|b|h}}{c}\right)
\end{displaymath}
is  similar, but the details are slightly different. 
  Note that $b<0$ implies   \begin{equation}\label{newsize} \mathcal{N}  + \mathcal{K} \ll \mathcal{M} \leq \mathcal{M}_0. \end{equation} By Lemma \ref{boundtrafo}, the integral transform  $\check{\Phi}(t)$ is negligible  unless 
\begin{equation}\label{defT-}
 |t| \ll \mathcal{T}_- := C^{\varepsilon}\left(1+\Bigl(\frac{\mathcal{M} N}{C^2}\Bigr)^{1/2}\right). 
\end{equation}
Applying the opposite sign Kuznetsov formula, we  obtain $\sum_{-}  =   \mathcal{M}_-(h) + \mathcal{E}_-(h)  + {\rm O}(C^{-A})$ where (after a change of variables $x \mapsto 4\pi \sqrt{|b|} x$)
\begin{equation}\label{defM-}
\begin{split} 
&  \mathcal{M}_-(h) = \frac{2 M^2}{\Lambda C} \int_0^{\infty}   \sum_{\substack{f \in \mathcal{B}(\ell_1\ell_2)\\ |t_f| \leq \mathcal{T}_-}}    \frac{ \sqrt{h} \rho_f(h)}{\cosh(\pi t_f)}  \sum_{\substack{b < 0\\ |b| \asymp \mathcal{K}}} \mathcal{J}^-_{2it_f}\left(4\pi  \sqrt{|b|}x\right)    w_1\left(\frac{\sqrt{h}}{Cx}\right) \sqrt{|b|} \rho_f(b ) \gamma_-(b, h, x) \frac{dx}{x},\\
 \end{split}
\end{equation}
with
\begin{displaymath}
\begin{split}
 \gamma_-(b, h, x) = \sum_{\substack{\ell_1n - \ell_2m = b\\ \ell_1n \asymp \mathcal{N}, \ell_2m\asymp \mathcal{M}}} &
  \lambda_1(m)\lambda_2(n)     \mathring{V}_{z, \eta M}\left(\frac{x^2\ell_2m M}{h}\right) W_{\pm}\left(x^2 \ell_1 n, \frac{x^2 \ell_1 n M}{h}\right) e \left(\pm 2x \sqrt{\ell_1n}\right). \\
\end{split} 
\end{displaymath}
By Mellin inversion and \eqref{supnorm}, we have up to a negligible error 
\begin{equation}\label{defmathcalM}
\begin{split} 
  \mathcal{M}_-(h) = \frac{2 M^2}{\Lambda C} \int_0^{\infty}\int_{1 - iC^{\varepsilon}\mathcal{T}_-} ^{1 + iC^{\varepsilon}\mathcal{T}_-}   &  \sum_{\substack{f \in \mathcal{B}(\ell_1\ell_2)\\ |t_f| \leq \mathcal{T}_-}}  \widehat{\mathcal{J}}^-_{2it_f}(s)   \frac{ \sqrt{h} \rho_f(h)}{\cosh(\pi t_f)} w_1\left(\frac{\sqrt{h}}{Cx}\right) \\
&\times \sum_{\substack{b < 0\\ |b| \asymp \mathcal{K}}} \left(4\pi  \sqrt{|b|}x \right)^{-s}    
 \sqrt{|b|} \rho_f(b ) \gamma_-(b, h, x)\, \frac{ds}{2\pi i} \, \frac{dx}{x}.\\
 \end{split}
\end{equation}
An analogous formula holds for $\mathcal{E}_-(h) $. 

\subsection{Conclusion}

Before we sum over $h$ in the next section, we pause for a moment and summarize our discussion by stating the following decomposition.
\begin{prop}\label{conclusion} Let $\ell_1, \ell_2, h \in \mathbb{N}$, $M, N \geq 1$. Let $C = N^{1000}$, $\delta = 1/C$, assume \eqref{sizeh}, and define $\mathcal{N}_0$, $\mathcal{M}_0$ by \eqref{sizes}. Let $w_0$ be a fixed smooth function with support in $[1, 2]$, define $\Lambda$ as in \eqref{defLambda} using \eqref{defww}, and let $w_1(x) = w_0(x)/x$. Define $\mathcal{T}_h$, $\mathcal{T}_+$ and $\mathcal{T}_-$ as in \eqref{defT} and \eqref{defT-}. Assume that $V_{1, 2}$ satisfy \eqref{assonV}, let $W $ be as in the discussion after \eqref{assonV} and $V_2^{\dagger}$ be the Fourier transform of $V_2$, and define $\mathring{V}$ as in \eqref{hankel} and, for $z,\eta\in\mathbb{R}$, $V_{z, \eta M}(x) = V_1(x) e(-\eta M x)e(zxM/N)$. Let $W^{\ast}(z, w)$ be defined by \eqref{defWast} and correspondingly $W_{\pm}$ by \eqref{decompW}.  Finally   recall the special functions \eqref{defJ+}. With this notation   define 
$ \mathcal{H}_{+}(h) $, $ \mathcal{M}_{\pm}(h) $, $\mathcal{E}_{\pm}(h)$ as in \eqref{defmathcalH}, \eqref{defM+}, \eqref{defM-}. 

Then the smooth shifted convolution sum $\mathcal{D}(\ell_1, \ell_2, h, N, M ) $ defined in \eqref{defD} equals
\begin{equation}\label{specdecomp}
\begin{split}
   \mathcal{D}(\ell_1, \ell_2, h, N, M ) &  = \frac{1}{2\delta} \int_{-\delta}^{\delta}\int_{-C^{\varepsilon}}^{C^{\varepsilon}} V_2^{\dagger}(z) e\left(\frac{zh}{N}\right) \sum_{\mathcal{N} \leq \mathcal{N}_0}\underset{\mathcal{M}, \mathcal{K} \leq \mathcal{N}}{\sum_{\mathcal{M} \leq \mathcal{M}_0}  \sum_{\mathcal{K} \leq \mathcal{N}_0 }}\big( \mathcal{H}_{+}(h) + \mathcal{M}_{+}(h) + \mathcal{E}_{+}(h) \big)dz\,  d\eta \\
  & + \frac{1}{2\delta} \int_{-\delta}^{\delta} \int_{-C^{\varepsilon}}^{C^{\varepsilon}} V_2^{\dagger}(z) e\left(\frac{zh}{N}\right) \sum_{\mathcal{N} \leq \mathcal{M}_0}\sum_{\mathcal{M} \leq \mathcal{M}_0}  \sum_{\mathcal{K} \leq \mathcal{M}_0 }\big( \mathcal{M}_{-}(h) + \mathcal{E}_{-}(h) \big)\, dz\, d\eta+ {\rm O}(C^{-1/3})
 \end{split}  
\end{equation}
where $\mathcal{N}, \mathcal{M}, \mathcal{K}$ run over numbers $\geq 1$ of the form $\mathcal{N}_0 2^{-\nu}$ or $\mathcal{M}_02^{-\nu}$, $\nu \in \mathbb{N}$. 
\end{prop}

\section{Shifted convolution sums on average}
\label{ShiftedSumsAverageSection}

In this section, we use Proposition \ref{conclusion} to study averages of shifted convolution sums 
  $\mathcal{S}(\ell_1, \ell_2, d, N, M) = \sum_{r } \mathcal{D}(\ell_1, \ell_2, rd, N, M)$ 
over multiples of a positive
integer $d$, which were defined in \eqref{average}. 
  In particular, we will prove Proposition \ref{prop3}. Write
\[ \beta := {\rm lcm}(\ell_1, \ell_2, d). \] 
Our general assumption \eqref{sizeh} is still in place, 
so that $ \mathcal{D}(\ell_1, \ell_2, rd, N, M)$ vanishes unless $r \asymp N/d$.
We keep the notation from the previous section and import in particular the inequalities \eqref{sizeQ}, \eqref{sizeLambda},   \eqref{sizes}, \eqref{defT}, \eqref{defT-}. We start by considering 
\begin{displaymath}
  \sum_{r \asymp N/d} e\left(\frac{zrd}{N}\right) \mathcal{H}_+(rd) = \sum_{\substack{r_2 \ll N/d\\ r_2 \mid \beta^{\infty}}} \sum_{\substack{r_1 \asymp N/(dr_2)\\ (r_1, \beta)= 1}} e\left(\frac{zr_1r_2d}{N}\right)  \mathcal{H}_+(r_1r_2d)
\end{displaymath}
where $\mathcal{H}_+$ was defined in \eqref{defmathcalH}. We will sacrifice cancellation in the $x$-integral (in some typical ranges there is very little cancellation anyway) \eqref{defmathcalH} and just note that the range of integration is 
\begin{equation}\label{sizeX}
x \asymp X_+ := \frac{\sqrt{\mathcal{K}N}}{C}.
\end{equation}
\subsection{Separation of variables} We need to separate the variables $h=r_1r_2d, b, n, m$, scattered in the various smooth weight functions. We do this by brute force, expressing each weight function as an inverse Mellin transform. Since all of them are essentially non-oscillating (at least in typical ranges), this can be done with little loss.  With this in mind we write
\begin{displaymath}
\begin{split}
 & w_1\left(\frac{4\pi \sqrt{|b|r_1r_2d}}{Cx}\right) \mathring{V}_{z, \eta M}\left(\frac{x^2\ell_2m M}{(4\pi)^2|b|r_1r_2d} \right) W_{\pm}\left(\frac{x^2 \ell_1 n}{(4\pi)^2 |b|}, \frac{x^2 \ell_1 n M}{(4\pi)^2|b|r_1r_2d}\right)\vartheta_x\left(\frac{\ell_2m}{|b|}\right) \\
  &= \frac{1}{(2\pi i)^5} \int_{(0)} \int_{(\varepsilon)} \int_{(1/4-\varepsilon)}\int_{(\varepsilon)} \int_{(0)} \widehat{w}_1(s_1) \widehat{\mathring{V}}_{z, \eta M}(s_2 )\widehat{W}_{\pm}(s_3, s_4) \widehat{\vartheta_x}(s_5)\\
&    \times  \left(\frac{4\pi\sqrt{ |b|r_1r_2d}}{Cx}\right)^{-s_1}  \!\!\left(\frac{x^2\ell_2m M}{(4\pi)^2 |b|r_1r_2d}\right) ^{-s_2}\!\! \left(\frac{x^2 \ell_1 n}{(4\pi)^2|b|}\right)^{-s_3}\!\!\left( \frac{x^2 \ell_1 n M}{(4\pi)^2|b|r_1r_2d}\right)^{-s_4}\!\! \left(\frac{\ell_2m}{|b|}\right)^{-s_5} ds_5\,  ds_4\, ds_3 \,ds_2 \,ds_1.
  \end{split}
\end{displaymath}
The multiple integral is absolutely convergent, and we recall in particular Corollary \ref{cor8}.  The $s_1, \ldots, s_4$-integrals are rapidly converging and can be truncated at $|\Im s_j | \leq C^{\varepsilon}$ at the cost of a negligible error. By \eqref{boundvarphi} the $s_5$-integral can be truncated at
\begin{equation}\label{sizeS}
  |\Im s_5| \leq S:= C^{\varepsilon} \Bigl(1+ \frac{X_+\mathcal{M}}{\sqrt{\mathcal{K}\mathcal{N}}}\Bigr).  
\end{equation}
It is convenient to re-write the last line of the penultimate display as
\begin{displaymath}
\begin{split}
&  \left(\frac{x}{X_+}\right)^{s_1-2(s_2+s_3+s_4)} \left(\frac{|b|}{\mathcal{K}}\right)^{-\frac{s_1}{2}+s_2+s_3+s_4+s_5} \left(\frac{\ell_1n}{\mathcal{N}}\right)^{-s_3-s_4} \left(\frac{\ell_2m}{\mathcal{M}}\right)^{-s_2-s_5} (r_1r_2d)^{-\frac{s_1}{2}+s_2+s_4} \\
  & \times \frac{C^{s_1}\mathcal{K}^{-\frac{s_1}{2}+s_2+s_3+s_4+s_5}}{M^{s_2+s_4}(X_+/(4\pi))^{-s_1+2(s_2+s_3+s_4)}\mathcal{N}^{s_3+s_4}\mathcal{M}^{s_2+s_5}} \ll C^{\varepsilon} \frac{\mathcal{K}^{1/4}}{X_+^{1/2} \mathcal{N}^{1/4}} .
\end{split}  
\end{displaymath}

We substitute this back into \eqref{defmathcalH}, estimate the $x$- and $s_j$-integrals trivially and  finally   apply the Cauchy-Schwarz inequality to get
 \begin{equation}\label{finalboundH+}
  \sum_{\substack{r_1 \asymp N/(dr_2)\\ (r_1, \beta) = 1}} e\left(\frac{zr_1r_2d}{N}\right)\mathcal{H}_+(r_1r_2d)  \ll \frac{C^{\varepsilon}M^2}{\Lambda C} \frac{\mathcal{K}^{1/4}}{X_+^{1/2} \mathcal{N}^{1/4}} S\,  \left(\Xi_{1, +}^{\mathcal{H}} \,  \Xi_{2, +}^{\mathcal{H}}\right)^{1/2} 
 \end{equation}
 where
 \begin{displaymath} 
 \begin{split}
 & \Xi_{1, +}^{\mathcal{H}} = \max_{|u_4| \leq C^{\varepsilon}}  \sum_{\substack{2 \leq k \leq \mathcal{T}_h\\ k \text{ even}}} \Gamma(k) \sum_{f \in \mathcal{B}_k(\ell_1\ell_2)}    \Bigl|\sum_{\substack{r_1 \asymp N/(dr_2)\\ (r_1, \beta) = 1}} e\left(\frac{zr_1r_2d}{N}\right)r_1^{2\varepsilon + iu_4}\sqrt{r_1r_2d} \rho_f(r_1r_2d) \Bigr|^2,  \\
  &\Xi_{2, +}^{\mathcal{H}} =  \max_{\substack{|u_2|  \leq C^{\varepsilon}\\ |u_1|, |u_3| \leq S \\ x \asymp X_+}}     \sum_{\substack{2 \leq k \leq \mathcal{T}_h\\ k \text{ even}}}| J_{k-1}(x)|^2 \,\Gamma(k)   \sum_{f \in \mathcal{B}_k(\ell_1\ell_2)}     \Bigl|  \sum_{|b| \asymp \mathcal{K}}\sqrt{|b|} \rho_f(b )\gamma^{\ast}(b )\Bigr|^2,
 \end{split} 
\end{displaymath}  
 with
 \begin{displaymath}
\gamma^{\ast}(b )  =   \left(\frac{|b|}{\mathcal{K}} \right)^{1/4+\varepsilon + iu_3} \sum_{\substack{\ell_1n - \ell_2m = b\\ \ell_1n \asymp \mathcal{N}, \ell_2m \asymp \mathcal{M}}}   \left(\frac{\ell_1n}{\mathcal{N}} \right)^{-1/4  + iu_2}   \left(\frac{\ell_2m}{\mathcal{M}} \right)^{-\varepsilon + iu_1} \lambda_1(m)\lambda_2(n).     
  \end{displaymath}
The same analysis works \emph{mutatis mutandis} for the Eisenstein and Maa{\ss} spectrum, giving similar expressions $\Xi_{1/2, +}^{\mathcal{E}}$ and $\Xi_{1/2, +}^{\mathcal{M}}$. 

\subsection{The spectral large sieve} We proceed to estimate the various  $\Xi^{\star}_{j, +}$ for $j \in \{1, 2\}$, $\star \in \{\mathcal{H}, \mathcal{E}, \mathcal{M}\}$.  We have
\begin{displaymath}
  \sum_{b} |\gamma^{\ast}(b )|^2 \ll \int_0^1\Biggl| \sum_{\ell_2m\asymp \mathcal{M}} \lambda_1(m)   \left(\frac{\ell_2m}{\mathcal{M}}\right)^{-\varepsilon + iu_1 }e(\ell_2m\alpha) \Biggr|^2\Biggl| \sum_{\ell_1 n\asymp \mathcal{N}} \lambda_2(n)   \left(\frac{\ell_1n}{\mathcal{N}}\right)^{-1/4 +  iu_2}e(-\ell_1n\alpha) \Biggr|^2  d\alpha. 
\end{displaymath}
Since $u_2$ is small, we can successfully apply Wilton's bound \eqref{wilton} and partial summation. This does not work efficiently for the $m$-sum, but having estimated the $n$-sum by its sup-norm, we can open the square and use   \eqref{RS} to conclude that 
\begin{equation}\label{2norm}
   \sum_{b} |\gamma^{\ast}(b )|^2 \ll C^{\varepsilon} \frac{\mathcal{N}}{\ell_1}  \sum_{m \asymp \mathcal{M}/\ell_2} |\lambda_1(m)|^2 \ll C^{\varepsilon} \frac{\mathcal{N}\mathcal{M}}{\ell_1\ell_2},
\end{equation}
uniformly in $u_1, u_2, u_3$.  

In order to bound the Bessel function $J_{k-1}(x)$ in $\Xi_{2, +}^{\mathcal{H}}$, we recall the size of $x$ in \eqref{sizeX} and $k$ in \eqref{defT}. If $X_+ \geq 10^3 C^{\varepsilon}(1+X_+^{1/2})$, then $J_{k-1}(x) \ll x^{-1/2}$ by \eqref{uniform}. The opposite assumption  $X_+ < 10^3  C^{\varepsilon}(1+X_+^{1/2})$ implies $X_+ \ll C^{2\varepsilon}$ and hence trivially $J_{k-1}(x) \ll 1 \ll C^{\varepsilon} x^{-1/2}$. 
By  the large sieve inequality (Lemma \ref{largesieve}) and \eqref{2norm} we obtain
\begin{equation}\label{boundXi}
  \Xi_{2, +}^{\mathcal{H}} \ll \frac{C^{\varepsilon}}{X_+} \left(\mathcal{T}_h^2 + \frac{\mathcal{K}}{\ell_1\ell_2}\right)  \frac{\mathcal{N}\mathcal{M}}{\ell_1\ell_2}.
\end{equation}
Similarly one shows
\begin{equation}\label{boundXi'}
   | \Xi_{2, +}^{\mathcal{E}} |+| \Xi_{2, +}^{\mathcal{M}} |  \ll  
  \frac{C^{\varepsilon}}{X_+} \left(\mathcal{T}_+^2 + \frac{\mathcal{K}}{\ell_1\ell_2}\right)  \frac{\mathcal{N}\mathcal{M}}{\ell_1\ell_2}.
\end{equation}
\bigskip 

By \eqref{deligne} and  \eqref{multRama} we obtain
\begin{displaymath}
  \Xi_{1, +}^{\mathcal{H}} \ll \max_{|u_4| \leq C^{\varepsilon}}  C^{\varepsilon}  \sum_{\delta \mid \ell_1 \ell_2}   \sum_{\substack{2 \leq k \leq \mathcal{T}_h\\ k \text{ even}}} \Gamma(k) \sum_{f \in \mathcal{B}_k(\ell_1\ell_2)}  \Bigl|\sum_{\substack{r_1 \asymp N/(dr_2)\\ (r_1, \beta) = 1}} \alpha(r_1)\sqrt{r_1\delta} \rho_f(r_1\delta) \Bigr|^2. 
\end{displaymath}
where
\begin{equation}\label{defalpha}
\alpha(r_1) = \alpha_{r_2d, u_4}(r_1) = e\left(\frac{zr_1r_2d}{N}\right) r_1^{2\varepsilon + iu_4}.
\end{equation}
The large sieve (Lemma \ref{largesieve}) yields
\begin{equation}\label{simpleXi1}
  \Xi_{1, +}^{\mathcal{H}} \ll C^{\varepsilon} \sum_{\delta \mid \ell_1\ell_2} \left(\mathcal{T}_h^2 + \frac{N\delta}{dr_2\ell_1\ell_2}\right)\frac{N}{dr_2} \ll C^{\varepsilon}   \left(\mathcal{T}_h^2 + \frac{N }{dr_2 }\right)\frac{N}{dr_2} . 
\end{equation}
(We could be more careful here with powers of $\ell_1\ell_2$, but this is not necessary.) Using  \eqref{multEis} we obtain analogously 
\begin{equation}\label{boundXi1'}
  \Xi_{1,+}^{\mathcal{E}} \ll C^{\varepsilon} \left(\mathcal{T}_+^2 + \frac{N}{dr_2}\right)\frac{N}{dr_2}. 
\end{equation}
Note that the upper bounds in \eqref{boundXi'} and \eqref{boundXi1'} majorize those in  \eqref{boundXi} and \eqref{simpleXi1}.  Finally we apply Theorem  \ref{avoid} to obtain
\begin{equation}\label{complicatedXi}
\begin{split}
   \Xi_{1, +}^{\mathcal{M}}  & =    \max_{|u_4| \leq C^{\varepsilon}}  \sum_{\substack{|t_f|  \leq \mathcal{T}_+\\ f \in \mathcal{B}(\ell_1\ell_2)  }} \frac{1}{\cosh(\pi t_f)}      \Bigl|\sum_{\substack{r_1 \asymp N/(dr_2)\\ (r_1, \beta) = 1}} \alpha(r_1)\sqrt{r_1r_2d} \rho_f(r_1r_2d) \Bigr|^2 \\   
   &   \ll C^{\varepsilon} (\ell_1\ell_2, r_2d) \left(\mathcal{T}_++\frac{(r_2d)^{1/2}}{(\ell_1\ell_2)^{1/2}}\right)\left(\mathcal{T}_+ + \frac{N}{dr_2 (\ell_1\ell_2)^{1/2}}\right) \frac{N}{dr_2}.
 \end{split}   
 \end{equation}
  Combining \eqref{boundXi} -- 
 \eqref{complicatedXi}, we  conclude the final bound
 \begin{equation}\label{Xifinal}
 \begin{split}
  \big(| \Xi_{1, +}^{\mathcal{H}}| & + | \Xi_{1, +}^{\mathcal{E}}| + | \Xi_{1, +}^{\mathcal{M}}|\big) \big(| \Xi_{2, +}^{\mathcal{H}}| + | \Xi_{2, +}^{\mathcal{E}}| + | \Xi_{2, +}^{\mathcal{M}}|\big) \\
  &\ll  \frac{C^{\varepsilon}}{X_+}\left(\Bigl(\mathcal{T}_++\frac{(r_2d)^{1/2}}{(\ell_1\ell_2)^{1/2}}\Bigr)  \Bigl(\mathcal{T}_+ + \frac{N}{dr_2 (\ell_1\ell_2)^{1/2}}\Bigr)+  \frac{N}{dr_2} \right) \frac{N}{dr_2} \left(\mathcal{T}_+^2 + \frac{\mathcal{K}}{\ell_1\ell_2}\right)   \frac{(\ell_1\ell_2, r_2d)\mathcal{N}\mathcal{M}}{\ell_1\ell_2}. 
  \end{split}
\end{equation}

\subsection{Conclusion of the plus-case} It is now a matter of   book-keeping.
Combining 
 \eqref{sizeLambda}, 
 \eqref{defT},  \eqref{sizeX}, \eqref{sizeS},  \eqref{finalboundH+} and  \eqref{Xifinal},   we obtain  
\begin{displaymath}
\begin{split}
 &\sum_{\substack{r_1 \asymp N/(dr_2)\\ (r_1, q) = 1}} e\left(\frac{zr_1r_2d}{N}\right) \big( |\mathcal{H}_+(r_1r_2d)| + |\mathcal{M}_+(r_1r_2d)| +  |\mathcal{E}_+(r_1r_2d)|\big)\\
 &  \ll C^{ \varepsilon} \frac{M^2 \mathcal{N}^{1/4}\mathcal{M}^{1/2} (\ell_1\ell_2(\ell_1\ell_2, r_2d))^{1/2}  }{C^2 (dr_2)^{1/2} \mathcal{K}^{1/4}} \left(1 + \frac{N^{1/2}\mathcal{M}}{C\mathcal{N}^{1/2}}\right)  \left(1 + \Bigl(\frac{\mathcal{K}N}{C^2}\Bigr)^{1/4} + \Bigl(\frac{\mathcal{M} N}{C^2}\Bigr)^{1/2}+\Bigl(\frac{\mathcal{K}}{\ell_1\ell_2}\Bigr)^{1/2}\right)\\
 & \times \left( \left(1 + \Bigl(\frac{\mathcal{K}N}{C^2}\Bigr)^{1/8} + \Bigl(\frac{\mathcal{M} N}{C^2}\Bigr)^{1/4}+ \frac{(r_2d)^{1/4}}{(\ell_1\ell_2)^{1/4}}\right)  \left(1 + \Bigl(\frac{\mathcal{K}N}{C^2}\Bigr)^{1/8} + \Bigl(\frac{\mathcal{M} N}{C^2}\Bigr)^{1/4}+  \frac{N^{1/2}}{(dr_2)^{1/2}(\ell_1\ell_2)^{1/4}} \right)  +  \frac{N^{1/2}}{(dr_2)^{1/2}}\right).
   \end{split}
\end{displaymath}
We multiply out the 136 terms, and and write each term as
\[ \mathcal{M}^{\alpha} \mathcal{K}^{\beta} \mathcal{N}^{\gamma} C^{\delta} \times \text{expression in } N, M, \ell_1, \ell_2, d, r_2. \]
At this point it is important to recall \eqref{sizeQ}, \eqref{sizes} and   the size conditions $\mathcal{M}, \mathcal{K} \leq \mathcal{N}$   in the summation condition of the first line of \eqref{specdecomp}. We conclude that all terms with
\[ 2\alpha + 2\gamma + 2\max(\beta, 0) < -\delta - 1/3 \]
are less than $C^{-1/4}$ and therefore negligible. This applies to all terms except those involving   the last term $(\mathcal{K}/\ell_1\ell_2)^{1/2}$ in the second parenthesis on the right hand side. Hence we obtain the bound 
\begin{displaymath}
\begin{split}
C^{ \varepsilon} & \frac{M^2  \mathcal{M}^{1/2} (\ell_1\ell_2, r_2d)^{1/2 } }  {C^2 (dr_2)^{1/2}}\left((\mathcal{N}\mathcal{K})^{1/4} + \frac{N^{1/2}\mathcal{M}\mathcal{K}^{1/4}}{C\mathcal{N}^{1/4}}\right)   \left[\left(1 + \Bigl(\frac{\mathcal{K}N}{C^2}\Bigr)^{1/8} + \Bigl(\frac{\mathcal{M} N}{C^2}\Bigr)^{1/4}+ \frac{(r_2d)^{1/4}}{(\ell_1\ell_2)^{1/4}}\right)\right.\\
&\times \left. \left(1 + \Bigl(\frac{\mathcal{K}N}{C^2}\Bigr)^{1/8} + \Bigl(\frac{\mathcal{M} N}{C^2}\Bigr)^{1/4}+ \frac{N^{1/2}}{(dr_2)^{1/2}(\ell_1\ell_2)^{1/4}} \right) + \frac{N^{1/2}}{(dr_2)^{1/2}} \right]+ C^{-1/4}. 
   \end{split}
\end{displaymath}
In the first parenthesis we   cancel $(\mathcal{K}/\mathcal{N})^{1/4} \leq 1$. Having done this, all terms are increasing in $\mathcal{K}, \mathcal{M}, \mathcal{N}$, and we insert \eqref{sizes}. This gives the final bound
 \begin{equation}\label{finalH+}
\begin{split}
& \sum_{\substack{r_1 \asymp N/(dr_2)\\ (r_1, q) = 1}}  e\left(\frac{zr_1r_2d}{N}\right) \left(|\mathcal{H}_+(r_1r_2d)| + |\mathcal{M}_+(r_1r_2d)| +   |\mathcal{E}_+(r_1r_2d)|\right) \\
 & \ll  C^{ \varepsilon} \frac{M^{3/2} (\ell_1\ell_2, r_2d)^{1/2 }     }  {C(dr_2)^{1/2}}  \frac{CN^{1/2}}{M}  \left(\left(\frac{N^{1/4}}{M^{1/4}} + \frac{(r_2d)^{1/4}}{(\ell_1\ell_2)^{1/4}}\right)\left(\frac{N^{1/4}}{M^{1/4}} + \frac{N^{1/2}}{(dr_2)^{1/2}(\ell_1\ell_2)^{1/4}}\right) + \frac{NM^{1/2}}{dr_2} \right)  \\
 & \ll C^{ \varepsilon} (\ell_1\ell_2, r_2d)^{1/2  } \left(\frac{N}{(dr_2)^{1/2}} + \frac{N^{5/4}M^{1/4} }{dr_2(\ell_1\ell_2)^{1/4}} + \frac{N^{3/4}M^{1/4} }{(dr_2\ell_1\ell_2)^{1/4}} + \frac{NM^{1/2} }{(dr_2)^{3/4}(\ell_1\ell_2)^{1/2}} + \frac{NM^{1/2}}{dr_2} \right)\\
 & \ll C^{ \varepsilon} (\ell_1\ell_2, d)^{1/2  } \left(\frac{N}{d^{1/2}} + \frac{N^{5/4}M^{1/4} }{d(\ell_1\ell_2)^{1/4}} + \frac{N^{3/4}M^{1/4} }{d^{1/4}} + \frac{NM^{1/2} }{d^{3/4}(\ell_1\ell_2)^{1/2}} + \frac{NM^{1/2}}{d} \right).
   \end{split}
\end{equation}
(Here, of course, the term $C^{-1/4}$ can be absorbed.)

\subsection{The minus-case} The treatment of $\mathcal{M}_-$ and $\mathcal{E}_-$ is similar in spirit, but  the details are slightly different and considerably less involved. In particular, we can afford to be somewhat lossy in our estimations. We recall from \eqref{defM-} that  the range of integration is 
\begin{equation}\label{sizeX-}
x \asymp X_- := \frac{\sqrt{N}}{C}
\end{equation}
which is quite different from the previous case. 
 We separate variables in
\begin{displaymath}
\begin{split}
 & 
  \mathring{V}_{\eta M}\left(\frac{x^2\ell_2m M}{r_1r_2d}\right) W_{\pm}\left(x^2 \ell_1 n, \frac{x^2 \ell_1 n M}{r_1r_2d}\right)  \\
  &= \frac{1}{(2\pi i)^3}   \int_{(\varepsilon)} \int_{(1/4-\varepsilon)}\int_{(\varepsilon)}  \widehat{\mathring{V}}_{z, \eta M}(s_1)\widehat{W}_{\pm}(s_2, s_3)     \left(\frac{x^2\ell_2m M}{r_1r_2d}\right)^{-s_1}  \left( x^2 \ell_1 n \right)^{-s_2}\left( \frac{x^2 \ell_1 n M}{r_1r_2d}\right)^{-s_3} ds_3 \,ds_2 \,ds_1
  \end{split}
\end{displaymath}
by Mellin inversion. All integrals are rapidly converging and can be truncated at $|\Im s_j | \leq C^{\varepsilon}$ at the cost of a negligible error. We substitute this back into \eqref{defmathcalM}, estimate the $x$-, $s$- and $s_j$-integrals trivially (using \eqref{supnorm}) and apply the Cauchy-Schwarz inequality getting
\begin{equation}\label{prelim1}
 \sum_{\substack{r_1 \asymp N/(dr_2)\\ (r_1, q) = 1}}e\left(\frac{zr_1r_2d}{N}\right) \mathcal{M}_-(r_1r_2d)    \ll C^{ \varepsilon} \frac{M^2\ell_1\ell_2}{C^3} \frac{1}{ X_-^{1/2}\mathcal{N}^{1/4}}  \frac{\mathcal{T}_-}{\mathcal{K}^{1/2} X_-} \Xi_{1, -}^{1/2} \Xi_{2, -}^{1/2}
\end{equation}
where
 \begin{displaymath} 
 \begin{split}
 & \Xi_{1, -} = \max_{\substack{|u_3| \leq C^{\varepsilon}\\ x \asymp X_-}}  \sum_{\substack{f \in \mathcal{B}(\ell_1\ell_2)\\ |t_f| \leq \mathcal{T}_- }} \frac{1}{\cosh(\pi t_f)}     \Bigl|\sum_{\substack{r_1 \asymp N/(dr_2)\\ (r_1, \beta) = 1}} \tilde{\alpha}(r_1) w_1\left(\frac{\sqrt{r_1r_2d}}{Cx}\right)\sqrt{r_1r_2d}\rho_f(r_1r_2d) \Bigr|^2,  \\
  &\Xi_{2, -} =  \max_{\substack{|u_1|, |u_2|  \leq C^{\varepsilon}\\   x \asymp X_-}}     \sum_{\substack{f \in \mathcal{B}(\ell_1\ell_2)\\ |t_f| \leq \mathcal{T}_- }} \frac{1}{\cosh(\pi t_f)}     \Bigl|  \sum_{|b| \asymp \mathcal{K}}\sqrt{|b|} \rho_f(b )\gamma^{\ast}(b )\Bigr|^2
 \end{split} 
\end{displaymath}  
with $\mathcal{T}_-$ as in \eqref{defT-}, 
\[ \tilde{\alpha}(r_1) = r_1^{2\varepsilon + iu_3} e\left(\frac{zr_1r_2d}{N}\right) \]
and
 \begin{displaymath}
\gamma^{\ast}(b )  =   \sum_{\substack{\ell_1n - \ell_2m = b\\ \ell_1n \asymp \mathcal{N}, \ell_2m \asymp \mathcal{M}}}   \left(\frac{\ell_1n}{\mathcal{N}} \right)^{-1/4  + iu_1}   \left(\frac{\ell_1m}{\mathcal{M}} \right)^{-\varepsilon + iu_2} \lambda_1(m)\lambda_2(n) e(\pm 2 x \sqrt{\ell_1n}).     
  \end{displaymath}
As in \eqref{2norm} we find 
\begin{displaymath}
  \sum_{b  } |\gamma^{\ast}(b )|^2 \ll C^{\varepsilon} \frac{\mathcal{N} \mathcal{M}}{\ell_1\ell_2},
\end{displaymath}  
uniformly in $x, u_1, u_2$, and hence by the large sieve
\begin{displaymath}
  \Xi_{2, -} \ll \left(\mathcal{T}^2_- + \frac{\mathcal{K}}{\ell_1\ell_2}\right) C^{\varepsilon} \frac{\mathcal{N} \mathcal{M}}{\ell_1\ell_2}.
\end{displaymath}
The estimation of $\Xi_{1, -}$ is similar to the preceding analysis, but simpler. Here we apply \eqref{simpler} (in a weak version without the denominator $\ell$) to obtain 
\begin{displaymath}
  \Xi_{1, -} \ll C^{\varepsilon} \left(\mathcal{T}_-^2 + \frac{N}{dr_2}\right) \frac{N}{dr_2 } (dr_2)^{2\theta}(\ell_1\ell_2, dr_2)^{1-2\theta} 
\end{displaymath}
For the treatment of   Eisenstein case we can directly apply \eqref{multEis} and the large sieve as in \eqref{simpleXi1} -- \eqref{boundXi1'} getting a slightly stronger bound.  Substituting back into \eqref{prelim1} and recalling \eqref{defT-} and \eqref{sizeX-}, we obtain
\begin{displaymath}
\begin{split}
    \sum_{\substack{r_1 \asymp N/(dr_2)\\ (r_1, \beta) = 1}} & e\left(\frac{zr_1r_2d}{N}\right)\big(| \mathcal{M}_-(r_1r_2d)| + |\mathcal{E}_-(r_1r_2d)|\big)  \ll  C^{\varepsilon} 
\frac{M^2 N^{1/4} (\ell_1\ell_2)^{1/2}\mathcal{M}^{1/2}\mathcal{N}^{1/4}(\ell_1\ell_2, dr_2)^{1/2-\theta} }{C^{5/2} (dr_2)^{1/2 -\theta}} \\
&\times \frac{1 + (\mathcal{M}N/C^2)^{1/2}}{(\mathcal{K}N/C^2)^{1/2}}  \left(1 + \Bigl(\frac{\mathcal{M}N}{C^2}\Bigr)^{1/2} + \Bigl(\frac{\mathcal{K}}{\ell_1\ell_2}\Bigr)^{1/2}\right) \left(1 + \Bigl(\frac{\mathcal{M}N}{C^2}\Bigr)^{1/2} + \Bigl(\frac{N}{dr_2}\Bigr)^{1/2}\right). 
\end{split}
\end{displaymath}  
As before, we use  \eqref{sizeQ} to argue that  in the penultimate parenthesis only the third term contributes non-negligibly. The resulting expression is increasing in $\mathcal{M}, \mathcal{N}, \mathcal{K}$ each of which are bounded by $\mathcal{M}_0$, see \eqref{newsize}. Now a straightforward calculation similar to the above  shows the bound
\begin{equation}\label{final-}
\begin{split}
&    \sum_{\substack{r_1 \asymp N/(dr_2)\\ (r_1, \beta) = 1}}  e\left(\frac{zr_1r_2d}{N}\right)\big( |\mathcal{M}_-(r_1r_2d)|  + |\mathcal{E}_-(r_1r_2d)|\big)\\
    & \ll   C^{\varepsilon} (dr_2)^{\theta} (\ell_1\ell_2, dr_2)^{1/2-\theta} \left(\frac{M^{1/4}N^{3/4}}{(dr_2)^{1/2}} + \frac{M^{3/4} N^{3/4}}{dr_2}\right) \ll  C^{\varepsilon} d^{\theta} (\ell_1\ell_2, d)^{1/2} \left(\frac{M^{1/4}N^{3/4}}{d^{1/2}} + \frac{M^{3/4} N^{3/4}}{d}\right)
    \end{split}
\end{equation}

\subsection{Conclusion} We sum   \eqref{finalH+} and \eqref{final-} over $r_2 \mid \beta^{\infty}$; by Rankin's trick it is easy to see that 
\[ \sum_{\substack{r \leq X\\ r \mid \beta^{\infty}}}1 \ll (X \beta)^{\varepsilon}. \]
Using $\theta \leq 1/4$ and $N \geq M$, we conclude 
\begin{equation}\label{almost}
\begin{split}
  \mathcal{S}(\ell_1, \ell_2, d, N, M) \ll   N^{\varepsilon} (\ell_1\ell_2, d)^{1/2  } & \left(\frac{N}{d^{1/2}} + \frac{N^{5/4}M^{1/4} }{d } + \frac{N^{3/4}M^{1/4} }{d^{1/4} } + \frac{NM^{1/2} }{d^{3/4} }  \right). 
  \end{split}
\end{equation}
We remove the factor $(\ell_1\ell_2, d)^{1/2}$  as follows. 
We decompose 
\[ \ell_1 = \ell_1'\tilde{\ell} \delta_1 \delta, \quad  \ell_2 = \ell_2'\tilde{\ell} \delta_2 \delta, \quad d = d' \delta_1\delta_2 \delta \]
where $\delta = (d, \ell_1, \ell_2)$, $\delta_1 = (d, \ell_1)/\delta$, $\delta_2 = (d, \ell_2)/\delta$, $\tilde{\ell} = (\ell_1, \ell_2)/\delta$. 
Using \eqref{hecke}, we find 
\begin{displaymath}
\begin{split}
 \mathcal{S}(\ell_1, \ell_2, d, N, M) & =  \sum_r \sum_{\ell_1 n - \ell_2 m = rd} \lambda_1(m) \lambda_2(n) V\left(\frac{\ell_1 n}{N}\right)V\left(\frac{\ell_2 m}{M}\right) \\
 &= \sum_r \sum_{\ell_1'n - \ell_2m = d'r} \lambda_1(\delta_1m)\lambda_2(\delta_2n)  V\left(\frac{\delta_2\ell_1 n}{N}\right)V\left(\frac{\delta_1\ell_2 m}{M}\right) \\
 & = \sum_{g \mid \delta_2} \sum_{h \mid  \delta_1  } \mu(g) \mu(h) \lambda_2\left( \frac{\delta_2}{g}\right) \lambda_1\left( \frac{\delta_1}{h}\right) \mathcal{S}\left(\ell_1'g, \ell_2'h, d', \frac{N}{\delta\delta_1\delta_2\tilde{\ell}}, \frac{M}{\delta\delta_1\delta_2\tilde{\ell}}\right).
\end{split}
\end{displaymath} 
Using only a trivial bound for the Hecke eigenvalues ($\lambda(n) \ll n^{1/2}$) and noting that $(\ell_1'\delta_2, \ell_2'\delta_1, d') = 1$,  an application of \eqref{almost} now completes the proof of Proposition \ref{prop3}.

\section{Weyl Differencing}\label{sec7}

The rest of the paper is devoted to the proof of Theorem \ref{klooster-short}.  We begin with the following differencing lemma.

\begin{lemma}
\label{DifferencingLemma}
Let the functions $b,b_{1i},b_{2i}:\mathbb{Z}\to\mathbb{C}$ ($1\leqslant i\leqslant I$), $r_2\in\mathbb{N}$, and $R_2\in\mathbb{R}$ be such that
\[ b(m)=\sum_{i=1}^Ib_{1i}(m)b_{2i}(m)\quad(m\in\mathbb{Z}) \]
as well as
\[ b_{2i}(m+ r_2)=b_{2i}(m),\quad |b_{2i}(m)|\leqslant R_2 \quad(m\in\mathbb{Z},\,\,1\leqslant i\leqslant I). \]
Further, assume that the support of each $b_{1i}$ is contained in $(A, A+M]$, and let $H\in\mathbb{N}$.  
Then
\begin{align*}
\Bigl|\sum_{m}b(m)\Bigr|^2&\ll
\left( H r_2R_2^2 + \frac{R_2^2H^2r^2_2}{M}\right)I\sum_{i=1}^I\sum_{A<m\leqslant A+M}|b_{1i}(m)|^2\\
&\qquad\qquad\qquad+H r_2R_2^2I\sum_{0<|h|\leqslant\frac{M}{H r_2}}\Bigl|\sum_{i=1}^I\sum_{m}b_{1i}(m+hH r_2)\overline{b_{1i}(m)}\Bigr|.
\end{align*}
\end{lemma}

\begin{proof}
Let initially $b:\mathbb{Z}\to\mathbb{C}$ be arbitrary. We have
\begin{align*}
&\sum_{A<m\leqslant A+M}\,\,\sum_{\substack{h\in\mathbb{Z}\\ A<m+hH r_2\leqslant A+M}}b(m+hH r_2)\\
&\qquad=\sum_{A<m\leqslant A+M}b(m)\cdot\#\big\{(m_1,h):A<m_1\leqslant A+M,\,m=m_1+hH r_2\big\}\\
&\qquad=\sum_{A<m\leqslant A+M}b(m)\cdot\#\big\{h\in\mathbb{Z}:A<m-hH r_2\leqslant A+M\big\}\\
&\qquad=\sum_{A<m\leqslant A+M}b(m)\left(\frac{M}{H r_2}+\text{O}(1)\right)
=\frac{M}{H r_2}\sum_{A<m\leqslant A+M}b(m)+\text{O}\Bigl(\sum_{A<m\leqslant A+M}|b(m)|\Bigr).
\end{align*}
Therefore,
\begin{equation}
\label{PlaceToSmooth}
\begin{aligned}
\frac{M^2}{H^2 r_2^2}\Bigl|\sum_{A<m\leqslant A+M}b(m)\Bigr|^2 & \ll \Bigl|\sum_{A<m\leqslant A+M}\,\,\sum_{\substack{ h\in\mathbb{Z}\\A<m+hH r_2\leqslant A+M}}b(m+hH r_2)\Bigr|^2+\Bigl(\sum_{A<m\leqslant A+M}|b(m)|\Bigr)^2\\
&\ll M\sum_{A<m\leqslant A+M}\Bigl|\sum_{\substack{h\in\mathbb{Z}\\A<m+hH r_2\leqslant A+M}}b(m+hH r_2)\Bigr|^2+M\sum_{A<m\leqslant A+M}|b(m)|^2.
\end{aligned}
\end{equation}

Let $b(m)$ be as in the statement of Lemma~\ref{DifferencingLemma}. Using the Cauchy-Schwarz inequality and applying \eqref{PlaceToSmooth} with $b_{[i]}(m)=b_{1i}(m)b_{2i}(m)$, we have that
\[ \frac{M^2}{H^2r_2^2I}\Bigl|\sum_{A<m\leqslant A+M}b(m)\Bigr|^2\ll M\sum_{i=1}^I\sum_{A<m\leqslant A+M}\Bigl|\sum_{\substack{h\in\mathbb{Z}\\ A<m+hHr_2\leqslant A+M}}b_{[i]}(m+hHr_2)\Bigr|^2+M\sum_{i=1}^I\sum_{A<m\leqslant A+M}|b_{[i]}(m)|^2. \]
Since each $b_{2i}$ is $r_2$-periodic and bounded by $R_2$, we have for every individual $i$, $m$ that
\begin{align*}
\Bigl|\sum_{\substack{h\in\mathbb{Z}\\ A<m+hH r_2\leqslant A+M}}b_{[i]}(m+hH r_2)\Bigr|^2
&=\Bigl|b_{2i}(m)\sum_{\substack{h\in\mathbb{Z}\\ A<m+hH r_2\leqslant A+M}}b_{1i}(m+hH r_2)\Bigr|^2\\
&\leqslant  R_2^2\Bigl| \sum_{\substack{h\in\mathbb{Z}\\A<m+hH r_2\leqslant A+M}}b_{1i}(m+hH r_2)\Bigr|^2.
\end{align*}
Substituting this estimate above, we obtain
\def\msk{\quad}
\begin{align*}
&\frac{M^2}{H^2 r_2^2I}\left|\sum_{A<m\leqslant A+M}b(m)\right|^2\\
&\msk  \ll MR_2^2\sum_{i=1}^I\sum_{A<m\leqslant A+M}\left|\sum_{h}b_{1i}(m+hH r_2)\right|^2+M\sum_{i=1}^I\sum_{m}|b_{1i}(m)b_{2i}(m)|^2\\
&\msk\ll MR_2^2\sum_{i=1}^I\sum_{A<m\leqslant A+M}\sum_{h}|b_{1i}(m+hH r_2)|^2\\
&\msk\qquad +MR_2^2\sum_{i=1}^I\sum_{A<m\leqslant A+M}\mathop{\sum\sum}_{h_1\neq h_2}b_{1i}(m+h_1H r_2)\overline{b_{1i}(m+h_2H r_2)} +MR_2^2\sum_{i=1}^I\sum_{m}|b_{1i}(m)|^2
\end{align*}
\begin{align*}
 &\msk\ll MR_2^2\sum_{i=1}^I\sum_{A<m\leqslant A+M}|b_{1i}(m)|^2\left(\frac{M}{H r_2}+\text{O}(1)\right)\\
&\msk\qquad +MR_2^2\sum_{0<|g|\leqslant\frac{M}{H r_2}}\left|\sum_{i=1}^I\sum_{A<m\leqslant A+M} \sum_{h}  b_{1i}\big(m+(h+g)H r_2\big)\overline{b_{1i}(m+hH r_2)}\right|\\
&\msk\ll MR_2^2\cdot\left(\frac{M}{H r_2}+1\right)\sum_{i=1}^I\sum_{m}|b_{1i}(m)|^2
 +M R_2^2\sum_{0<|g|\leqslant\frac{M}{H r_2}}\left|\sum_{i=1}^I\sum_{m}b_{1i}(m+gH r_2)\overline{b_{1i}(m)}\left(\frac{M}{H r_2}+\text{O}(1)\right)\right|\\
&\msk\ll \left(\frac{M^2R_2^2}{H r_2} + MR_2^2\right)\sum_{i=1}^I\sum_{m}|b_{1i}(m)|^2+\frac{M^2R_2^2}{H r_2} \sum_{0<|g|\leqslant\frac{M}{H r_2}}\left|\sum_{i=1}^I\sum_{m}b_{1i}(m+gH r_2)\overline{b_{1i}(m)}\right|,
\end{align*}
using again that $\#\big\{(m_1,h):A<m_1\leqslant A+M,\,m=m_1+hH r_2\big\}=M/H r_2+\text{O}(1)$ as well as the Cauchy-Schwarz inequality to estimate the error terms in the off-diagonal summands. Rearranging, we conclude the lemma. 
\end{proof}

The procedure used in the proof of Lemma~\ref{DifferencingLemma}, the ``$q$-analogue of Weyl differencing'', goes back at least to Postnikov~\cite{Postnikov} and Heath-Brown~\cite{HB}. Similar ideas are also prominent in \cite{Po}. The important point here is the generality in which the procedure applies: no particular structure (such as being a character, or an exponential of a rational function) is assumed for terms $b_{1i}$ and $b_{2i}$ beyond periodicity and a uniform bound for $b_{2i}$.

There are no conditions whatsoever on the coefficients $b_{1i}(m)$. In the applications we have in mind, however, the term $b_{1i}(m+gHr_2)\overline{b_{1i}(m)}$ will have a period that is a proper divisor of $r$. (This can happen for two reasons: either because $b_{1i}$ are already periodic modulo a proper divisor of $r$, or because we take $H$ to be a suitable divisor of $r$ that causes a shortening of the period for the particular sequence $b_{1i}$.) On the other hand, the length of the $m$-summation in the off-diagonal terms in the upper bound of Lemma~\ref{DifferencingLemma} is unchanged at $M$. In a typical situation, $M$ may be too short compared to the original modulus $r$ to expect any nontrivial bound (such as $M\asymp r^{1/2}$ or less with chaotically behaving summands $b(m)$), but its size may well be more favorable compared to the newly smaller modulus.

Finally, it will be important for our purposes that $b(m)$ is allowed to be a sum of finitely many terms $b_{1i}(m)b_{2i}(m)$ ($1\leqslant i\leqslant I$) to which differencing is applied separately although the $i$-sum in the off-diagonal contribution to the upper bound is kept inside the absolute values. The case $I=1$, on the other hand, already contains the full idea of differencing.

\bigskip


Incomplete exponential sums whose length exceeds the square-root of the modulus, can often be efficiently estimated by the process sometimes referred to as completion. This procedure, which for clarity we record separately as the following simple technical result, applies in great generality, see \cite[Lemma 12.1]{IK}.  For an $r_1$-periodic function $c : \mathbb{Z} \rightarrow \mathbb{C}$, let \[ \hat{c}(k) := \sum_{n=1}^{r_1} c(n) e\left(-\frac{nk}{r_1}\right) \] be its discrete Fourier transform. The important point is that $\hat{c}(k)$ are complete exponential sums. (The notation for discrete Fourier transform in this section and the Mellin transform in earlier sections will not lead to confusion.)

\begin{lemma}
\label{CompletionLemma}
Let $A\in\mathbb{Z}$, $r_1,M\in\mathbb{N}$, and let $c:\mathbb{Z}\to\mathbb{C}$ be such that
$ c(m+r_1)=c(m)$ for $ m\in\mathbb{Z})$.  
Then
\begin{align*}
\sum_{A<m\leqslant A+M}c(m)
\ll \sum_{|k|\leqslant r_1/2}|\hat{c}(k)|\min\left(\frac{M}{r_1}, \frac{1}{|k|}\right).
\end{align*}
\end{lemma}


Combining Lemmas \ref{DifferencingLemma} and \ref{CompletionLemma}, we have the following general result:  
\begin{theorem}
\label{CentralHBLemma}
Let $r,r_1,r_2\in\mathbb{N}$ be such that $r=r_1r_2$. Let the functions $b,b_{1i},b_{2i}:\mathbb{Z}\to\mathbb{C}$ ($1\leqslant i\leqslant I$), $R_1,R_2\in\mathbb{R}$ be such that
\[ b(m)=\sum_{i=1}^Ib_{1i}(m)b_{2i}(m) \quad (m\in\mathbb{Z}) \]
as well as
\[  \begin{aligned}
&b_{1i}(m+r_1)=b_{1i}(m),\quad |b_{1i}(m)|\leqslant R_1,\\
&b_{2i}(m+r_2)=b_{2i}(m),\quad |b_{2i}(m)|\leqslant R_2. 
\end{aligned} \quad (m\in\mathbb{Z},\,\,1\leqslant i\leqslant I). \]

Let $H\in\mathbb{N}$, and let, for every $h,k\in\mathbb{Z}$ and $1\leqslant i\leqslant I$,
\begin{equation}
\label{DefinitionHatB}
 \hat{B}_{1i,hH}(k)=\sum_{m\bmod r_1}b_{1i}(m+hHr_2)\overline{b_{1i}(m)}e\left(-\frac{km}{r_1}\right).
\end{equation}

Then, for every $A\in\mathbb{Z}$, $M\in\mathbb{N}$,
\begin{align*}
&\Bigl|\sum_{A<m\leqslant A+M}b(m)\Bigr|^2\ll (M+Hr_2)Hr_2(R_1R_2)^2I^2  +Hr_2R_2^2I\sum_{0<|h|\leqslant\frac{M}{Hr_2}} \sum_{|k|\leqslant r_1/2}\left|\sum_{i=1}^I\hat{B}_{1i,hH}(k)\right|\min\left(\frac{M}{r_1}, \frac{1}{|k|}\right). 
\end{align*}
\end{theorem}

\begin{proof}
The proof is immediate from Lemmas~\ref{DifferencingLemma} and \ref{CompletionLemma}. Specifically, we apply Lemma~\ref{DifferencingLemma} to
\[ b(m)\chi_{(A,A+M]}(m)=\sum_{i=1}^I\big(b_{1i}(m)\chi_{(A,A+M]}(m)\big)b_{2i}(m). \]
We estimate the resulting first, diagonal term trivially, while for off-diagonal terms we use Lemma~\ref{CompletionLemma} with the $r_1$-periodic function
\[ c(m)=\sum_{i=1}^Ib_{1i}(m+hHr_2)\overline{b_{1i}(m)}. \qedhere \]
\end{proof}

The role of the parameter $H$ in Theorem~\ref{CentralHBLemma} will become clear later.  Importantly in the applications such as the central application for our problem, the sum defining $\hat{B}_{1i,hH}(k)$ is a complete exponential sum modulo $r_1$. Note that the trivial bound is
\[ |\hat{B}_{1i,hH}(k)|\ll r_1R_1^2, \]
so the trivial bound on the right-hand side is $\ll (R_1R_2)^2I^2(MHr_2 +(Hr_2)^2
+Mr_1+M^2\log r_1)$. This is, for general $b_{1i}$, a step backwards from the trivial bound $\ll M^2(R_1R_2)^2I^2$ on the left-hand side.

For arithmetically defined functions $b_{1i}$, however, the complete sum defining $\hat{B}_{1i,hH}(k)$ inherits this arithmetic structure. It will often be the case that the sum $\hat{B}_{1i,hH}(k)$ can be multiplicatively split in a certain sense. For $r_1$ a prime, the remaining complete sum can be estimated using techniques of algebraic geometry. For $r_1$ a higher prime power, the sum can be treated by the method of $p$-adic stationary phase. We remark that completion followed by the method of $p$-adic stationary phase acts as the proper $p$-adic analogue of the $B$-process in the classical van der Corput's theory of exponential sums \cite{Mi}; see also \cite{BM} for an example involving Kloosterman sums. In either case, for $b_{1i}$ of algebro-geometric origin, we can often recover  square-root cancellation in $\hat{B}_{1i,hH}(k)$.

\section{Proof of Theorem \ref{klooster-short}} 

We now prepare for the proof of Theorem \ref{klooster-short}. We first make a small reduction to the case $q = r$ in the situation of Theorem \ref{klooster-short}. Indeed, suppose that Theorem \ref{klooster-short} is proved in this special case, and write $q = r r' q'$ where $r' \mid r^{\infty}$ and $(q', r) = 1$. Then by M\"obius inversion we have
\[ \sum_{\substack{A < m \leq A+M\\ (m, q) = 1}} S(m, n_1, r) S(m, n_2, r) = \sum_{f \mid q'} \mu(f) \sum_{\substack{A/f < m \leq (A+M)/f\\ (m, r) = 1}} S(m, fn_1, r) S(m, fn_2, r) \]
so that the general case follows from the special case.  Thus we are interested in the sequence $b(m)$ given by
\[ b(m)=\begin{cases} S(m,n_1,r)S(m,n_2,r), &(m,r)=1,\\ 0, &(m,r)>1, \end{cases} \]
for integers $n_1, n_2$ (not necessarily coprime to $r$). From now on, we implicitly assume that $(m,r)=1$. Moreover, the letter $q$ is now free, and we will use it (in a different meaning than in the rest of paper) with or without indices as prime powers occurring in the prime factorization of $r$.

\bigskip

Before we apply Theorem~\ref{CentralHBLemma} to this particular function $b(m)$,  we explain briefly some technical difficulties. Kloosterman sums enjoy twisted multiplicativity, but of course only for coprime moduli. In order to apply Theorem~\ref{CentralHBLemma}, we need to decompose $r = r_1r_2$ with $(r_1, r_2) = 1$ and $r_1$, $r_2$ in certain ranges.  However, if $r$ is highly squareful (for example, if $r$ is a pure prime power), such a decomposition may not be possible. In this case, however,  one can choose the parameter $H$ in Theorem~\ref{CentralHBLemma} to be a suitable divisor of $r_1$, which 
produces partly degenerate Kloosterman sums and reduces the period of the sequence $b_{1i}(m+hHr_2)\overline{b_{1i}(m)}$, so that correspondingly $\widehat{B}_{1i,hH}(k)$ vanishes often (see Lemmas \ref{primepowercase} and \ref{FinalEstimate}). In other words, the parameters $H$ and $r_2$, each in its own way, act to make the range of summation in the off-diagonal terms in the upper bound of Lemma~\ref{DifferencingLemma} more favorable compared to the period of the summands, but they apply separately, depending on the factorization of the modulus $r$.  The previous discussion motivates a different treatment of the squarefree and the squareful part of $r$ that we proceed to make precise now. We start with some notation.

\bigskip

Let $p > 2$ be a prime. For $\kappa \in \mathbb{N}$, we denote by $\textbf{M}_{p^{\kappa}}$  an arbitrary element of $p^{\kappa}\mathbb{Z}_p$, which may be different from line to line. This notation serves as a $p$-adic analogue of Landau's $\text{O}$-notation in Taylor expansions. For $s\geqslant 1 $ and $(A,p)=1$, let
\begin{equation}\label{gauss sum}
 \tau(A, p^s)=\begin{cases} 1, &2\mid s,\,\,p\text{ odd},\\ \big(\frac Ap\big), &2\nmid s,\,\,p\equiv 1\pmod 4,\\ \big(\frac Ap\big)i, &2\nmid s,\,\,p\equiv 3\pmod 4, 
 \end{cases} 
 \end{equation}
be the sign of the Gau{\ss} sum $\sum_{x\bmod p^s}e(Ax^2/p^s)=p^{s/2}\tau(A,p^s)$.

\bigskip


Next,  we collect facts and notations pertaining to square roots to prime power moduli, which arise in connection with the explicit evaluation of Kloosterman sums as in Lemma~\ref{KloostermanEvaluation}. While these square roots naturally arise in $p$-adic towers as in \cite{BM}, we keep our exposition elementary and only discuss square roots to a prime power modulus $p^{\kappa}$. This discussion applies separately at every odd prime $p$. For every $x\in(\mathbb{Z}/p^{\kappa}\mathbb{Z})^{\times}{}^2$, there are exactly two solutions $u\in(\mathbb{Z}/p^{\kappa}\mathbb{Z})^{\times}$ of the congruence $u^2\equiv x\pmod{p^{\kappa}}$. Fix once and for all a choice function $s:(\mathbb{Z}/p\mathbb{Z})^{\times}{}^2\to(\mathbb{Z}/p\mathbb{Z})^{\times}$ such that, for every $r\in(\mathbb{Z}/p\mathbb{Z})^{\times}$, the class $s(r )\in(\mathbb{Z}/p\mathbb{Z})^{\times}$ satisfies $s(r )^2\equiv r\pmod p$. Then, for every $x\in(\mathbb{Z}/p^{\kappa}\mathbb{Z})^{\times}$, we denote by $u_{1/2}^{[\kappa]}(x)$ the unique class $u\in(\mathbb{Z}/p^{\kappa}\mathbb{Z})^{\times}$ such that $u^2\equiv x\pmod{p^{\kappa}}$ and $u\in s(x+p\mathbb{Z})$. This gives way to a unique function $u_{1/2}^{[\kappa]}:(\mathbb{Z}/p^{\kappa}\mathbb{Z})^{\times}{}^2\to(\mathbb{Z}/p^{\kappa}\mathbb{Z})^{\times}$, which we may think of as a branch of the square-root. (Each choice of $s$ gives rise to a different branch of the square-root, but we will never need to consider other possible choices.) The values of $u_{1/2}^{[\kappa]}$ are compatible across different values of $\kappa$, in the sense that $u_{1/2}^{[\kappa_1]}(x)\equiv u_{1/2}^{[\kappa_2]}(x)\pmod{p^{\min(\kappa_1,\kappa_2)}}$, and hence we simply write $x_{1/2}$ for $u_{1/2}^{[\kappa]}(x)$ with a sufficiently high value of $\kappa$ (for example, the highest power of $p$ occurring as a modulus in the exponential sum of interest).

\bigskip
 
The following (essentially well-known) lemma appears for instance in \cite[Lemma 6]{BM}. 

\begin{lemma}
\label{KloostermanEvaluation}
Let $p > 2$ be a prime,   let $s\geqslant 2$, and let $S(m,n;p^s)$ be the usual Kloosterman sum. Let $(m, p) = 1$ and $p^{\nu} \exmid n$. 
Then $S(m, n; p^s) = 0$ unless
\[ \nu=0,\,\, mn \in(\mathbb{Z}/p\mathbb{Z})^{\times}{}^2, \]
in which case it equals 
\[ S(m,n;p^s)=p^{s/2}\sum_{\pm}\tau\big(\pm (mn )_{1/2},p^{ s}\big)e\left(\pm\frac{2(mn )_{1/2}}{p^{s}}\right). \]
 \end{lemma}

Suppose that $r=r_1r_2$ with
\[ (r_1,6r_2)=1, \]
and let
\begin{equation}
\label{Factorizationr1}
r_1=\prod_{j=1}^{J}q_j,\quad q_j=p_j^{s_j}, \quad p_j > 3,
\end{equation}
be the canonical factorization of $r_1$ into prime powers. We  write
\[ Q_j=r_1/q_j,\qquad Q_j\bar{Q}_j\equiv 1\pmod{q_j}. \]
We  denote all moduli $q_j=p_j^{s_j}$ with $s_j\geqslant 2$ as $q_1,\dots,q_{\rho}$, and for later purposes, we fix a divisor $r_1^{\sharp}$ of $r_1$ which will the product of some of the moduli $q_j$, $1 \leq j \leq \rho$. By rearranging, we may write
\begin{equation}
\label{Defr1sharp}
r_1^{\sharp}=\prod_{j=1}^{\varrho}q_j.
\end{equation}
for some  $\varrho\leqslant\rho$. By definition, $r_1^{\sharp}$ is squareful.  For the moment, we do not impose any further condition on $r_1^{\sharp}$.  (The final choice will satisfy $r_1^{\sharp}=(r_1,H^{\infty})$, but the need for this choice will only become apparent later.) 

Using the twisted multiplicativity of Kloosterman sums (which follows from the Chinese remainder theorem), we have that $b(m)=b_1(m)b_2(m)$ with
\begin{align*}
&b_1(m)=S(\bar{r}_2m,\bar{r}_2n_1,r_1)S(\bar{r}_2m,\bar{r}_2n_2,r_1)  = \prod_{j=1}^{J}S\big(m,\bar{Q}_j^2\bar{r}_2^2n_1,q_j\big)S\big(m,\bar{Q}_j^2\bar{r}_2^2n_2,q_j\big), \\
 & b_2(m)=S(\bar{r}_1m,\bar{r}_1n_1,r_2)S(\bar{r}_1m,\bar{r}_1n_2,r_2).
\end{align*}
Keeping in mind that $(\bar{r}_1m,r_2)=1$, we have according to Weil's bound
\begin{equation}\label{r2}
 |b_2(m)|\leqslant R_2:=d(r_2)^2r_2. 
 \end{equation}
Since $(2m, r_1)=1$, we see from Lemma \ref{KloostermanEvaluation} that $b_1(m)$ vanishes unless
\[ m n_1,m n_2, n_1 n_2\in(\mathbb{Z}/p_j\mathbb{Z})^{\times}{}^2\quad(1\leqslant j\leqslant\varrho), \]
in which case $b_1(m)$ splits as a sum of $4^{\varrho}\ll r^{\varepsilon}$ terms, which we naturally index by $\bseps\in\{\pm 1\}^{2\times\varrho}=(\epsilon_{ij})_{i=1}^2\,{}_{j=1}^{\varrho}$ as follows:
\[ b_1(m)=\sum_{\bseps\in\{\pm 1\}^{2\times\varrho}}b_1^{\bseps}(m),\]
\[ b_1^{\bseps}(m)=\prod_{j=1}^{\varrho}S^{\epsilon_{1j}}\big(m,\bar{Q}_j^2\bar{r}_2^2n_1;q_j\big)S^{\epsilon_{2j}}\big(m,\bar{Q}_j^2\bar{r}_2^2n_2;q_j\big)\prod_{j=\varrho+1}^{J}S\big(m,\bar{Q}_j^2\bar{r}_2^2n_1;q_j\big)S\big(m,\bar{Q}_j^2\bar{r}_2^2n_2;q_j\big), \]
\begin{displaymath}
   S^{\epsilon}(m,n;p^s)=p^{s/2}\tau\big(\epsilon (mn)_{1/2},p^{s}\big)e\left(\frac{2\epsilon  (mn)_{1/2}}{p^s}\right) \quad (s\geqslant 2,\,\,mn\in(\mathbb{Z}/p\mathbb{Z})^{\times}{}^2). 
   \end{displaymath}
Note that the Kloosterman sums $S(m,n,q)$ are real-valued, but the terms $S^{\epsilon}(m, n, p^s)$, in general, are not. 
   
 We are now ready to apply Theorem~\ref{CentralHBLemma}, with
\[ r=r_1r_2,\quad b(m)=\sum_{\bseps\in\{\pm 1\}^{2\times\varrho}}b_1^{\bseps}(m)b_2(m), \]
$R_2$ as in \eqref{r2}, and
\[ R_1=\max_{\bseps\in\{\pm 1\}^{2\times\varrho}}\big|b_1^{\bseps}(m)\big| \ll d(r_1)^2r_1. \]
We can conclude that
\begin{equation}\label{basicsum}
\begin{split}
&\Bigl|\sum_{\substack{A<m\leqslant A+M\\ (m,r)=1}}S(m,n_1,r)S(m,n_2,r)\Bigr|^2
\ll r^{\varepsilon}(M + Hr_2) Hr_2r^2 \\
&\qquad\qquad\qquad +r^{\varepsilon}Hr_2^3 \sum_{0<|h|\leqslant\frac{M}{Hr_2}} \sum_{|k|\leqslant\frac{r_1}2}\bigg|\sum_{\bseps\in\{\pm 1\}^{2\times\varrho}}\hat{B}_{1,hH}^{\bseps}(r_1, r_2, k)\bigg|\min\left(\frac{M}{r_1}, \frac{1}{|k|}\right),
\end{split}
\end{equation}
where, as in \eqref{DefinitionHatB}, the terms $\hat{B}^{\bseps}_{1,hH}(r_1,r_2,k)$ are given by complete sums
\[ \hat{B}^{\bseps}_{1,hH}(r_1, r_2, k) =\sumstar_{m\bmod r_1}b_1^{\bseps}\big(m+hHr_2\big)\overline{b_1^{\bseps}(m)}e\left(-\frac{km}{r_1}\right). \]
The sum of these terms $\hat{B}^{\bseps}_{1,hH}(r_1,r_2,k)$ is the central object of our estimation. We introduce some additional notation that allows us to state our results succinctly.

For $q=p^s$, $s\geqslant 2$, $n_1n_2\in(\mathbb{Z}/q\mathbb{Z})^{\times}{}^2$, and $\bseps=(\epsilon_1,\epsilon_2,\epsilon_3,\epsilon_4)\in\{\pm 1\}^4$, denote
\begin{equation}\label{defSigma1}
\begin{split}
\Sigma^{\bseps}(n_1,n_2,a,k;p^s)=\sumstar_{\substack{m\bmod p^s\\ m,m+a\in n_1(\mathbb{Z}/p\mathbb{Z})^{\times}{}^2}}S^{\epsilon_1}(m&+a,n_1,p^s) \overline{S^{\epsilon_2}(m,n_1,p^s)}\\
&\smash[t]{S^{\epsilon_3}(m+a,n_2,p^s) \overline{S^{\epsilon_4}(m,n_2,p^s)}e\left(-\frac{km}{p^s}\right)}.
\end{split}
\end{equation}
For a general (prime or a) prime power $q$, we let
\begin{equation}\label{defSigma}
  \Sigma(n_1,n_2,a,k;q)=\sum_{\substack{m\bmod q\\ (m(m+a), q) = 1}}S(m+a,n_1,q)S(m+a,n_2,q)S(m,n_1,q)S(m,n_2,q)e\left(-\frac{km}q\right).
  \end{equation}
Denote
\[ A_0=\{\pm 1\}^4,\quad A^{\sharp}=\{\bseps\in A_0:\epsilon_1=\epsilon_2,\,\,\epsilon_3=\epsilon_4\},\]
and, for an odd prime $q=p^s$ with $s\geqslant 2$,
\[ \Sigma^{\sharp}(n_1,n_2,a,k;q)=\sum_{\bseps\in A^{\sharp}}\Sigma^{\bseps}(n_1,n_2,a,k;q),\quad \Sigma(n_1,n_2,a,k;q)=\sum_{\bseps\in A_0}\Sigma^{\bseps}(n_1,n_2,a,k;q). \]

We may rewrite the innermost sum in \eqref{basicsum} as
\begin{equation}\label{decompBhat}
\begin{split}
\hat{B}_{1,hH}[r_1,r_2,k]&:=\sum_{\bseps\in\{\pm 1\}^{2\times\varrho}}\hat{B}^{\bseps}_{1,hH}(r_1, r_2, k)\\
&=\prod_{j=1}^{\varrho}\Sigma^{\sharp}\big(\bar{Q}_j^2\bar{r}_2^2n_1,\bar{Q}_j^2\bar{r}_2^2n_2,hHr_2,\bar{Q}_jk;q_j\big)\prod_{j=\varrho+1}^{J}\Sigma\big(\bar{Q}_j^2\bar{r}_2^2n_1,\bar{Q}_j^2\bar{r}_2^2n_2,hHr_2,\bar{Q}_jk;q_j\big).
\end{split}
 \end{equation}
We see that it suffices to obtain upper bounds for the complete sums $\Sigma(n_1,n_2,a,k;q)$ and $\Sigma^{\sharp}(n_1,n_2,a,k;p^s)$ as above. These bounds are provided in the following result whose proof we postpone to the next section.

We need just a bit more notation. For an integer $n$ let $\text{rad}(n)$ denote its squarefree kernel and $\omega(n)$ the number of its prime factors. For a finite set $T$ and $q\in\mathbb{N}$, we denote
\begin{equation}\label{notation}
 (T,q)=\mathop{\mathrm{lcm}}\{(t,q):t\in T\}, 
 \end{equation}
and $n+T=\{n+t:t\in T\}$ as usual. Finally for a positive integer  $n $ we denote by $n_{\square}$ the largest  integer whose square divides $n$. (In particular, for a prime power $p^s$ we have $(p^s)_{\square} = p^{[s/2]}$.)

Then, collecting the results of Lemma~\ref{primecase}, the decomposition~\eqref{Decomposition}, the reduction formula~\eqref{ReductionFormula}, and Lemmata~\ref{primepowercase} and \ref{FinalEstimate}, we obtain the following result.

\begin{lemma}
\label{FinalEstimateSigma}
Let $q=p^s$, where $p>3$ is a prime and $s\geqslant 1$, and let $n_1,n_2,a,k\in\mathbb{Z}$.
\begin{enumerate}
\item If $p\mid a$ and $s\geqslant 2$, then
\[ \Sigma^{\sharp}(n_1,n_2,a,k;q)\ll q^{5/2}\sum_{\substack{\delta\in\{1,(q,n_1-n_2)\}\\ \delta\mid k,\,\,(\delta a,q/p)\mid k}}(q,\delta a,k)^{1/2}. \]
\item If $p\mid a$, or if $s=1$, then
\[ \Sigma(n_1,n_2,a,k;q)\ll q^{5/2}\sum_{\substack{\delta\in\{1,(q,n_1-n_2)\}\\ \delta\mid k}}\sum_{\substack{\delta'\in\{1,(q,a)\}\\ (\delta\delta',q/p)\mid k}}(q,\delta\delta',k)^{1/2}. \]
\item There exists a finite set $T\subset\mathbb{Z}\setminus p\mathbb{Z}$, of absolutely bounded size, depending on $q$, $n_1$, and $n_2$ only, such that, for every $k\in\mathbb{Z}$ and every $p\nmid a$,
\[ \Sigma(n_1,n_2,a,k;q)\ll q^{5/2}\sum_{\substack{\delta\in\{1,(q,n_1-n_2)\}\\\delta\mid k}}\delta^{1/2}\left(\Big(\frac{k}{\delta}\Big)^2a-T,\Big(\frac{q}{\delta}\Big)_{\square}\right)^{1/2},\]
and the second factor in the sum may be omitted whenever $q/\delta$ is cube-free.
\end{enumerate}
\end{lemma}

\begin{proof}
We show how Lemma~\ref{FinalEstimateSigma} follows from the results of Section~\ref{EstimationCompleteSumsSection}.

If $s=1$, then Lemma~\ref{primecase} shows that $\Sigma(n_1,n_2,a,k;q)\ll q^{5/2}$, except if $q\mid(a(n_1-n_2))$ and $q\mid k$, in which case the upper bound obtained is $\Sigma(n_1,n_2,a,k;q)\ll q^3$. This estimate is absorbed by the upper bound in (2), specifically by the term corresponding to $\delta=\delta'=1$ in the former and by the term corresponding to $\delta=(q,n_1-n_2)$, $\delta'=(q,a)$ in the latter case. Moreover, if $s=1$ and $q\nmid a$, the estimate of Lemma~\ref{primecase} is also allowable in (3) with $\delta=(q,n_1-n_2,k)$.

Consider now the case $s\geqslant 2$. According to the decomposition~\eqref{Decomposition}, the sum $\Sigma^A(n_1,n_2,a,k;q)$ (with $A\in\{A_0,A^{\sharp}\}$ and $A=A^{\sharp}$ only if $p\mid a$) can be written as a finite linear combination 
\[ \Sigma=q^2\sum_{\bseps\in A}\tau^{[\bseps]}\Sigma[A^{[\bseps]}(n_1,n_2),B^{[\bseps]}(n_1,n_2),a,k;q], \]
with the parameters $A=A^{[\bseps]}(n_1,n_2)$ and $B=B^{[\bseps]}(n_1,n_2)$ given explicitly as in \eqref{DefinitionAepsBeps} and the sum $\Sigma[A,B,a,k;q]$ defined in \eqref{DefinitionSigmaAB}. The contribution of terms with $p\nmid A$ or $p\nmid B$ can be estimated by Lemmas~\ref{primepowercase} and \ref{FinalEstimate} and absorbed in the terms corresponding to $\delta=1$ in (1)--(3) above as follows:
\begin{itemize}
\item If $p\nmid a$, we apply Lemma~\ref{FinalEstimate}~(1) to obtain (3) (expanding $T$ to account for all choices of $A$ and $B$).
\item If $p\mid a$, we estimate the terms with $A\equiv B\pmod p$ (and then, as will be seen from \eqref{DefinitionAepsBeps}, $A=B$) and $A\not\equiv B\pmod p$ separately. For the terms in which $A=B$, which are the only ones that arise in the estimation of $\Sigma^{\sharp}(n_1,n_2,a,k;q)$, we apply Lemma~\ref{primepowercase} and obtain (1) and the terms in (2) with $\delta'=(q,a)$. For the terms in which $A\not\equiv B\pmod p$, we apply Lemma~\ref{FinalEstimate}~(2) and obtain the terms in (2) with $\delta'=1$.
\end{itemize}

Terms with $p\mid A$ and $p\mid B$ will be seen to appear if and only if $p\mid (q,n_1-n_2)$, in which case, denoting $\delta=(q,n_1-n_2)$, we have $\delta\exmid A,B$, and $\Sigma[A,B,a,k;q]=0$ unless $\delta\mid k$. If $\delta=q$, then all of (1)--(3) hold for the trivial reason that all upper bounds are at least $q^3$. Otherwise, by applying the reduction formula \eqref{ReductionFormula}, we have that
\[ \Sigma[A,B,a,k;q]=\delta\cdot\Sigma[A/\delta,B/\delta,a,k/\delta;q/\delta], \]
where $p\nmid (A/\delta)$ and $p\nmid (B/\delta)$. The remaining sum is treated as above and is seen to be bounded by the terms corresponding to $\delta=(q,n_1-n_2)$ in (1)--(3).
\end{proof}

Applying Lemma~\ref{FinalEstimateSigma} to the individual factors in \eqref{decompBhat}, and with a quick application of the Chinese Remainder Theorem, we obtain the following crucial estimate.

\begin{prop}\label{expsums} Let $r = r_1r_2$ with $(r_1, 6r_2) = 1$, and let $r_1^{\sharp}$ be a squareful divisor of $r_1$, with factorizations of $r_1$ and $r_1^{\sharp}$ as in \eqref{Factorizationr1} and \eqref{Defr1sharp}. Let $h$ and $H$ be non-zero integers with $r_1^{\sharp}\mid(hH)^{\infty}$, and let $k \in\mathbb{Z}$. Write 
\[ \tilde{r}_1:=\prod_{\substack{q_j\exmid r_1,\,\,\mu(q_j)=0,\\(q_j,hH)=1}}q_j, \quad r_1=r^{\flat}_1\tilde{r}_1; \]
in particular, $r_1^{\sharp}\mid r_1^{\flat}$. Then, there exists for every $\tilde{\delta}\mid\tilde{r}_1$ a set $T_{\tilde{\delta}}$, of cardinality $\textnormal{O}\big(C^{\omega(\tilde{r}_1)}\big)$ for some absolute constant $C$, with elements depending on $r_1, \tilde{r}_1, \tilde{\delta}, n_1, n_2$ only, and with all elements coprime to $\tilde{r}_1$,  such that the sum $\hat{B}_{1,hH}[r_1,r_2,k]$ defined in \eqref{decompBhat} satisfies
\begin{align*}
\hat{B}^{\bseps}_{1, hH}[r_1, r_2, k] \ll  
r_1^{5/2}\sum_{\delta^{\flat}\mid(r^{\flat}_1,n_1-n_2,k)}&\sum_{\substack{(r^{\sharp}_1,hH)\mid\delta'\mid(r^{\flat}_1,hH)\\(\delta^{\flat}\delta',r^{\flat}_1/\textnormal{rad}\,r^{\flat}_1)\mid k}}\sum_{\tilde{\delta}\mid(\tilde{r}_1,n_1-n_2,k)}\\
&(r_1,\delta^{\flat}\delta'\tilde{\delta},k)^{1/2}\left(\Big(\frac{k}{\tilde{\delta}}\Big)^2hHr_2-T_{\tilde{\delta}},\Big(\frac{\tilde{r}_1}{\tilde{\delta}}\Big)_{\square}\right)^{1/2},
\end{align*}
where the second factor may be omitted whenever $\tilde{r}_1/\tilde{\delta}$ is cube-free.
\end{prop}

With Proposition~\ref{expsums}, we are ready for the proof of Theorem~\ref{klooster-short}. Denote the sum to be estimated as
\[ S=\sum_{\substack{A<m\leqslant A+M\\(m,r)=1}}S(m,n_1,r)S(m,n_2,r). \]

\def\rad{\mathop{\textnormal{rad}\,}}

Fix a decomposition $r=r_1r_2$ with $(r_1,6r_2)=1$ and a divisor
\[ H\mid\frac{r_1}{\text{rad}\,r_1}, \]
both to be suitably specified later. We set
\[ r_1^{\sharp}=(r_1,H^{\infty}). \]
It is then clear that $r_1^{\sharp}$ is a squareful divisor of $r_1$ of the type considered in \eqref{Defr1sharp}, and that $r_1^{\sharp}\mid H^{\infty}$.

Using the basic estimate on $S$ in \eqref{basicsum} and Proposition~\ref{expsums}, we have that
\begin{align*}
|S|^2&\ll r^{\varepsilon}(M+Hr_2)Hr_2r^2+r^{\varepsilon}Hr_1^{5/2}r_2^3\sum_{\substack{r_1=r_1^{\flat}\tilde{r}_1\\H\mid r_1^{\flat},\,(r_1^{\flat},\tilde{r}_1)=1}}\sum_{d^{\flat}\mid(r_1^{\flat},n_1-n_2)}\sum_{\tilde{d}\mid(\tilde{r}_1,n_1-n_2)}\sum_{\substack{0<|h|\leqslant\frac{M}{Hr_2}\\(h,\tilde{r}_1)=1}}\sum_{(r^{\sharp}_1,hH)\mid d'\mid (r^{\flat}_1,hH)}\\
&\qquad\qquad\sum_{\substack{|k|\leqslant r_1/2,\,\,d^{\flat}\tilde{d}\mid k,\\(d^{\flat}d',r_1^{\flat}/\rad(r_1^{\flat}))\mid k}}(r_1,d^{\flat}d'\tilde{d},k)^{1/2}\min\left(\frac{M}{r_1},\frac1{|k|}\right)\left(\left(\frac k{\tilde{d}}\right)^2hHr_2-T_{\tilde{d}},\left(\frac{\tilde{r}_1}{\tilde{d}}\right)_{\square}\right)^{1/2}\\
&=r^{\varepsilon}(M+Hr_2)Hr_2r^2+r^{\varepsilon}Hr_1^{5/2}r_2^3\sum_{\substack{r_1=r_1^{\flat}\tilde{r}_1\\ H\mid r_1^{\flat},\,(r_1^{\flat},\tilde{r}_1)=1}}\sum_{d^{\flat}\mid(r_1^{\flat},n_1-n_2)}\sum_{\tilde{d}\mid(\tilde{r}_1,n_1-n_2)}\sum_{\substack{0<|h|\leqslant\frac{M}{Hr_2}\\(h,\tilde{r}_1)=1}}\sum_{(r^{\sharp}_1,hH)\mid d'\mid (r^{\flat}_1,hH)}\\
&\qquad\qquad\sum_{\substack{d',\,(d^{\flat}d',r_1^{\flat}/\rad(r_1^{\flat}))\mid d^{\sharp}\\ d^{\sharp}\mid(d^{\flat}d',r_1^{\flat})}}(d^{\sharp}\tilde{d})^{1/2}
\sum_{|\ell|\leqslant\tfrac{r_1}{2d^{\sharp}\tilde{d}}}
\min\left(\frac{M}{r_1},\frac1{d^{\sharp}\tilde{d}|\ell|}\right)\left(d^{\sharp}{}^2\ell^2hHr_2-T_{\tilde{d}},\left(\frac{\tilde{r}_1}{\tilde{d}}\right)_{\square}\right)^{1/2}.
\end{align*}

We collect various contributions to the right-hand side. The contribution of the terms with $\ell=0$ is
\begin{align*}
&\ll r^{\varepsilon}Hr_1^{5/2}r_2^3(r_1,H(n_1-n_2))^{1/2}\cdot\frac{M}{Hr_2}\cdot\frac{M}{r_1}\\
&\ll r^{\varepsilon}M^2H^{1/2}r_2^2r_1^{3/2}(r_1,n_1-n_2)^{1/2}\\
&\ll r^{\varepsilon}M^2r^{3/2}(Hr_2)^{1/2}(r,n_1-n_2)^{1/2}.
\end{align*}

As for the contributions of the terms with $h,\ell\neq 0$, we majorize the contribution of the four innermost ($h$, $d'$, $d^{\sharp}$, and $\ell$) sums above by
\begin{align*}
&\ll \sum_{H\mid d'\mid r_1^{\flat}}\sum_{d'\mid d^{\sharp}\mid r_1^{\flat}}\sum_{0<|\ell|\leqslant\frac{r_1}{2d^{\sharp}\tilde{d}}}\frac1{(d^{\sharp}\tilde{d})^{1/2}|\ell|}\sum_{0<|h|\leqslant\frac{M}{Hr_2}}\left(h-\big(\overline{d^{\sharp}{}^2\ell^2Hr_2}\cdot T_{\tilde{d}}\big),\left(\frac{\tilde{r}_1}{\tilde{d}}\right)_{\square}\right)^{1/2}\\
&\ll \frac{r^{\varepsilon}}{H^{1/2}}\sum_{\delta\mid(\tilde{r}_1/\tilde{d})_{\square}}\delta^{1/2}\left(1+\frac{M}{Hr_2\delta}\right)\ll\frac{r^{\varepsilon}}{H^{1/2}}\left((\tilde{r}_1)_{\square}^{1/2}+\frac{M}{Hr_2}\right).
\end{align*}
We remark that, if $r_1$ (and hence $\tilde{r}_1$) is cube-free, then the term involving $(\tilde{r}_1)_{\square}^{1/2}$ may be omitted.

Executing the outside three ($r_1=r_1^{\flat}\tilde{r}_1$, $d^{\flat}$, and $\tilde{d}$) summations and collecting all terms, we have that
\[ |S|^2\ll r^{\varepsilon}(M+Hr_2)Hr_2r^2+r^{\varepsilon}M^2r^{3/2}(Hr_2)^{1/2}(r,n_1-n_2)^{1/2}+r^{\varepsilon}r^{5/2}(Hr_2)^{1/2}(r_1/r_1^{\sharp})_{\square}^{1/2}+r^{\varepsilon}\frac{Mr^{5/2}}{(Hr_2)^{1/2}}. \]
This estimate holds for every decomposition $r=r_1r_2$ with $(r_1,6r_2)=1$ and every divisor $H\mid (r_1/\rad r_1)$. Note that the upper bound depends only on the product $Hr_2$ rather than on the individual factors of $H$ and $r_2$. Conceptually, this comes as no surprise, since the product $Hr_2$ was used as the single differencing step in Lemma~\ref{DifferencingLemma}. Also, note that $r_1/r_1^{\sharp}=r/(r,(Hr_2)^{\infty})$.

This brings us to the statement of Theorem~\ref{klooster-short}. For a given divisor $s\mid r$ satisfying $(r, 6^{\infty}) \mid s$, define
\[ H=\left(s,\Big(s,\frac rs\Big)^{\infty}\right),\quad r_2=\frac{s}H,\quad r_1=\frac{r}{r_2}. \]
This choice of $H$ and the decomposition $r=r_1r_2$ satisfy all our conditions, and we have proved
\[ |S|^2\ll r^{\varepsilon}Mr^2s+r^{\varepsilon}\frac{Mr^{5/2}}{s^{1/2}}+r^{\varepsilon}r^2s^2+r^{\varepsilon}M^2r^{3/2}s^{1/2}(r,n_1-n_2)^{1/2}+\sigma^2, \]
where
\begin{equation}
\label{Definitionsigma}
\sigma^2=r^{\varepsilon}r^{5/2}s^{1/2}\left(\frac{r}{(r,s^{\infty})}\right)_{\square}^{1/2}
\end{equation}
satisfies all the stated properties. This completes the proof of Theorem~\ref{klooster-short}. \qed

\section{Estimation of complete sums}
\label{EstimationCompleteSumsSection}

\subsection{Preliminaries} We start with two important lemmas that we will use at several stages of the  fairly long and technical proof of Lemma~\ref{FinalEstimateSigma}. The following lemma is a special case of \cite[Theorem 5]{Bo} which is already implicit in Weil's work. 

\begin{lemma}\label{alggeo} Let $p$ be a prime, and let $f_1, f_2 \in (\mathbb{Z}/p\mathbb{Z})[x]$ be two coprime polynomials, not both of which are constant. Then
\[ \Bigl|\sum_{\substack{x\bmod{p} \\ f_2(x) \not\equiv 0 \bmod{p}}} e\left(\frac{f_1(x) \bar{f}_2(x)}{p}\right) \Bigr|\leq (\deg f_1 + 2\deg f_2 - 1)\sqrt{p} + 1. \]
\end{lemma}

The next lemma is of Hensel type.

\begin{lemma}
\label{HenselsLemma1}
Let $1\leqslant\kappa\leqslant\lambda$, $A\subseteq\mathbb{Z}/p^{\lambda}\mathbb{Z}$, $A+p^{\kappa}\mathbb{Z}\subseteq A$, $f:A\to\mathbb{Z}/p^{\lambda}\mathbb{Z}$, $f_1:A\to(\mathbb{Z}/p^{\lambda}\mathbb{Z})^{\times}$ be such that
\[ f(m+p^{\mu}t)-f(m)-p^{\mu}f_1(m)t\in p^{\mu+1}\mathbb{Z}/p^{\lambda}\mathbb{Z} \]
for all $m\in A$, $t\in\mathbb{Z}$, and $\kappa\leqslant\mu<\lambda$. Then, for all $\kappa\leqslant\mu\leqslant\lambda$, the number $K(p^{\mu})$ of solutions of the congruence
\[ f(m)\equiv \omega\pmod{p^{\mu}} \]
in $m\in A$ modulo $p^{\mu}\mathbb{Z}$ satisfies \[ K(p^{\mu})=K(p^{\kappa}). \]
\end{lemma}

Before heading to the proof, we remark that, in applications of Lemma~\ref{HenselsLemma1}, the condition that $f_1(m)\subseteq(\mathbb{Z}/p\mathbb{Z})^{\times}$ only needs to be checked for $m\in A$ satisfying $f(m)\equiv\omega\pmod{p^{\kappa}}$. This is immediate from the proof but also follows from the statement by applying it with the restricted domain $A\cap f^{-1}(\omega)$.

\begin{proof}
Let $\kappa\leqslant\mu<\lambda$. We prove that $K(p^{\mu})=K(p^{\mu+1})$. Indeed, let $m\in A$ be such that $f(m)\equiv\omega\pmod{p^{\mu}}$. Every solution $m_1\in A$ modulo $p^{\mu+1}$ such that $m_1\equiv m\pmod{p^{\mu}}$ is of the form $m+p^{\mu}t$ for some $t\in\mathbb{Z}/p\mathbb{Z}$. According to the condition of the problem, we have that
\[ f(m+p^{\mu}t)-f(m)-p^{\mu}f_1(m)t\in p^{\mu+1}\mathbb{Z}/p^{\lambda}\mathbb{Z}. \]
We are given that $f(m)\equiv\omega\pmod{p^{\mu}}$, so we can write $f(m)\equiv\omega+p^{\mu}F_m\pmod{p^{\lambda}}$ for some $F_m\in\mathbb{Z}/p^{\lambda-\mu}\mathbb{Z}$. In light of the above display, the congruence $f(m_1)\equiv\omega\pmod{p^{\mu+1}}$ is equivalent to
\begin{gather*}
\omega+p^{\mu}F_m+p^{\mu}f_1(m)t\equiv\omega\pmod{p^{\mu+1}},\\
f_1(m)t\equiv -F_m\pmod{p}.
\end{gather*}
Since $f_1(m)\in(\mathbb{Z}/p^{\lambda}\mathbb{Z})^{\times}$, we above congruence is equivalent to $t\equiv -\overline{f_1(m)}F_m\pmod{p}$, and hence
\[ m_1\equiv m+p^{\mu}t\equiv m-p^{\mu}\overline{f_1(m)}F_m\pmod{p^{\mu+1}}. \]
Denoting by $A(p^{\mu})$ the set of solutions of $f(m)\equiv\omega\pmod{p^{\mu}}$ in $m\in A$ modulo $p^{\mu}\mathbb{Z}_p$, this shows in one move that the canonical reduction map $A(p^{\mu+1})\to A(p^{\mu})$ is both surjective and injective; hence $K(p^{\mu})=K(p^{\mu+1})$. The equality $K(p^{\mu})=K(p^{\kappa})$ for every $\kappa\leqslant\mu\leqslant\lambda$ follows immediately.
\end{proof}

\subsection{Prime case} We now turn to the estimation of  $\Sigma(n_1,n_2,a,k;q)$ for $q$ prime. The following result 
settles the second half of Lemma \ref{FinalEstimateSigma}(2). A more general version is contained in the forthcoming preprint \cite{FKM}.  

\begin{lemma}\label{primecase} Let $q$ be a  prime, and let $n_1, n_2, a, k \in\mathbb{Z}$. Then, the sum $\Sigma(n_1,n_2,a,k;q)$ defined in \eqref{defSigma} satisfies the bound
\begin{displaymath}
   \Sigma(n_1,n_2,a,k;q)  \ll q^{5/2}(q, a(n_1-n_2),k)^{1/2}
\end{displaymath}
with an absolute implied constant.
\end{lemma}

\begin{proof}  Let us first assume that $q \mid n_1$, but $q \nmid n_2$. Then by Weil's bound for Kloosterman sums and standard bounds for Ramanujan sums we have
\[ |\Sigma ( n_1, n_2,a, k;q )|  \leq 4 q\sum_{\substack{m \bmod q\\ (m(m+a), q) = 1)}}  |S(m+a, 0, q)S(m, 0,q)|  \leq 4q^2. \]
The same bound holds by symmetry if $q \nmid n_1$, but $q \mid n_2$. Similarly, if $q \mid n_1$ and $q \mid n_2$, then 
\[ |\Sigma ( n_1, n_2,a, k;q\big)|  \leq\sum_{\substack{m \bmod q\\ (m(m+a), q) = 1}}  |S(m+a, 0, q)S(m, 0, q)|^2   \leq   q. \]

This leaves us with the generic case $q \nmid n_1 n_2$. Here, $\Sigma ( n_1, n_2,a, k;q) = q^2\Sigma^{\circ}$ where
\begin{displaymath}
\begin{split}
& \Sigma^{\circ} =  \sum_{\substack{m \bmod q\\ (m(m+a), q) = 1}} {\rm Kl}_2(n_1(m+a), q) {\rm Kl}_2( n_2(m+a), q){\rm Kl}_2(n_1m, q) {\rm Kl}_2 (n_2m, q) e\left(-\frac{k m}{q}\right)
\end{split}
\end{displaymath}
where ${\rm Kl}_2(m, q) = q^{-1/2} S(1, m, q)$. If $q \mid a(n_1 - n_2)$ and $q\mid k$, then we estimate trivially with Weil's bound, getting the bound $|\Sigma^{\circ} | \leq 16 q$.

On the hand, if $q \nmid a(n_1 - n_2)$ or $q\nmid k$, then we use independence of Kloosterman sheafs (as developed by Katz). We use this in the form of the explicit result on uniform distribution of angles of Kloosterman sums due to Fouvry--Michel--Rivat--S\'ark\H{o}zy~\cite{FMRS} (see also \cite[Proposition 3.2]{FGKM}). 
Among the four linear forms $\ell_1(m)=n_1(m+a)$, $\ell_2(m)=n_2(m+a)$, $\ell_3(m)=n_1m$, and $\ell_4(m)=n_2m$, there may be four, two, or one distinct form(s) modulo $q$, depending on whether neither, one, or both of $q\mid a$ and $q\mid (n_1-n_2)$ hold. We group terms corresponding to the same forms together and find a finite set $\mathcal{L}$ of linear forms over $\mathbb{F}_q$ and integers $\lambda_{\ell}\in\{1,2,4\}$ such that
\[ \Sigma^{\circ}=\sum_{\substack{m\in\mathbb{F}_q\\\ell(m)\neq 0\,(\forall\ell\in\mathcal{L})}}\prod_{\ell\in\mathcal{L}}\mathrm{Kl}_2(\ell(m),q)^{\lambda_{\ell}}e\left(-\frac{km}q\right). \]
Writing $\textrm{Kl}_2(\ell(m),q)=2 \cos\theta(\ell(m))$ and using elementary trigonometry, the term
\[ \textrm{Kl}_2(\ell(m),q)^{\lambda_{\ell}}=[2 \cos\theta(\ell(m))]^{\lambda_{\ell}} \]
can be rewritten as a finite linear combination of $ \sym_k\theta(\ell(m))$ for some $|k|\leqslant\lambda_{\ell}$, $k\equiv\lambda_{\ell}\pmod 2$, where $\sym_k\theta=\sin((k+1)\theta)/\sin\theta$. Corresponding to this, the sum $\Sigma^{\circ}$ can be written as a finite linear combination (with coefficients of absolutely bounded size) of sums of the form
\[ \Sigma^{\circ}_j= \sum_{\substack{m\in\mathbb{F}_q\\\ell(m)\neq 0\,(\forall\ell\in\mathcal{L}}}\prod_{\ell\in\mathcal{L}}\sym\nolimits_{k_{\ell,j}}(\theta(\ell(m)))e\left(-\frac{km}q\right) \]
for some $|k_{\ell,j}|\leqslant\lambda_{\ell}$, $k_{\ell,j}\equiv\lambda_{\ell}\pmod 2$.

According to \cite[Lemma 2.1]{FMRS}, we have the estimate
\[ \Sigma^{\circ}_j\ll q^{1/2} \]
as long as it is not the case that all $k_{\ell,j}=0$ for all $\ell\in\mathcal{L}$ \textit{and} $k=0$ in $\mathbb{F}_q$. This is ensured by our non-degeneracy condition that $q\nmid a(n_1-n_2)$ or $q\nmid k$; in the former case, $|\mathcal{L}|=4$ and $|k_{\ell,j}|=\lambda_{\ell}=1$ for all $\ell\in\mathcal{L}$, while, if $q\nmid k$, then $k\neq 0$ in $\mathbb{F}_p$. Putting everything together, we have that
\[  \Sigma\big( n_1, n_2,a, k;q\big) \ll q^{5/2} \]
if $q \nmid n_1n_2$ and if in addition $q \nmid a(n_1 - n_2)$ or $q\nmid k$, with an absolute implied constant.  \end{proof}

\subsection{Setup of the prime power case} In the case of squareful moduli, the estimation of the multiple exponential sum $\Sigma(n_1,n_2,a,k;q)$ requires the deep tools of algebraic geometry only in some degenerate cases, but nevertheless (or because of this) the argument turns out to be very involved. In this subsection, we prepare ground for this estimation by reducing and decomposing the problem to one of the two distinctly different cases.

We are considering a sum of the form
\begin{equation}
\label{GeneralSumn1n2}
\begin{aligned}
 \Sigma:=\Sigma^A(n_1,n_2,a,k;p^s)=\sumstar_{\substack{m\bmod p^s\\ m,m+a\in n_1(\mathbb{Z}/p\mathbb{Z})^{\times}{}^2}}&\sum_{\bseps\in A}S^{\epsilon_1}(m+a,n_1;p^s)\overline{S^{\epsilon_2}(m,n_1;p^s)}\\
 &S^{\epsilon_3}(m+a,n_2;p^s)\overline{S^{\epsilon_4}(m,n_2;p^s)}e\left(-\frac{km}{p^s}\right),
\end{aligned}
\end{equation}
where $A\in\{A_0,A^{\sharp}\}$, $A=A^{\sharp}$ only if $p\mid a$,  
\begin{align*}
S^{\epsilon}(m,n;p^s)
&=p^{s/2}\tau(\epsilon\cdot(mn)_{1/2},p^s)e\left(\frac{2\epsilon\cdot(mn)_{1/2}}{p^s}\right)\\
&=p^{s/2}\tau(\epsilon\cdot\epsilon(mu,n\bar{u})(mu)_{1/2}(n\bar{u})_{1/2},p^s)e\left(\frac{2\epsilon\cdot\epsilon((mu)_{1/2},(n\bar{u})_{1/2})(mu)_{1/2}(n\bar{u})_{1/2}}{p^s}\right),
\end{align*}
and $u\in(\mathbb{Z}/p\mathbb{Z})^{\times}$ is a fixed representative of the class $n_1(\mathbb{Z}/p\mathbb{Z})^{\times}{}^2=n_2(\mathbb{Z}/p\mathbb{Z})^{\times}{}^2$.

Considering the product of the $\tau$-factors in \eqref{GeneralSumn1n2}, we note that
\begin{align*}
T(m,n_1,n_2,a;p^s)&=\tau\big(\epsilon_1((m+a)u)_{1/2}(n_1\bar{u})_{1/2},p^s\big)\overline{\tau\big(\epsilon_2(mu)_{1/2}(n_1\bar{u})_{1/2},p^s\big)}\\
&\qquad\qquad\tau\big(\epsilon_3((m+a)u)_{1/2}(n_2\bar{u})_{1/2},p^s\big)\overline{\tau\big(\epsilon_4(mu)_{1/2}(n_2\bar{u})_{1/2},p^s\big)}=\tau^{[\bseps]}
\end{align*}
depends only on the product $\epsilon_1\epsilon_2\epsilon_3\epsilon_4$ and the parity of $s$ (using the explicit formula \eqref{gauss sum} for the sign of the Gau{\ss} sum). By relabeling $\bseps$ as necessary, we can write
\[ \Sigma=\sumstar_{\bseps\in A}\tau^{[\bseps]}p^{2s}\sumstar_{\substack{m\bmod p^s\\ m,m+a\in n_1(\mathbb{Z}/p\mathbb{Z})^{\times}{}^2}}e\left(\frac{f^{[\bseps]}(m,n_1,n_2,a,k)}{p^s}\right). \]
Here, we have denoted
\begin{align*}
f^{[\bseps]}(m,n_1,n_2,a,k)
&=2\epsilon_1((m+a)u)_{1/2}(n_1\bar{u})_{1/2}-2\epsilon_2(mu)_{1/2}(n_1\bar{u})_{1/2}\\
&\qquad\qquad+2\epsilon_3((m+a)u)_{1/2}(n_2\bar{u})_{1/2}-2\epsilon_4(mu)_{1/2}(n_2\bar{u})_{1/2}-km\\
&=2A\big((m+a)u\big)_{1/2}-2B(mu)_{1/2}-km,
\end{align*}
where
\begin{equation}
\label{DefinitionAepsBeps}
A=A^{[\bseps]}(n_1,n_2)=\epsilon_1(n_1\bar{u})_{1/2}+\epsilon_3(n_2\bar{u})_{1/2},\quad
B=B^{[\bseps]}(n_1,n_2)=\epsilon_2(n_1\bar{u})_{1/2}+\epsilon_4(n_2\bar{u})_{1/2}.
\end{equation}
Corresponding to the above, we may further write
\begin{equation}
\label{Decomposition}
\Sigma=p^{2s}\sum_{\bseps\in A}\tau^{[\bseps]}\Sigma[A^{[\bseps]}(n_1,n_2),B^{[\bseps]}(n_1,n_2),a,k;p^s],
\end{equation}
where we write more generally
\begin{equation}
\label{DefinitionSigmaAB}
\Sigma[A,B,a,k;p^s]=\sumstar_{\substack{m\bmod p^s\\ m,m+a\in u(\mathbb{Z}/p\mathbb{Z})^{\times}{}^2}}e\left(\frac{f[m,A,B,a,k]}{p^s}\right)
\end{equation}
and
\begin{equation}
\label{DefinitionfmABak}
f[m,A,B,a,k]=2A\big((m+a)u\big)_{1/2}-2B(mu)_{1/2}-km.
\end{equation}
 
Note that, in any case,
\[ A^2-B^2\in\big\{0,\pm 4(n_1\bar{u})_{1/2}(n_2\bar{u})_{1/2}\big\}. \]
We also make the important remark that
\[ \big((n_1\bar{u})_{1/2}+(n_2\bar{u})_{1/2}\big)\big((n_1\bar{u})_{1/2}-(n_2\bar{u})_{1/2}\big)=\bar{u}(n_1-n_2). \]
This shows that $A\equiv 0\pmod p$ or $B\equiv 0\pmod p$ is possible only if $n_1\equiv n_2\pmod p$. Moreover, if $p^{\nu}\exmid(n_2-n_1)$, then $p^{\nu}\exmid A$ if $\epsilon_3=-\epsilon_1$ and $p\nmid A$ otherwise, and analogously for $B$; this also formally holds for $\nu=\infty$.

Suppose that $\nu>0$ and $p\mid A, B$; then, $p^{\nu}\exmid A,B$. It is immediate that $\Sigma[A,B,a,k;p^s]=0$ unless $p^{\nu'}\mid k$, where $\nu'=\min(\nu,s)$. From now on, assume that $p^{\nu'}\mid k$. It is also obvious that, if $\nu\geqslant s$, then $\Sigma[A,B,a,k;p^s]=p^s$. If, on the other hand, $\nu<s$ and $p^{\nu}\exmid A$, $p^{\nu}\exmid B$, then
\begin{equation}
\label{ReductionFormula}
\Sigma[A,B,a,k;p^s]=p^{\nu}\cdot \Sigma\left[\frac{A}{p^{\nu}},\frac{B}{p^{\nu}},a,\frac{k}{p^{\nu}};p^{s-\nu}\right].
\end{equation}

Therefore, it suffices to prove an estimate for the sum $ \Sigma[A,B,a,k;p^s]$ defined in \eqref{DefinitionSigmaAB} (or a finite $\bseps$-average thereof) for  $p\nmid A$ or $p\nmid B$, and   for $s\geqslant 1$. 
We consider the following two situations separately, keeping as a standing condition that $p\nmid A$ or $p\nmid B$.

The case when $p\mid a$ and $A\equiv B\pmod p$ is addressed in Section~\ref{pmidaSection}. Note that, in this case, actually $A=B$. Referring back to \eqref{DefinitionfmABak}, we see that this case is distinguished in that the branches of the square-root in $\big((m+a)u\big)_{1/2}$ and $(mu)_{1/2}$ are aligned so that the leading terms cancel out and, as will be seen, an additional factor of size $|a|_p$ emerges.

The remaining cases, when $p\mid a$ and $A\not\equiv B\pmod p$ as well as when $p\nmid a$, are treated in Section~\ref{pnmidaSection}. In this case, no particular alignment of square-roots occurs, but Hensel liftings become much more delicate, and, if $p\nmid a$, singular critical points are encountered in the stationary phase analysis.

The final results of Sections~\ref{pmidaSection} and \ref{pnmidaSection} are the following Lemmas~\ref{primepowercase} and \ref{FinalEstimate}, respectively.

\begin{lemma}\label{primepowercase} Let $q=p^s$, where $p > 3$ is a prime and $s \geqslant 1$, and let $A,a,k\in\mathbb{Z}$ with $p\mid a$.  Then, the sum $\Sigma[A,A,a,k;q]$ defined in \eqref{DefinitionSigmaAB} 
satisfies
\[ \sum_{\epsilon\in\{\pm 1\}}\Sigma[\epsilon A,\epsilon A,a,k;q] \ll q^{1/2}(q,Aa,k)^{1/2} \]
with an absolute implied constant. Moreover, the left-hand side vanishes unless $(Aa,q/p)\mid k$.
\end{lemma}

\begin{lemma}
\label{FinalEstimate}
Let $q=p^s$, where $p>3$ is a prime and $s\geqslant 1$, and let $A,B\in\mathbb{Z}$ be such that $p\nmid A$ or $p\nmid B$. 
\begin{enumerate}
\item There exists a finite set $T\subset\mathbb{Z}\setminus p\mathbb{Z}$, of absolutely bounded size, depending on $q$, $A$, and $B$ only, such that, for every $k\in\mathbb{Z}$ and every $p\nmid a$,
\[ \sum_{\bseps\in\{\pm 1\}^2}\Sigma[\epsilon_1A,\epsilon_2B,a,k;q]\ll q^{1/2}\big(k^2a-T,q_{\square}\big)^{1/2}, \]
where $q=q_{\square}^2q_1$ with $q_1\in\{1,p\}$, the sum on the left-hand side may be omitted for $s\geqslant 2$, and the second factor may be omitted if $s=2$ or (more generally) if $(k^2a-T,q)\mid p^2$.
\item If $A\not\equiv B\pmod p$, then, for every $p\mid a$ and every $k\in\mathbb{Z}$, 
\[ \sum_{\epsilon\in\{\pm 1\}}\Sigma[\epsilon A,\epsilon B,a,k;q]\ll q^{1/2}. \]
\end{enumerate}
\end{lemma}

\section{Proof of Lemma~\ref{primepowercase}}
\label{pmidaSection}

In this section, we estimate $\Sigma[A,B,a,k;q]$ for $q = p^s$ with $s \geqslant 2$, $p \mid a$, and $A=B$, and prove Lemma~\ref{primepowercase}. We start by noting that, in the case $s=1$,
\[ \Sigma[A,A,a,k;p^s]=\sumstar_{\substack{m\bmod p\\ m\in u(\mathbb{Z}/p\mathbb{Z})^{\times}{}^2}}e\left(-\frac{km}p\right), \]
which can be estimated (and anyway formally falls under the same condition $\nu+\alpha\geqslant s$) as in \eqref{bound3} below. Therefore, in what follows we may and do assume that $s\geqslant 2$.
 
By the assumption $p\mid a$, we can write
\[ a=p^{\alpha}a_0,\quad \alpha \geq 1, \quad p \nmid a_0.\]

In this case, the summation in \eqref{DefinitionSigmaAB} is over $m\in u(\mathbb{Z}/p\mathbb{Z})^{\times}{}^2$, so that we may write $m=\bar{u}x^2$ for some $x \in  (\mathbb{Z}/p\mathbb{Z})^{\times}$. The phase $f[m,A,A,a,k]$ defined in \eqref{DefinitionfmABak} can be rewritten as
\[ f[\bar{u}x^2,A,A,a,k]=2A(x^2+au)_{1/2}-2A(x^2)_{1/2}-k\bar{u}x^2=2A\epsilon_xx\big((1+p^{\alpha}a_0u\bar{x}^2)^{1/2}-1\big)-k\bar{u}x^2, \]
where, for $p\nmid x$, $\epsilon_x:=(x^2)_{1/2}\bar{x}$ depends on $x\bmod p$ only.

As $x\in(\mathbb{Z}/p\mathbb{Z})^{\times}{}^2$, we see that $m=\bar{u}x^2$ runs over all admissible values of $m$ twice. Thus,
\begin{align*}
\sum_{\epsilon\in\{\pm 1\}}\Sigma[\epsilon A,\epsilon A,a,k;p^s]
&=\frac12\sum_{\epsilon\in\{\pm 1\}}\sumstar_{x\bmod p^s}e\left(\frac{2A\epsilon\epsilon_xx\big((1+p^{\alpha}a_0u\bar{x}^2)^{1/2}-1\big)-k\bar{u}x^2}{p^s}\right)\\
&=\sum_{\epsilon\in\{\pm 1\}}\tilde{\Sigma}[\epsilon A,a,k;p^s],
\end{align*}
where
 \begin{gather}
 \label{box}
  \tilde{\Sigma}[A,a,k;p^s]
=\frac12\sumstar_{x\bmod p^s}e\left(\frac{\tilde{f}(A,a,k;x)}{p^s}\right), \\
\label{box1}
\tilde{f}(A,a,k;x)=2Ax\big((1+p^{\alpha}a_0u\bar{x}^2)^{1/2}-1\big)-k\bar{u}x^2.
\end{gather}

We proceed to estimate the sum $\tilde{\Sigma}[A,a,k;p^s]$ defined as in \eqref{box} for an arbitrary $A\in\mathbb{Z}$, and we define $\nu=\min(\ord_pA,s)$. For every $\kappa\geqslant 1$, we find that
\begin{align*}
&\overline{x+p^{\kappa}t} =\bar{x}-\bar{x}^2\cdot p^{\kappa}t+\bar{x}^3\cdot p^{2\kappa}t^2+\mathbf{M}_{p^{3\kappa}},\\
&\overline{x+p^{\kappa}t}^2 =\bar{x}^2-2\bar{x}^3\cdot p^{\kappa}t+3\bar{x}^4\cdot p^{2\kappa}t^2+\mathbf{M}_{p^{3\kappa}},\\
& \big(1+ p^{\alpha}a_0u\cdot\overline{x+p^{\kappa}t}^2\big)^{1/2} =\Big(\big(1+p^{\alpha}a_0u\bar{x}^2\big)-2a_0u\bar{x}^3\cdot p^{\kappa+\alpha}t+3a_0u\bar{x}^4\cdot p^{2\kappa+\alpha}t^2+\mathbf{M}_{p^{3\kappa+\alpha}}\Big)^{1/2}\\
&\quad\quad\quad\quad=\big(1+p^{\alpha}a_0u\bar{x}^2\big)^{1/2}-\overline{1+p^{\alpha}a_0u\bar{x}^2}^{1/2}\cdot a_0u\bar{x}^3\cdot p^{\kappa+\alpha}t+3\cdot\bar{2}\cdot a_0u\bar{x}^4\cdot p^{2\kappa+\alpha}t^2+\mathbf{M}_{p^{2\kappa+\alpha+1}},
\end{align*}
and so, finally,
\begin{equation}
\label{TwoTermExpansion}
\begin{aligned}
\big(x+{}&p^{\kappa}t\big)\Big(\big(1+p^{\alpha}a_0u\cdot\overline{x+p^{\kappa}t}^2\big)^{1/2}-1\Big)\\
&=x\Big(\big(1+p^{\alpha}a_0u\bar{x}^2\big)^{1/2}-1\Big)-\overline{1+p^{\alpha}a_0u\bar{x}^2}^{1/2}\cdot a_0u\bar{x}^2\cdot p^{\kappa+\alpha}t
+3\cdot\bar{2}\cdot a_0u\bar{x}^3\cdot p^{2\kappa+\alpha}t^2\\
&\qquad\quad+\Big(\big(1+p^{\alpha}a_0u\bar{x}^2\big)^{1/2}-1\Big)\cdot p^{\kappa}t- a_0u\bar{x}^3\cdot p^{2\kappa+\alpha}t^2+\mathbf{M}_{p^{2\kappa+\alpha+1}}\\
&=x\Big(\big(1+p^{\alpha}a_0u\bar{x}^2\big)^{1/2}-1\Big)+\Big(\overline{1+p^{\alpha}a_0u\bar{x}^2}^{1/2}-1\Big)\cdot p^{\kappa}t+\bar{2}\cdot a_0u\bar{x}^3\cdot p^{2\kappa+\alpha}t^2+\mathbf{M}_{p^{2\kappa+\alpha+1}}.
\end{aligned}
\end{equation}

Using \eqref{box1} and \eqref{TwoTermExpansion}, we have that, for every $\kappa\geqslant 1$,
\begin{align*}
\tilde{f}\big(A,a,k;x+p^{\kappa}t\big) = \tilde{f}(A,a,k;x)+2\Big[A\Big(&\overline{1+p^{\alpha}a_0u\bar{x}^2}^{1/2}-1\Big)-k\bar{u}x\Big]\cdot p^{\kappa}t\\
&+\big(Aa_0u\bar{x}^3p^{\alpha}-k\bar{u}\big)\cdot p^{2\kappa}t^2+\mathbf{M}_{p^{2\kappa+\nu+\alpha+1}}.
\end{align*}
At this point, note that
\[ \ord_p\Big[A\Big(\overline{1+p^{\alpha}a_0u\bar{x}^2}^{1/2}-1\Big)\Big]=\ord_p\big(Aa_0u\bar{x}^3p^{\alpha}\big)=\nu+\alpha. \]

\bigskip

We first consider the principal case when 
\begin{equation}\label{a}
\nu+\alpha\leqslant s-1.
\end{equation}
 Let $\omega=\ord_pk$, and define  $\kappa_{\star}$ and $j$   by
\begin{equation}\label{b}
  s=\min\big(\nu+\alpha,\omega\big)+2\kappa_{\star}+j,\quad \kappa_{\star}\geqslant 0,\,\,j\in\{0,1\}. \end{equation}
Then, for $\upsilon\in\{0,1\}$ (and $\upsilon=1$ if $\kappa_{\star}=0$), we have that 
\begin{equation}\label{computation}
\begin{split}
\tilde{\Sigma}[A,a,k;p^s] &=\frac12\frac1{p^{s-\kappa_{\star}-\upsilon}}\sumstar_{x\bmod p^s}\sum_{t\bmod p^{s-\kappa_{\star}-\upsilon}}e\left(\frac{\tilde{f}\big(A,a,k;x+p^{\kappa_{\star}+\upsilon}t\big)}{p^s}\right)\\
&=\frac12p^{-s+\kappa_{\star}+\upsilon}\sumstar_{x\bmod p^s}e\left(\frac{\tilde{f}(A,a,k;x)}{p^s}\right)\times\\
&\qquad\qquad \times\sum_{t\bmod p^{s-\kappa_{\star}-\upsilon}}e\Biggl(\frac{A\big(\overline{1+p^{\alpha}a_0u\bar{x}^2}^{1/2}-1\big)-k\bar{u}x}{p^{s-\kappa_{\star}-\upsilon}}t +\frac{Aa_0u\bar{x}^3p^{\alpha}-k\bar{u}}{p^{s-2\kappa_{\star}-2\upsilon}}t^2\Biggr).
\end{split}
\end{equation}
We first use this formula with $\upsilon=j$. With this choice, there is no quadratic term in the inner sum, and in fact it vanishes unless
\begin{equation}\label{star}
 \ord_p\Big[A\Big(\overline{1+p^{\alpha}a_0u\bar{x}^2}^{1/2}-1\Big)-k\bar{u}x\Big]\geqslant s-\kappa_{\star}-j=\min\big(\nu+\alpha,\omega\big)+\kappa_{\star}, 
 \end{equation}
when it equals $p^{s-\kappa_{\star}-j}$. We see that we cannot have $\omega<\nu+\alpha$, for then \eqref{a} and \eqref{b} would imply $\kappa_{\star}\geqslant 1$, contradicting \eqref{star}. Hence from now  on we assume
\begin{equation}\label{omega}
  \omega \geq \nu + \alpha.
\end{equation}

We now distinguish two subcases, namely $\nu+\alpha\leqslant s-2$ and $\nu+\alpha= s-1$. In the former case, we have that  $\omega=\nu+\alpha$, for if $\omega > \nu + \alpha$, then \eqref{star} implies $\kappa = 0$, contradicting \eqref{b}. Now, \eqref{star} implies that
\[ A(-\bar{2} p^{\alpha} a_0u\bar{x}^2)  - k\bar{u}x \equiv 0 \pmod{p^{\omega+\kappa_{\star}}}, \]
and so 
\[ Aa_0u\bar{x}^3p^{\alpha}-k\bar{u}\equiv -3k\bar{u}\pmod{p^{\omega+\kappa_{\star}}}. \]
In particular, the left-hand side has order $p^{\omega}$ (since $p > 3$). Using \eqref{computation} with $\upsilon=0$, we are left with a constant sum if $j=0$ and a nondegenerate quadratic Gau{\ss} sum modulo $p$ if $j=1$; in either case, it follows that
\begin{equation}\label{c}
\tilde{\Sigma}[A,a,k;p^s]\ll p^{-\frac j2}\cdot\#\left\{x\bmod p^s:A\Big(\overline{1+p^{\alpha}a_0u\bar{x}^2}^{1/2}-1\Big)\equiv k\bar{u}x\bmod{p^{\nu+\alpha+\kappa_{\star}}}\right\}
\end{equation}
if $\ord_pk=\nu+\alpha$, and $\tilde{\Sigma}[A,a,k;p^s]=0$ otherwise.

We bound the number of solutions of the congruence modulo $p^{\nu+\alpha+\kappa_{\star}}$ in \eqref{c} using Lemma~\ref{HenselsLemma1}. Write
\[ A=p^{\nu}A_0, \quad k=p^{\nu+\alpha}k_0, \]
with $(A_0,p)=(k_0,p)=1$. In light of $\overline{1+p^{\alpha}a_0u\bar{x}^2}^{1/2}-1=-\bar{2}p^{\alpha}a_0u\bar{x}^2+\mathbf{M}_{p^{2\alpha}}$, we have that
\[ f(x):=p^{-\nu-\alpha}\Big[A\Big(\overline{1+p^{\alpha}a_0u\bar{x}^2}^{1/2}-1\Big)-k\bar{u}x\Big]=A_0\frac{\overline{1+p^{\alpha}a_0u\bar{x}^2}^{1/2}-1}{p^{\alpha}}-k_0\bar{u}x \]
is a map $(\mathbb{Z}/p^s\mathbb{Z})^{\times}\to\mathbb{Z}/p^s\mathbb{Z}$. The congruence $f(x)\equiv 0\pmod p$ implies that
\begin{gather*}
k_0\bar{u}x\equiv -\bar{2}a_0uA_0\bar{x}^2\pmod p,\\
x^3\equiv -\bar{2}a_0u^2\bar{k}_0A_0\pmod p
\end{gather*}
and hence has $\text{O}(1)$ solutions.  Moreover, since, for every $\kappa\geqslant 1$,
\begin{align*}
\overline{1+p^{\alpha}a_0u\cdot\overline{x+p^{\kappa}t}^2}^{1/2}
&=\overline{\big(1+p^{\alpha}a_0u\bar{x}^2\big)+p^{\alpha}a_0u\big(-2\bar{x}^3p^{\kappa}t+\mathbf{M}_{p^{2\kappa}}\big)}^{1/2}\\
&=\overline{1+p^{\alpha}a_0u\bar{x}^2}^{1/2}+\big(\overline{1+p^{\alpha}a_0u\bar{x}^2}^{1/2}\big)^3p^{\alpha}a_0u\bar{x}^3\cdot p^{\kappa}t+\mathbf{M}_{p^{\alpha+2\kappa}},
\end{align*}
we have that
\[ f(x+p^{\kappa}t)-f(x)-p^{\kappa}f_1(x)t\in p^{\kappa+1}\mathbb{Z}/p^s\mathbb{Z} \]
for every $\kappa\geqslant 1$, with
\[ f_1(x)=A_0a_0u\bar{x}^3-k_0\bar{u}\equiv  -3k_0\bar{u}\not\equiv 0\pmod p \]
for every $x$ such that $f(x)\equiv 0\pmod p$.

By Lemma~\ref{HenselsLemma1}, we conclude that the congruence $f(x)\equiv 0\pmod{p^{\kappa_{\star}}}$ has $\text{O}(1)$ solutions $x$ modulo $p^{\kappa_{\star}}$ and hence
\[ \text{O}(p^{s-\kappa_{\star}})=\text{O}\big(p^{(s+j)/2+(\nu+\alpha)/2}\big) \]
solutions in $x$ modulo $p^s$ with the notation as in \eqref{b}. 
Substituting this bound into \eqref{c}, we conclude   for $ \nu+\alpha\leqslant s-2$  that
\begin{equation}\label{bound1a}
\tilde{\Sigma}[A,a,k;p^s]\ll
\begin{cases} p^{\frac12s+\frac12(\nu+\alpha)},&\ord_pk=\nu+\alpha,\\
0,&\text{else}.
\end{cases}
\end{equation}

\bigskip

Our second subcase is $\nu+\alpha=s-1$. Here, we find that $\tilde{f}(A,a,k;x)$ is an even function of $x$ such that $\tilde{f}(A,a,k;x)=p^{s-1}\tilde{f}_1(A,a,k;x)$ with
\[ \tilde{f}_1(A,a,k;x)\equiv A_0a_0u\bar{x}-k_1\bar{u}x^2\pmod p, \]
where $k=p^{s-1}k_1$, $\ord_pk_1\geqslant 0$. Therefore, by \eqref{box}--\eqref{box1}, 
\[ \tilde{\Sigma}[A,a,k;p^s]=p^{s-1}\sumstar_{x\bmod p}e\left(\frac{\epsilon a_0u\bar{x}-k_1\bar{u}x^2}p\right). \]
The resulting sum can be estimated by Lemma \ref{alggeo} as $\ll p^{1/2}$ if $\ord_pk_1=0$ and becomes the Ramanujan sum (and is hence $\ll 1$) if $\ord_pk_1\geqslant 1$. Hence, for $  \nu+\alpha=s-1$,
\begin{equation}\label{bound2}
\tilde{\Sigma}[A,a,k;p^s]\ll
\begin{cases}
p^{s-1}, &p^s\mid k,\\
p^{\frac12s+\frac12(\nu+\alpha)}, &\ord_pk=\nu+\alpha,\\
0, &\text{else}. \end{cases}
\end{equation}
This completes the analysis of the case $\nu + \alpha \leqslant s-1$.

In the complementary case when  $\nu+\alpha\geqslant s$, we are dealing with a quadratic Gau{\ss} sum:
\[ \tilde{\Sigma}[A,a,k;p^s]=\frac12\sumstar_{x\bmod  p^s}e\left(-\frac{k\bar{u}x^2}{p^s}\right), \]
which vanishes unless $k=p^{s-1}k_1$ for some $k_1\in\mathbb{Z}/p\mathbb{Z}$, in which case it is $\ll p^{s-1/2}$ if $\ord_pk_1=0$ and $\ll p^s$ if $p\mid k$. Therefore, for $ \nu+\alpha\geqslant s$,
\begin{equation}\label{bound3}
\tilde{\Sigma}[A,a,k;p^s]\ll
\begin{cases}
p^{s}, &p^s\mid k,\\
p^{s-\frac12}, &\ord_pk=s-1,\\
0, &\text{else}. \end{cases}
\end{equation}

Combining our findings \eqref{bound1a}, \eqref{bound2}, \eqref{bound3} completes the proof of Lemma~\ref{primepowercase}. \qed


\section{Proof of Lemma~\ref{FinalEstimate}}
\label{pnmidaSection}


In this section, we estimate the sum $\Sigma[A,B,a,k;q]$ defined in \eqref{DefinitionSigmaAB} for $q=p^s$ with $s\geqslant 1$, $p\nmid A$ or $p\nmid B$, and either of the following two conditions holds:
\begin{enumerate}
\item\label{coprimecase} $p\nmid a$, or
\item\label{unalignedcase} $p\mid a$ and $A\not\equiv B\pmod p$.
\end{enumerate}
As the final result of this section, we obtain a proof of Lemma~\ref{FinalEstimate}.

We remark that the argument in the previous section relied heavily on the fact that $p\mid a$ and $A=B$, which results in a specific alignment of the branches of the square-root. This section's argument, which addresses all remaining cases, is different (and harder), in particular due to the possible presence of singular critical points in the stationary phase analysis. Recall the notation \eqref{notation}.


\subsection{Preliminaries}
We start with some useful differencing formulas.  Recall our assumption that $p\neq 2$. Note that, for every $\kappa\geqslant 1$ and every $t\in\mathbb{Z}_p$,
\begin{equation}
\label{BasicExpansions}
\begin{aligned}
\big(m+p^{\kappa}t\big)_{1/2}&=m_{1/2}+\bar{2}\cdot\overline{m_{1/2}}p^{\kappa}t-\bar{8}\cdot\overline{m_{1/2}}^3\cdot p^{2\kappa}t^2+\mathbf{M}_{p^{3\kappa}},\\
\overline{(m+p^{\kappa}t)_{1/2}}&=\overline{m_{1/2}}-\bar{2}\cdot\overline{m_{1/2}}^3\cdot p^{\kappa}t+3\cdot\bar{8}\cdot\overline{m_{1/2}}^5\cdot p^{2\kappa}t^2+\mathbf{M}_{p^{3\kappa}},\\
\overline{(m+p^{\kappa}t)_{1/2}}^3&=\overline{m_{1/2}}^3-3\cdot\bar{2}\cdot\overline{m_{1/2}}^5\cdot p^{\kappa}t+\mathbf{M}_{p^{2\kappa}}.
\end{aligned}
\end{equation}
Denote
\begin{align*}
g(m,A,B,a)&=Au\overline{\big((m+a)u\big)_{1/2}}-Bu\overline{(mu)_{1/2}},\\
g_1(m,A,B,a)&=-\bar{2}Au^2\overline{\big((m+a)u\big)_{1/2}}^3+\bar{2}Bu^2\overline{(mu)_{1/2}}^3,\\
g_2(m,A,B,a)&=3\cdot\bar{4}\cdot Au^3\overline{\big((m+a)u\big)_{1/2}}^5-3\cdot\bar{4}\cdot Bu^3\overline{(mu)_{1/2}}^5.
\end{align*}
Using \eqref{BasicExpansions}, we thus obtain the following differencing expansions:
\begin{equation}
\label{Expansionsofgg1g2}
\begin{aligned}
f[m+p^{\kappa}tA,B,a,k]&=f[m,A,B,a,k]+\big(g(m,A,B,a)-k\big)\cdot p^{\kappa}t\\
&\mskip 220mu 
+\bar{2}\cdot g_1(m,A,B,a)\cdot p^{2\kappa}t^2+\mathbf{M}_{p^{3\kappa}},\\
g(m+p^{\kappa}t,A,B,a)&=g(m,A,B,a)+g_1(m,A,B,a)\cdot p^{\kappa}t\\
&\mskip 220mu +\bar{2}\cdot g_2(m,A,B,a)\cdot p^{2\kappa}t^2+\mathbf{M}_{p^{3\kappa}},\\
g_1(m+p^{\kappa}t,A,B,a)&=g_1(m,A,B,a)+g_2(m,A,B,a)\cdot p^{\kappa}t+\mathbf{M}_{p^{2\kappa}}.
\end{aligned}
\end{equation}

\subsection{The prime case}
In this subsection, we address the case $s=1$ and prove an estimate for
\begin{equation}
\label{SigmaHatFirst}
\begin{aligned}
\hat{\Sigma}&=\sum_{\bseps\in\{\pm 1\}^2}\Sigma[\epsilon_1A,\epsilon_2B,a,k;p]\\
& =\sum_{\bseps\in\{\pm 1\}^2}\sumstar_{\substack{m\bmod p\\ m,m+a\in u(\mathbb{Z}/p\mathbb{Z})^{\times}{}^2}}e\left(\frac{2\epsilon_1A\big((m+a)u\big)_{1/2}+2\epsilon_2B(mu)_{1/2}-km}p\right).
\end{aligned}
\end{equation}

We first consider the case \eqref{coprimecase}, when $p\nmid a$. Denoting $x=\epsilon_1\big((m+a)u\big)_{1/2}$ and $y=\epsilon_2(mu)_{1/2}$, we have that 
\[ (x+y)(x-y)=x^2-y^2=au, \]
so that $v=x+y\in(\mathbb{Z}/p\mathbb{Z})^{\times}$ and $x-y=au\bar{v}$, as well as $v^2,-(au\bar{v})^2\not\equiv au\pmod p$, that is, $v^2\not\equiv\pm au\pmod p$. Conversely, if $v\in(\mathbb{Z}/p\mathbb{Z})^{\times}$ is arbitrary such that $v^2\not\equiv\pm au\pmod p$, and if we choose
\[ x=\bar{2}(v+au\bar{v})\quad\text{and}\quad y=\bar{2}(v-au\bar{v}) \]
so that $x+y=v$ and $x-y=au\bar{v}$, then $x^2-y^2=au$ and so $x^2=(m+a)u$ and $y^2=mu$ for some $m\in(\mathbb{Z}/p\mathbb{Z})^{\times}$. In this case, $m,m+a\in u(\mathbb{Z}/p\mathbb{Z})^{\times}{}^2$ is automatic, and
\[ m\equiv\bar{u}y^2\equiv\bar{4}\bar{u}(v-au\bar{v})^2\pmod p. \]
This discussion shows that
\begin{equation}
\label{SigmaHatSecond}
\hat{\Sigma}=\sumstar_{\substack{v\bmod p\\v^2\not\equiv\pm au\bmod p}}e\left(\frac{R(v)}p\right),
\end{equation}
where $R(v)$ is a rational function given by
\begin{equation}
\label{RvDefinition}
R(v)=A(v+au\bar{v})+B(v-au\bar{v})-\bar{4}k\bar{u}(v-au\bar{v})^2
\end{equation}
(which can never be constant modulo $p$). By Lemma \ref{alggeo} we conclude 
\begin{equation}\label{primecasebound}
  \hat{\Sigma} \ll p^{1/2}.
  \end{equation} 
  
In the easier case \eqref{unalignedcase}, the inner sum in \eqref{SigmaHatFirst} in over $m\in u(\mathbb{Z}/p\mathbb{Z})^{\times}{}^2$, and, by writing $m=\bar{u}x^2$, we have that
\[ \sum_{\epsilon\in\{\pm 1\}}\Sigma[\epsilon A,\epsilon B,a,k;p]=\sumstar_{x\bmod p}e\left(\frac{2(A-B)x-k\bar{u}x^2}p\right)\ll p^{1/2}, \]
by the evaluation of the Gauss sum~\eqref{gauss sum}, or by an application of Lemma~\ref{alggeo}.

\subsection{Lemmata on Hensel liftings}
\label{HenselLemmata}
Estimating the sum $\Sigma$ using the method of stationary phase involves solving congruences of the form
\[ g(m,A,B,a)\equiv k\pmod{p^{\kappa}}. \]
The following lemma is concerned with the base case $\kappa=1$.

\begin{lemma}
\label{O1Solutionsg}
Let $p\neq 2$, $p\nmid A$ or $p\nmid B$, and either $p\nmid a$, or $p\mid a$ and $A\not\equiv B\pmod p$. Then, the congruence
\begin{equation}
\label{StationaryPointCondition_s2}
g(m,A,B,a)\equiv k\pmod p.
\end{equation}
has $\textnormal{O}(1)$ solutions in $m$ modulo $p$.
\end{lemma}

\begin{proof}
We may rewrite \eqref{StationaryPointCondition_s2} as
\begin{equation}
\label{rewritten}
A\overline{\big((m+a)u\big)_{1/2}}\equiv B\overline{(mu)_{1/2}}+k\bar{u}\pmod p,
\end{equation}

If $p\nmid a$, we obtain by repeated squaring from \eqref{rewritten} that
\[ A^2\overline{m+a}\equiv B^2\overline{m}+2Bk\overline{(mu)_{1/2}}+k^2\bar{u}\pmod p, \]
\[ \big(A^2\overline{m+a}-B^2\overline{m}-k^2\bar{u}\big)^2\equiv 4B^2\overline{mu}k^2\pmod p. \]
Expanding and multiplying by $m^2(m+a)^2$, we obtain the congruence
\begin{align*}
k^4\bar{u}^2&m^2(m+a)^2-4B^2k^2\bar{u}m(m+a)^2\\
&+2B^2k^2\bar{u}(m+a)^2-2A^2k^2\bar{u}m^2-2A^2B^2m(m+a)+B^4(m+a)^2+A^4m^2\equiv 0\pmod p.
\end{align*}
We immediately see that, if $k\not\equiv 0\pmod p$, we have a quartic equation, and so it can have at most four solutions mod $p$.

We next consider the case when $k\equiv 0\pmod p$. In this case, after squaring the condition \eqref{StationaryPointCondition_s2}, we have that
\[ A^2m\equiv B^2(m+a)\pmod p. \]
This congruence has precisely one solution (for given $A$, $B$) in the case when $A^2-B^2\not\equiv 0\pmod p$ and no solutions when $A^2\equiv B^2\not\equiv 0\pmod p$; in all these cases, \eqref{StationaryPointCondition_s2} has $\text{O}(1)$ solutions; this completes the proof of our lemma in the case $p\nmid a$.

If $p\mid a$ and $A\not\equiv B\pmod p$, then \eqref{rewritten} is equivalent to the congruence
\[ (A-B)\overline{(mu)_{1/2}}\equiv k\bar{u}\pmod p, \]
which has at most one solution, given by
\[ m\equiv (A-B)^2\bar{k}^2u\pmod p. \]
(In particular, there are no solutions if $p\mid k$.) This proves our lemma in the case $p\mid a$, $A\not\equiv B\pmod p$.

Although we will not need this, we remark that, for $p^{\alpha}\exmid a$, it follows by exactly the same argument that the congruence $g(m,A,B,a)\equiv k\pmod{p^{\alpha}}$ has at most one solution modulo $p^{\alpha}$, given explicitly by $m\equiv (A-B)^2\bar{k}^2u\pmod{p^{\alpha}}$ (this being a solution of exactly one of the two congruences corresponding to the two pairs $(\epsilon A,\epsilon B)$ entering the statement of Lemma~\ref{FinalEstimate}).
\end{proof}

An immediate consequence of Lemma~\ref{O1Solutionsg} and Lemma~\ref{HenselsLemma1} is the following statement. Note that, when $p\mid a$ and $A\not\equiv B\pmod p$, the congruence $g_1(m,A,B,a)\equiv 0\pmod p$ has no solutions.

\begin{lemma}
\label{NonsingularSolutionsCount}
Let $p\neq 2$, $p\nmid A$ or $p\nmid B$, and $\kappa\geqslant 1$.
\begin{enumerate}
\item If $p\nmid a$, then the congruence
\begin{equation}
\label{SPC_again}
g(m,A,B,a)\equiv k\pmod{p^{\kappa}}
\end{equation}
has $\textnormal{O}(1)$ solutions in $m$ modulo $p^{\kappa}$ such that
\[ g_1(m,A,B,a)\not\equiv 0\pmod p. \]
\item If $p\mid a$ and $A\not\equiv B\pmod p$, then \eqref{SPC_again}
has $\textnormal{O}(1)$ solutions in $m$ modulo $p^{\kappa}$.
\end{enumerate}
\end{lemma}

The remainder of this subsection is concerned with the singular solutions to $g(m,A,B,a)\equiv k\pmod{p^{\kappa}}$ in the case $p\nmid a$, that is, those solutions for which $g_1(m,A,B,a)\equiv 0\pmod p$. The following lemma, which will ensure non-singularity of certain congruences, is an elementary exercise.

\begin{lemma}
\label{NonSingularity}
Let $p\neq 2$, $p\nmid a$, and $p\nmid A$ or $p\nmid B$. Then, the system of congruences
\[ g(m,A,B,a)\equiv g_1(m,A,B,a)\equiv 0\pmod p \]
has no solutions in $m$. If additionally $p\neq 3$, then the system of congruences
\[ g_1(m,A,B,a)\equiv g_2(m,A,B,a)\equiv 0\pmod p \]
has no solutions in $m$.
\end{lemma}

\begin{proof}
We consider the first statement; the second is entirely analogous. Assume that
\[ A\overline{\big((m+a)u\big)_{1/2}}\equiv B\overline{(mu)_{1/2}}\pmod p,\qquad A\overline{\big((m+a)u\big)_{1/2}}^3\equiv B\overline{(mu)_{1/2}}^3\pmod p. \]
Then $(m+a)u\equiv mu\pmod p$, contradicting $p\nmid (au)$.
\end{proof}

We can use the previous simple observation in the proof of the following.

\begin{lemma}
\label{Henselg1}
Let $p\not\in\{2,3\}$, $p\nmid a$, and $p\nmid A$ or $p\nmid B$. Then the congruence
\[ g_1(m,A,B,a)\equiv 0\pmod p \]
has $\textnormal{O}(1)$ solutions $m^{\flat}_1,\dots,m^{\flat}_{\omega}$. Furthermore, for every $\kappa\geqslant 1$, the congruence
\[ g_1(m,A,B,a)\equiv 0\pmod {p^{\kappa}} \]
has exactly $\omega$ solutions modulo $p^{\kappa}$. In fact, these solutions may be written as $m^{[\kappa]}_1,\dots,m^{[\kappa]}_{\omega}$ with $m^{[\kappa]}_i\equiv m^{\flat}_i$ for every $1\leqslant i\leqslant\omega$.
\end{lemma}

\begin{proof}
We start with the congruence $g_1(m,A,B,a)\equiv 0\pmod p$, which we rewrite as
\[ A\overline{\big((m+a)u\big)_{1/2}}^3\equiv B\overline{(mu)_{1/2}}^3. \]
Squaring both sides and rearranging, it follows that
\[ \bar{A}^2(m+a)^3-\bar{B}^2m^3\equiv 0\pmod p. \]
This is at most a cubic congruence in $m$ modulo $p$, and certainly its leading and constant coefficients cannot both vanish. Therefore it has $\text{O}(1)$ solutions modulo $p$, say, $m^{\flat}_1,\dots,m^{\flat}_{\omega}$. According to Lemma~\ref{NonSingularity}, each of these solutions satisfies $g_2(m,A,B,a)\equiv 0\pmod p$. Thus the remaining claims follow, in light of \eqref{Expansionsofgg1g2}, from  Lemma \ref{HenselsLemma1}.
\end{proof}

Applying Lemma~\ref{Henselg1} with $\kappa=s$, we obtain $\omega$ solutions
\[ m_1,\dots,m_{\omega} \]
satisfying $g_1(m,A,B,a)\equiv 0\pmod{p^s}$; we denote
\begin{equation}
\label{Definitionki}
k_i=g(m_i,A,B,a).
\end{equation}
We stress again that, according to Lemma~\ref{NonSingularity}, all of these solutions satisfy
\begin{equation}
\label{kiNotZero}
k_i\not\equiv 0\pmod p
\end{equation}
as well as
\[ g_2(m_i,A,B,a)\not\equiv 0\pmod p. \]
We are now ready for the following lemma, which is of key importance in solving our stationary phase problem.

\begin{lemma}
\label{CountingSingularSolutionsLemma}
Let $p\not\in\{2,3\}$, $p\nmid a$, and $p\nmid A$ or $p\nmid B$. Also, let $k\in\mathbb{Z}$ and $1\leqslant\kappa\leqslant s$. Write $\kappa=2\kappa_{\star}+j$ with $j\in\{0,1\}$. The congruence
\[ g(m,A,B,a)\equiv k\pmod{p^{\kappa}} \]
can have solutions such that
\[ g_1(m,A,B,a)\equiv 0\pmod p \]
only if
\[ I(k)=\big\{1\leqslant i\leqslant\omega:k\equiv k_i\bmod{p^{\min(\kappa,2)}}\big\}\neq\emptyset. \]
For each $i\in I(k)$, let
\[ s_i=\begin{cases} p^{\kappa_{\star}}, &p^{\kappa}\mid (k-k_i),\\ p^{\mu}, &p^{2\mu}\exmid (k-k_i)\text{ for some }1\leqslant\mu\leqslant\kappa_{\star},\\ 0,&\text{else}.\end{cases} \]
Then the congruence $g(m,A,B,a)\equiv a\pmod{p^{\kappa}}$ has at most
\[ \ll\sum_{i\in I(k)}s_i \]
solutions modulo $p^{\kappa}$ such that $g_1(m,A,B,a)\equiv 0\pmod p$. In particular, denoting
\[ \rho(k)=\max_{1\leqslant i\leqslant\omega}\ord_p(k-k_i), \]
this number of solutions is
\[ \textnormal{O}\left(p^{\left\lfloor\frac12\min(\rho(k),\kappa)\right\rfloor}\right). \]
\end{lemma}

\begin{proof}
According to Lemma~\ref{Henselg1}, every $m$ such that $g_1(m,A,B,a)\equiv 0\pmod p$ satisfies $m\equiv m_i\pmod p$ for exactly one $1\leqslant i\leqslant\omega$. If $m\equiv m_i\pmod {p^s}$, then according to \eqref{Expansionsofgg1g2} we have that
\[ g(m,A,B,a)\equiv k_i\pmod{p^s}, \quad g_1(m,A,B,a)\equiv 0\pmod{p^s}. \]
Otherwise, write $m=m_i+p^{\mu}t$ for some $1\leqslant\mu<s$ and $p\nmid t$. Using \eqref{Expansionsofgg1g2}, we find that
\[ p^{2\mu}\exmid \big(g(m,A,B,a)-k_i\big),\quad p^{\mu}\exmid g_1(m,A,B,a). \]
In either case, we see that $p^2\mid\big(g(m,A,B,a)-k_i\big)$.

If $\kappa=1$, this shows that solutions of $g(m,A,B,a)\equiv k\pmod p$ such that $g_1(m,A,B,a)\equiv 0\pmod p$ exist only if $k\equiv k_i\pmod p$ for some $1\leqslant i\leqslant\omega$ and that each such solution $m$ must satisfy $m\equiv m_i\pmod p$ for some $i\in I(k)$; in particular, the number of solutions is $\text{O}(1)$. This completes the proof in the case $\kappa=1$.

If $2\leqslant\kappa\leqslant s$, then $k\equiv k_i\pmod{p^2}$ and so $I(k)\neq\emptyset$. We distinguish two cases: $k\equiv k_i\pmod{p^{\kappa}}$ and $k\not\equiv k_i\pmod{p^{\kappa}}$.

In the first case, write $\kappa=2\kappa_{\star}+j$, $\kappa_{\star}\geqslant 1$, $j\in\{0,1\}$. We have that $p^{\kappa}\mid (k-k_i)$, and the congruence to be solved is equivalent to
\[ g(m,A,B,a)\equiv k_i\pmod{p^{\kappa}},\quad m\equiv m_i\pmod p. \]
One solution of this congruence is $m\equiv m_i\pmod{p^{\kappa}}$. Otherwise, and writing $m=m_i+p^{\mu}t$ for some $1\leqslant\mu<s$ and $p\nmid t$, we cannot have $2\mu<\kappa$, that is, we must have $2\mu\geqslant 2\kappa_{\star}+j$ and hence $\mu\geqslant\kappa_{\star}+j$. Keeping in mind that we must have $m\equiv m_i\pmod{p^{\mu}}$, we obtain at most
\[ \text{O}(p^{\kappa-\mu})=\text{O}(p^{\kappa_{\star}}) \]
solutions for $m$ modulo $p^{\kappa}$.

In the second case, let $p^{\lambda}\exmid (k-k_i)$ for some $2\leqslant\lambda<\kappa$. In that case, $p^{\lambda}\exmid \big(g(m,A,B,a)-k_i)$, and so we must have $\lambda=2\mu$ for some $1\leqslant\mu<\lambda<\kappa$. Therefore, $p^{\mu}\exmid (m-m_i)$, and $p^{\mu}\exmid g_1(m,A,B,a)$.

Fix one such solution $m_0$. We now count the number of solutions of
\begin{equation}
\label{ToBeCountedSemiSingular}
g(m,A,B,a)\equiv k\pmod{p^{\mu+\varsigma}},\quad m\equiv m_0\pmod{p^{\mu}}
\end{equation}
modulo $p^{\varsigma}$ for every $\mu\leqslant\varsigma\leqslant\kappa$. Note that $\kappa-\mu\geqslant \mu+1$. For $\varsigma=\mu$, we obviously have exactly one such solution.

We use the second relationship from \eqref{Expansionsofgg1g2}:
\begin{equation}
\label{WithIota}
g(m+p^{\iota}t,A,B,a)=g(m,A,B,a)+g_1(m,A,B,a)\cdot p^{\iota}t+\bar{2}\cdot g_2(m,A,B,a)\cdot p^{2\iota}t^2+\mathbf{M}_{p^{3\iota}}.
\end{equation}
Using \eqref{WithIota} with $\iota=\mu$, we see that the congruence
\[ g(m_0+p^{\mu}t,A,B,a)\equiv k\pmod{p^{2\mu+1}} \]
is equivalent to
\[ \bar{2}\cdot g_2(m_0,A,B,a)\cdot t^2+\frac{g_1(m_0,A,B,a)}{p^{\mu}}t+\frac{g(m_0,A,B,a)-k}{p^{2\mu}}\equiv 0\pmod p. \]
This is a nontrivial quadratic congruence in $t$, and so it has $\text{O}(1)$ solutions in $t$ modulo $p$. Corresponding to this are $\text{O}(1)$ solutions $m$ modulo $p^{\mu+1}$ of \eqref{ToBeCountedSemiSingular} with $\varsigma=\mu+1$.

We now prove that, given a $\varsigma\geqslant \mu+1$ and a solution of $m_1$ of
\[ g(m,A,B,a)\equiv k\pmod{p^{\mu+\varsigma}}, \]
there exists a unique $m_2$ modulo $p^{\varsigma+1}$ such that
\[ g(m,A, B,a)\equiv k\pmod{p^{\mu+\varsigma+1}},\quad m_2\equiv m_1\pmod{p^{\varsigma}}. \]
Indeed, writing $m_2=m_1+p^{\varsigma}t$ and using \eqref{WithIota} with $\iota=\varsigma$ (and noting that $2\varsigma\geqslant\mu+\varsigma+1$), the congruence $g(m_1+p^{\varsigma}t,A,B,a)\equiv k\pmod{p^{\mu+\varsigma}}$ is equivalent to
\[ \frac{g_1(m,A,B,a)}{p^{\mu}}t+\frac{g(m,A,B,a)-k}{p^{\mu+\varsigma}}\equiv 0\pmod{p}. \]
This is a nontrivial linear congruence in $t$, and so it has a unique solution in $t$ modulo $p$. Corresponding to this is a unique solution $m_2$ modulo $p^{\varsigma+1}$ of $g(m,A,B,a)\equiv k\pmod{p^{\mu+\varsigma+1}}$ such that $m_2\equiv m_1\pmod{p^{\varsigma}}$.

Putting everything together, we have proved that, for every $\varsigma\geqslant\mu$, the system \eqref{ToBeCountedSemiSingular} has $\text{O}(1)$ solutions modulo $p^{\varsigma}$. In particular, there are $\text{O}(1)$ solutions modulo $p^{\kappa-\mu}$ of
\[ g(m,A,B,a)\equiv k\pmod{p^{\kappa}}\quad m\equiv m_0\pmod{p^{\mu}}. \]
Adding over all $\text{O}(1)$ values of $m_0$, we finally obtain
\[ \text{O}(p^{\mu}) \]
solutions of $g(m,A,B,a)\equiv k\pmod{p^{\kappa}}$ modulo $p^{\kappa}$. 
\end{proof}

We see from Lemma~\ref{CountingSingularSolutionsLemma} that the numbers $k_i$ play a central role in counting the solutions to $g(m,A,B,a)\equiv k\pmod{p^{\kappa}}$ such that $g_1(m,A,B,a)\equiv 0\pmod p$. In the following lemma, we make the dependence of these special values on the parameter $a$ a bit more explicit.

\begin{lemma}
\label{FiniteSetT}
Let $p\neq 2$, $p\nmid a$, and $p\nmid A$ or $p\nmid B$, and let $k_1, \ldots, k_{\omega}$ be defined as in \eqref{Definitionki}. There exists a finite set $T \subseteq \mathbb{Z}\setminus p\mathbb{Z}$ of absolutely bounded cardinality whose elements  depend on $p^s$, $A$, and $B$ only, such that for every $k\in\mathbb{Z}$ and every $1\leqslant\lambda\leqslant k$, the congruence $k\equiv k_i\pmod{p^{\lambda}}$ for some $1\leqslant i\leqslant\omega$ implies that
\[ k^2a\equiv t\pmod{p^{\lambda}} \]
for some $t\in T$.
\end{lemma}

\begin{proof}
Let $\{v_1,v_2\}\in(\mathbb{Z}/p\mathbb{Z})^{\times}$ be fixed representatives of the two cosets of the subgroup $(\mathbb{Z}/p\mathbb{Z})^{\times}{}^2$, and let $\bar{v}_jv_j\equiv 1\pmod{p^s}$. According to Lemma~\ref{Henselg1}, each of the four congruences
\[ g_1(m,A,\epsilon B,v_j)\equiv 0\pmod{p^s}, \]
where $\epsilon\in\{\pm 1\}$ and $j\in\{1,2\}$, has $\text{O}(1)$ solutions modulo $p^s$, which we denote as
\[ m^{\epsilon,v_j}_1,\dots,m^{\epsilon,v_j}_{\omega(\epsilon,v_j)}. \]
Let
\[ k^{\epsilon,v_j}_r=g\big(m^{\epsilon,v_j}_r,A,\epsilon B,v_j\big). \]
Recall that $k^{\epsilon,v_j}_r\not\equiv 0\pmod p$ by \eqref{kiNotZero}. We claim that the set
\[ T=\left\{\big(k^{\epsilon,v_j}_r\big)^2\bar{v}_j:\epsilon\in\{\pm 1\},\,\,j\in\{1,2\},\,\,1\leqslant r\leqslant\omega(\epsilon,v_j)\right\} \]
satisfies all our properties.

Clearly, it suffices to prove that, for every $p\nmid a$, and with $k_1, \ldots, k_{\omega}$ defined as in \eqref{Definitionki}, we have that, for every $1\leqslant i\leqslant\omega$, there exists a $t\in T$ such that
\[ k_i^2a\equiv g(m_i,A,B,a)^2a\equiv t\pmod{p^s}. \]

Indeed, let $v_j\in\{v_1,v_2\}$ be the chosen representative of the coset $a(\mathbb{Z}/p\mathbb{Z})^{\times}{}^2$. The values $m=m_i$ are solutions of the congruence
\[ A\overline{\big((m+a)u\big)_{1/2}}^3-B\overline{(mu)_{1/2}}^3\equiv 0\pmod{p^s}. \]
Write
\[ m\equiv a\bar{v}_jx\pmod{p^s}, \]
and let $\epsilon_1=\epsilon_1(x,v_j,a)$ and $\epsilon_2=\epsilon_2(x,v_j,a)$ be such that
\begin{gather*}
\big((x+v_j)u(a\bar{v}_j)\big)_{1/2}\equiv\epsilon_1\cdot\big((x+v_j)u\big)_{1/2}(a\bar{v}_j)_{1/2}\pmod{p^s},\\
\big((xu)(a\bar{v}_j)\big)_{1/2}=\epsilon_2\cdot (xu)_{1/2}(a\bar{v}_j)_{1/2}\pmod{p^s}.
\end{gather*}
We stress that $\epsilon_1$ and $\epsilon_2$ may depend on $x$, and that the definition of the set $T$ is such that this causes no problem. With this change of variables, the above congruence is equivalent to the following congruence in $x$ such that $(x+v_j)u,\,xu\in(\mathbb{Z}/p\mathbb{Z})^{\times}{}^2$:
\[ g_1(x,A,\epsilon_1\epsilon_2B,v_j)=A\overline{\big((x+v_j)u\big)_{1/2}}^3-\epsilon_1\epsilon_2 B\overline{(xu)_{1/2}}^3\equiv 0\pmod{p^s}. \]
According to Lemma~\ref{Henselg1}, this means that
\[ x\equiv m^{\epsilon_1\epsilon_2,v_j}_r\pmod{p^s} \]
for some $1\leqslant r\leqslant\omega(\epsilon_1\epsilon_2,v_j)$. Consequently, we find that
\begin{align*}
k_i=g(m_i,A,B,a)
&=A\overline{\big((m+a)u\big)_{1/2}}-B\overline{(mu)_{1/2}}\\
&\equiv\epsilon_1A\overline{\big((x+v_j)u\big)_{1/2}}\cdot\overline{(a\bar{v}_j)_{1/2}}-\epsilon_2 B\overline{(xu)_{1/2}}\cdot\overline{(a\bar{v}_j)_{1/2}}\\
&\equiv\epsilon_1g\big(m^{\epsilon_1\epsilon_2,v_j}_r,A,\epsilon_1\epsilon_2B,v_j\big)\overline{(a\bar{v}_j)_{1/2}}\pmod{p^s}.
\end{align*}
This final congruence implies that
\[ k_i^2a\equiv \big(k^{\epsilon_1\epsilon_2,v_j},r\big)^2v_j\pmod{p^s}, \]
and the right-hand side is an element of the set $T$ by construction.
\end{proof}

As a consequence of Lemmas~\ref{CountingSingularSolutionsLemma} and \ref{FiniteSetT}, we obtain the following compact statement.

\begin{lemma}
\label{FinalCountSingular}
Let $p>3$, $p\nmid a$, and $p\nmid A$ or $p\nmid B$, and let $1\leqslant\kappa\leqslant s$. There exists a finite set $T \subseteq \mathbb{Z}\setminus p\mathbb{Z}$ of absolutely bounded cardinality whose elements  depend on $p^s$, $A$, and $B$ only,  such that the number of solutions of the congruence
\[ g(m,A,B,a)\equiv k\pmod{p^{\kappa}} \]
such that
\[ g_1(m,A,B,a)\equiv 0\pmod p \]
is at most
\[ \textnormal{O}\left(p^{\left\lfloor\frac12\min(\tilde{\rho}(k^2a),\kappa)\right\rfloor}\right), \]
where
\[ \tilde{\rho}(\ell)=\max_{t\in T}\ord_p(\ell-t). \]
\end{lemma}

\subsection{Stationary phase estimates}
In this subsection, we use the facts from subsection~\ref{HenselLemmata} about the number of solutions to the stationary phase congruence \eqref{BasicCongruence}, below, to estimate the sum
\[ \tilde{\Sigma}:=\Sigma[A,B,a,k;p^s]=\sumstar_{\substack{m\bmod p^s\\ m,m+a\in n_1(\mathbb{Z}/p\mathbb{Z})^{\times}{}^2}}e\left(\frac{f[m,A,B,a,k]}{p^s}\right). \]
We recall our general assumptions 
\[ s\geqslant 2,\,\,p> 3,\text{ and }p\nmid A\text{ or }p\nmid B, \]
as well as
\[ \text{either}\quad p\nmid a\quad\text{or}\quad p\mid a,\,\,A\not\equiv B\pmod p. \]

Write $s=2\kappa+j$, $\kappa\geqslant 1$, $j\in\{0,1\}$. Applying the usual stationary phase argument and the first equality in \eqref{Expansionsofgg1g2}, we find that
\begin{equation}
\label{StationaryPhaseBase}
\begin{aligned}
\tilde{\Sigma}=p^{\kappa}&\sumstar_{\substack{m\bmod{p^{\kappa}},\,\,m,m+a\in n_1(\mathbb{Z}/p\mathbb{Z})^{\times}{}^2\\g(m,A,B,a)\equiv k\bmod{p^{\kappa}}}}e\left(\frac{f[m,A,B,a,k]}{p^s}\right)\\
&\qquad\qquad{}\times\sum_{t\bmod{p^j}}e\left(\frac{\big[\big(g(m,A,B,a)-k\big)/p^{\kappa}\big]\cdot t+\bar{2}g_1(m,A,B,a)\cdot t^2}{p^j}\right).
\end{aligned}
\end{equation}

The outer sum is indexed by solutions of the congruence
\begin{equation}
\label{BasicCongruence}
g(m,A,B,a)\equiv k\pmod{p^{\kappa}}
\end{equation}
modulo $p^{\kappa}$. We write 
\[ \tilde{\Sigma}=\tilde{\Sigma}_0+\tilde{\Sigma}_1, \]
where $\tilde{\Sigma}_0$ and $\tilde{\Sigma}_1$ denote the contributions to the right-hand side of \eqref{StationaryPhaseBase} from those solutions to \eqref{BasicCongruence} for which $g_1(m,A,B,a)\not\equiv 0\pmod p$ and those for which $g_1(m,A,B,a)\equiv 0\pmod p$, respectively.

We first consider $\tilde{\Sigma}_0$. According to Lemma~\ref{NonsingularSolutionsCount}, the congruence \eqref{BasicCongruence} has $\text{O}(1)$ solutions such that $g_1(m,A,B,a)\not\equiv 0\pmod p$. Moreover, in this case, the inner sum in \eqref{StationaryPhaseBase} is a non-trivial quadratic Gau{\ss} sum and is $\text{O}(p^{j/2})$. Combining everything, we find that
\[ \tilde{\Sigma}_0\ll p^{\kappa+(j/2)}=p^{s/2}. \]

We next consider the sum $\tilde{\Sigma}_1$; note that this sum can only be nonempty if $p\nmid a$. Let the finite set $T$ and $\tilde{\rho}(\ell)$ be as in Lemma~\ref{FinalCountSingular}. The number of solutions of \eqref{BasicCongruence} such that $g_1(m,A,B,a)\equiv 0\pmod p$ is
\[ \text{O}\left(p^{\left\lfloor\frac12\min(\tilde{\rho}(k^2a),\kappa)\right\rfloor}\right). \]
If $j=0$, then this shows that
\[ \tilde{\Sigma}_1\ll p^{s/2+\left\lfloor\frac12\min(\tilde{\rho}(k^2a),\kappa)\right\rfloor}. \]
If $j=1$, then the inner sum in \eqref{StationaryPhaseBase} is actually a complete exponential sum with a linear phase, so that only the terms with $g(m,A,B,a)\equiv k\pmod{p^{\kappa+1}}$ contribute. We find that, in this case,
\[ \tilde{\Sigma}_1\ll p^{\kappa+\left\lfloor\frac12\min(\tilde{\rho}(k^2a),\kappa+1)\right\rfloor}=p^{s/2+\lfloor\frac12\min(\tilde{\rho}(k^2a),\kappa+1)\rfloor-\frac12}. \]

Putting everything together, we have proved the following estimate.
\begin{lemma}
\label{CoprimeCase}
Let $p > 3$, and $p\nmid A$ or $p\nmid B$, and $s\geqslant 2$.
\begin{enumerate}
\item If $p\nmid a$, then, letting the finite set $T\subset\mathbb{Z}\setminus p\mathbb{Z}$ and $\tilde{\rho}(\ell)$ be as in Lemma~\ref{FinalCountSingular}, we have that
\[ \tilde{\Sigma}\ll p^{s/2}+p^{\lfloor \frac12s\rfloor+\min\big(\lfloor\frac12\tilde{\rho}(k^2a)\rfloor,\lfloor\frac14(s+1)\rfloor\big)}. \]
In particular, writing $q=q_{\square}^2q_1$ with $q_1\in\{1,p\}$, we have that
\[ \tilde{\Sigma}\ll q^{1/2}\big(k^2a-T,q_{\square}\big)^{1/2}, \]
as well as $\tilde{\Sigma}\ll q^{1/2}$ if $\tilde{\rho}(k^2a)\leqslant 2$ or if $s=2$.
\item If $p\mid a$ and $A\not\equiv B\pmod p$, then
\[ \tilde{\Sigma}\ll q^{1/2}. \]
\end{enumerate}
\end{lemma}

Combining Lemma~\ref{CoprimeCase} and the bound \eqref{primecasebound}, which covers the case $s=1$, we obtain Lemma~\ref{FinalEstimate}. \qed

 
 \section{Proof of Theorem \ref{variant}}\label{mass} 
 
In this section we indicate the necessary changes if $f_1$ and $f_2$ are Maa{\ss} forms. The Voronoi formula, Lemma \ref{vor}, reads as follows (see e.g.\ \cite[Proposition 1]{HM}).
\begin{lemma}\label{vor1} Let $c \in \mathbb{N}$, $b \in \mathbb{Z}$, and assume $(b, c) = 1$. Let $V$ be a  smooth compactly supported function, and let $N > 0$. Let $\lambda(n)$ denote the normalized Hecke eigenvalues of a  cuspidal Maa{\ss} newform with spectral parameter $t$ for $\SL_2(\mathbb{Z})$.  Then
\begin{displaymath}
  \sum_n \lambda(n) e\left(\frac{bn}{c}\right) V\left(\frac{n}{N}\right) = \frac{N}{c}  \sum_{\pm} \sum_n \lambda(n) e\left(\mp\frac{\bar{b}n}{c}\right) \mathring{V}^{\pm}\left(\frac{n}{c^2/N}\right) 
\end{displaymath} 
where
\[ \mathring{V}^{\pm}(y) = \int_0^{\infty} V(x) \mathcal{J}_{2it}^{\pm}(4\pi \sqrt{xy}) dx \]
with the notation as in \eqref{defJ+}.  
\end{lemma} 
Note that $\mathring{V}^{\pm}(y)$ is again a Schwartz class function.
The Gamma factors in the Mellin transform of the weight function \eqref{afe}   depend on the parity of $\chi$, so we sum over odd and even characters separately\footnote{The dependence of the Gamma factors and the root number on the parity of $\chi$ is missing in \cite[p.\ 3-4]{St}.}. Note that the root number of $L(s, f_1 \otimes \chi) \overline{L(s, f_2\otimes \chi)}$ is the product of the signs of $f_1$ and $f_2$, and in particular independent of $\chi$. This yields a congruence condition $n \equiv \pm m$ (mod $d$) for various divisors $d\mid q$. The treatment of the diagonal term $n=m$ remains unchanged, but in \eqref{snmd} we define
\[  S_{N, M, d, q} := \frac{d}{(NM)^{1/2}} \sum_{\substack{n \equiv \pm m \bmod{d}\\ (nm, q) = 1\\ n \not= m}}  \lambda_1(m)\lambda_2(n)  V_1\left(\frac{m}{M}\right) V_2\left(\frac{n}{N}\right). \]
Correspondingly, the definition of $\mathcal{D}(\ell_1, \ell_2, h, N, M)$ in \eqref{defD} is changed into
\begin{equation}\label{defDnew}
\mathcal{D}(\ell_1, \ell_2, h, N , M ) = \sum_{\ell_1 n \mp \ell_2 m = h} \lambda_1(m)\lambda_2(n) V_1\left(\frac{\ell_2 m}{M}\right) V_2\left(\frac{\ell_1 n}{N}\right).
\end{equation}
As remarked in \cite{Bl}, the results in this paper hold for Maa{\ss} forms as well, and they are also insensitive to a change of sign in the summation condition. The proof of Proposition~\ref{bound1} in Section \ref{HeckeResidueClasses} requires only some extra signs at the appropriate places.

In Sections \ref{SpectralDecompositionSection} and \ref{ShiftedSumsAverageSection}, we need to keep track of various extra signs, which arise from two principal sources while following the arguments in Subsections~\ref{CircleMethodSubsection} and \ref{VoronoiSummationSubsection}. One source of extra signs comes from  \eqref{defDnew}, so that the analogue of \eqref{deta} is 
 \begin{displaymath}
\begin{split}
& \mathcal{D}_{z, \eta}(\ell_1, \ell_2, h, N, M) \\
 &= \frac{1}{\Lambda} \sum_{\ell_1\ell_2 \mid c} w_0\left(\frac{c}{C}\right) \underset{d  \bmod{c}}{\left. \sum \right.^{\ast}} \sum_{n, m} \lambda_1(m)\lambda_2(n) e\left(\frac{d}{c}(\ell_1n \mp \ell_2m - h)\right)     W_{\eta  M}\left(\pm \frac{\ell_1 n -h}{ M}\right)V_{\pm z, \eta M}\left(\frac{\ell_2 m}{M}\right) 
\end{split} 
 \end{displaymath}
 The other source of extra signs are the two applications of Lemma \ref{vor1} in the situation of \eqref{firstvor} and \eqref{secondvor}. In \eqref{secondvor}, we encounter integral transforms of the shape
$$ W_{\eta M}^{\ast}\left(\frac{h\ell_1 n}{c^2}, \pm \frac{M\ell_1n}{c^2}\right)$$
where $$W^{\ast}_{\eta M}(z, w) =  \int_0^{\infty} W_{\eta M}(y) \mathcal{J}_{2it}^{\pm}(4\pi\sqrt{yw + z}) dy.$$
Here $w$ can be negative, but by \eqref{sizeh} we can guarantee $5|w| \leq z$, and we always have $z \geq (4C^2)^{-1}$. In particular, we can add a smooth redundant weight function $W_0( h\ell_1 nc^{-2}, \pm  M\ell_1n c^{-2})$ with $z_0 = (4C^2)^{-1}$ as in the remark after Corollary \ref{cor8} without changing the expression. 

Now Lemma \ref{analyzeW} applies to the relevant integral transform with $\mathcal{J}^+_{2it}$ in place of $J_{\kappa-1}$ (here we assume the Selberg eigenvalue conjecture, i.e.\ $t \in \mathbb{R}$, for convenience). For $\mathcal{J}^-_{2it}$ one can simply use the rapid decay of the Bessel-$K$-function to obtain a trivial decomposition of the type \eqref{decompW} with
\[ W_+(z, w) = W^{\ast}(z, w) e(-2\sqrt{z})= \int_0^{\infty} W(y) \mathcal{J}_{2it}^{-}(4\pi\sqrt{yw + z}) dy \, e(-2\sqrt{z}) \]
and $W_-(z, w) = 0$. This satisfies the stronger bound
 \begin{equation}\label{boundWpm2}
   z^i |w|^j \frac{\partial^i}{\partial z^i}    \frac{\partial^j}{\partial w^j} W_{\pm}(z, w) \begin{cases}
   = 0, & z  \geq C^{\varepsilon},\\
   \ll C^{\varepsilon(i+j)} , & \text{otherwise.}\end{cases}
 \end{equation}
 for any   $i, j \in \mathbb{N}_0$. 
 
Hence in the case of terms involving  $\mathcal{J}_{2it}^{+}$ in the application of Lemma~\ref{vor1} to \eqref{secondvor}, the ranges in \eqref{sizes} remain the same. In the case of terms involving $\mathcal{J}_{2it}^{-}$, we have even stronger conditions
\begin{equation}\label{sizes1}
\ell_1 n \leq \mathcal{N}_0^{-} := \frac{C^{2+\varepsilon}}{N}, \quad \ell_2m \leq \mathcal{M}_0
\end{equation}
from \eqref{boundWpm2}. At the end of subsection~\ref{VoronoiSummationSubsection}, we thus end up with the spectral analysis of terms involving six types of Kloosterman sums:
\def\memph#1{\ \hbox to .65in{\emph{#1}\hfill}}
\begin{itemize}
\item \memph{Case I:} $S(\ell_1 n - \ell_2m, h, c)$, $\ell_1 n > \ell_2 m$. This is the case of $\Sigma_+$.
\item \memph{Case II:} $S(\ell_1 n - \ell_2m, h, c)$, $\ell_1 n < \ell_2 m$. This is the case of $\Sigma_-$.
\item \memph{Case III:} $S(-\ell_1 n - \ell_2 m, h, c)$ under the size constraint  \eqref{sizes1}.
\item \memph{Case IV:}  $S(\ell_1 n + \ell_2 m, h, c)$.
\item \memph{Case V:} $S(-\ell_1 n + \ell_2 m, h, c)$, $\ell_1 n > \ell_2 m$ under the size constraint  \eqref{sizes1}.
\item \memph{Case VI:}  $S(-\ell_1 n + \ell_2 m, h, c)$, $\ell_1 n < \ell_2 m$ under the size constraint  \eqref{sizes1}.
\end{itemize}
Case IV is identical to Case I with minor sign changes. Cases III, V and VI are much simpler than Cases I and II because of the stronger size conditions \eqref{sizes1} (coming from the rapid decay of the Bessel $K$-function), but formally one can treat Case VI as Case I using the same sign Kuznetsov formula, and Cases III and V as Case II using the opposite sign formula. Note that $\mathcal{N}_0^- \leq \mathcal{M}_0$, so that in the notation of Sections \ref{SpectralDecompositionSection} and \ref{ShiftedSumsAverageSection} we automatically have $\mathcal{M}, \mathcal{K}, \mathcal{N} \leq \mathcal{M}_0$ if \eqref{sizes1} holds. \qed

\section{Proof of Theorem \ref{theorem-rudnicksound}}   
 
Let $V$ be a fixed smooth function that is 1 on $[0, 1]$ and vanishes on $[2, \infty)$. Let
\[ X := q^{1/1000}. \]
Define
\[ A(\chi) := \sum_{a, b}  \frac{\lambda_1(a)\lambda_2(b) (\chi(a) \bar{\chi}(b) + \bar{\chi}(a) \chi(b))}{\sqrt{ab}}V\left(\frac{ab}{X}\right). \]
 By the Cauchy-Schwarz inequality, we have
\begin{displaymath}
\begin{split}
&\Bigl|  \sumstar_{\chi \bmod{q}}  L(1/2, f_1 \otimes \chi) \overline{L(1/2, f_2 \otimes \chi)} A(\chi) \Bigr|^2  \leq \sumstar_{\chi \bmod{q}}  \bigl(L(1/2, f_1 \otimes \chi) \overline{L(1/2, f_2 \otimes \chi)}\bigr)^2  \sumstar_{\chi \bmod{q}}    A(\chi)^2.
 \end{split}
 \end{displaymath}
Note that both $A(\chi)$  and $L(1/2, f_1 \otimes \chi) \overline{L(1/2, f_2 \otimes \chi)}$ are real (cf.\ \eqref{afe1}), so that we do not need absolute values on the right hand side. We conclude
\begin{equation}\label{CS}
 \sumstar_{\chi \bmod{q}}  \Bigl(L(1/2, f_1 \otimes \chi) \overline{L(1/2, f_2 \otimes \chi)}\Bigr)^2 \geq \frac{|S_1|^2}{S_2},
\end{equation}
where
\[ S_1 := \sumstar_{\chi \bmod{q}}  L(1/2, f_1 \otimes \chi) \overline{L(1/2, f_2 \otimes \chi)} A(\chi), \quad S_2 :=   \sumstar_{\chi \bmod{q}}    A(\chi)^2. \]
Clearly,
\[ S_2 =  2\sum_{d \mid q} \phi(d) \mu(q/d)\Bigl( \sum_{\substack{a_1a_2 \equiv b_1b_2 \bmod{d}\\ (a_1a_2 b_1b_2, p) = 1}} + \sum_{\substack{a_1b_2 \equiv a_2b_1 \bmod{d}\\ (a_1a_2 b_1b_2, p) = 1}}\Bigr) \frac{\lambda_1(a_1)\lambda_1(a_2) \lambda_2(b_1)\lambda_2(b_2)}{\sqrt{a_1a_2b_1b_2}} V\left(\frac{a_1b_1}{X}\right) V\left(\frac{a_2b_2}{X}\right). \]
Here $ d \in \{q, q/p\}$, and the support of $V$ implies that the congruences are equalities, so that
\[ S_2 = 2 \psi(q)  \Bigl( \sum_{\substack{a_1a_2 =b_1b_2  \\ (a_1a_2 b_1b_2, p) = 1}} + \sum_{\substack{a_1b_2 = a_2b_1 \\ (a_1a_2 b_1b_2, p) = 1}}\Bigr) \frac{\lambda_1(a_1)\lambda_1(a_2) \lambda_2(b_1)\lambda_2(b_2)}{\sqrt{a_1a_2b_1b_2}} V\left(\frac{a_1b_1}{X}\right) V\left(\frac{a_2b_2}{X}\right) = S_{21} + S_{22}, \]
say. By Mellin inversion we have
\[ S_{21} = 2 \psi(q)  \int_{(1)} \int_{(1)} \widehat{V}(s)\widehat{V}(t) X^{s+t} \sum_{\substack{a_1a_2 =b_1b_2  \\ (a_1a_2 b_1b_2, p) = 1}}\frac{\lambda_1(a_1)\lambda_1(a_2) \lambda_2(b_1)\lambda_2(b_2)}{(a_1b_1)^{s+\frac{1}{2}} (a_2b_2)^{t+\frac{1}{2}}} \frac{ds \, dt}{(2\pi i)^2}. \] 
In $\Re s, \Re t > -1/10$, say, the double Dirichlet series can be expanded into an Euler product: 
 \begin{displaymath}
\begin{split}
\prod_{p \nmid \ell}\Bigl(1 &+ \frac{2 \lambda_1(\ell) \lambda_2(\ell) }{\ell^{1+s + t}} + \frac{ \lambda_1(\ell) \lambda_2(\ell) }{\ell^{1+2s}} + \frac{ \lambda_1(\ell) \lambda_2(\ell) }{\ell^{1+2t}}  + O\bigl(\frac{1}{\ell^{3/2}}\bigr)\Bigr)  \\
  & = L(1 + s + t, f_1 \times f_2)^2 L(1 + 2s, f_1 \times f_2) L(1 + 2t, f_1 \times f_2) H_{21}(s, t)
  \end{split}
  \end{displaymath}
with a holomorphic Euler product $H_{21}(s, t) $ that converges absolutely  in $\Re s, \Re t > -1/10$ and   is uniformly  bounded (from above and beyond) in $q$ in the same vertical strip. Shifting contours, we find that
\[ S_{21}= 2 \psi(q) L(1, f_1 \times f_2)^4H(0, 0) + O(\psi(q) X^{-1/10}). \]
Similarly, we have
\begin{displaymath}
\begin{split}
S_{22} =  2 \psi(q)  \int_{(1)} \int_{(1)}  &  L(1 + s + t, f_1 \times f_1) L(1 + s + t, f_2 \times f_2)  L(1 + 2s, f_1 \times f_2) L(1 + 2t, f_1 \times f_2) \\
& \times \widehat{V}(s)\widehat{V}(t) X^{s+t}H_{22}(s, t) \frac{ds \, dt}{(2\pi i)^2}.
\end{split}
\end{displaymath}
The integrand in this double integral has a double pole at $s+t=0$ and two simple poles at $s=0$ and $t=0$.  We first shift to $\Re s, \Re t = 1/20$. Then we shift to $\Re s = -1/10$, picking up two poles at $s=0$ at $s = -t$,  and then to $\Re t = -1/10$ picking up one pole at $t=0$. In this way we obtain   
 $S_{22} \asymp \psi(q) (\log X)^2 \asymp \psi(q) (\log q)^2$, and we conclude
\begin{equation}\label{ss2}
S_2 \ll \psi(q) (\log q)^2.
\end{equation} 

\bigskip

Next we turn to the analysis of $S_1$. Here we use the approximate functional equation \eqref{afe1} to write
\[ S_1 = 2\sum_{d \mid q} \phi(d) \mu(q/d)\Bigl( \sum_{\substack{a_1a_2 \equiv b_1b_2 \bmod{d}\\ (a_1a_2 b_1b_2, p) = 1}} + \sum_{\substack{a_1b_2 \equiv a_2b_1 \bmod{d}\\ (a_1a_2 b_1b_2, p) = 1}}\Bigr) \frac{\lambda_1(a_1)\lambda_1(a_2) \lambda_2(b_1)\lambda_2(b_2)}{\sqrt{a_1a_2b_1b_2}} W\left(\frac{a_1b_1}{q^2}\right) V\left(\frac{a_2b_2}{X}\right). \]
Note that this has (by design) the same shape as $S_2$, except that the range of summation of the $a_1, b_1$ variables is much longer. We decompose $S_1 = M_1 + E_1$ where $M_1$ represents the diagonal contributions and $E_1$ is the rest. By the same argument as before, we find that
\begin{equation}\label{M1}
M_1 \asymp\psi(q)  \log X \log q \asymp \psi(q) (\log q)^2.
\end{equation}
For the error term, we first estimate trivially (using \eqref{deligne} for convenience)
\[ E_1 \ll q \sum_{d \in \{q, q/p\}}  \sum_{\substack{a_2, b_2 \ll X\\(a_2b_2, p) = 1}} \frac{1}{(a_2b_2)^{1/2 - \varepsilon}} \Bigl| \Bigl(\sum_{\substack{a_1a_2 \equiv b_1b_2 \bmod{d}\\ (a_1  b_1, p) = 1\\ a_1a_2 \not= b_1b_2}} + \sum_{\substack{a_1b_2 \equiv a_2b_1 \bmod{d}\\ (a_1  b_1 , p) = 1\\ a_1b_2 \not= a_2b_1}}\Bigr) \frac{\lambda_1(a_1)\lambda_2(b_1)  }{\sqrt{a_1 b_1 }} W\left(\frac{a_1b_1}{q^2}\right) \Bigr|. \]
We show in detail how to treat the first term, since the second one is very similar. Injecting a smooth partition of unity and arguing as in the beginning of Section \ref{off}, we need to estimate
\[ E(A, B) := \frac{q}{\sqrt{AB}} \sum_{d \in \{q, q/p\}}  \sum_{\substack{a_2, b_2 \ll X\\(a_2b_2, p) = 1}} \frac{1}{(a_2b_2)^{1/2 - \varepsilon}} \Bigl|\sum_{\substack{a_1a_2 \equiv b_1b_2 \bmod{d}\\ (a_1  b_1, p) = 1\\ a_1a_2 \not= b_1b_2}} \lambda_1(a_1)\lambda_2(b_1) V_1\left(\frac{a_1}{A}\right) V_2\left(\frac{b_1}{B}\right)\Bigr| \]
for smooth compactly supported weight functions $V_1, V_2$ satisfying \eqref{derivative}, and $AB \ll q^{2+\varepsilon}$. Without loss of generality, consider the case $B \geq A$. We estimate $E(A,B)$ in two ways. First, we remove the coprimality condition by M\"obius inversion, getting
\begin{displaymath}
\begin{split}
E(A, B) &\ll  \frac{q}{\sqrt{AB}} \sum_{d \in \{q, q/p\}}  \sum_{\substack{a_2, b_2 \ll X\\(a_2b_2, p) = 1}}  \sum_{f \mid g \mid p}  \frac{| \lambda_1(g/f)|}{(a_2b_2)^{1/2 - \varepsilon}} \Bigl|  \sum_{\substack{fga_1a_2 \equiv b_1b_2 \bmod{d}\\  fga_1a_2 \not= b_1b_2}}   \lambda_1(a_1)\lambda_2(b_1)   V_1\left(\frac{fga_1}{A}\right) V_2\left(\frac{b_1}{B}\right)\Bigr|  \\
& \ll \frac{q}{\sqrt{AB}} \sum_{d \in \{q, q/p\}}  \sum_{\substack{a_2, b_2 \ll X\\(a_2b_2, p) = 1}}  \sum_{f \mid g \mid p}  \frac{1}{(a_2b_2)^{1/2 - \varepsilon}}  |\mathcal{S}(fga_2, b_2, d, Aa_2, Bb_2)|,
\end{split}
\end{displaymath}
using the notation \eqref{average}. Provided $A \gg BX$ or $B \gg AX$ with a sufficiently large implied constant, we find  by Proposition \ref{prop3} that 
\begin{equation}\label{E1}
\begin{split}
E(A, B) &\ll \frac{q^{1+\varepsilon}X}{\sqrt{AB}} \left(\frac{BX}{q^{1/2}} + \frac{B^{5/4} A^{1/4}X^{3/2}}{q} + \frac{B^{3/4} A^{1/4}X}{q^{1/4}} + \frac{BA^{1/2}X^{3/2}}{q^{3/4}}\right)\\
& \ll \frac{q^{1+\varepsilon}}{\sqrt{AB}} X^{5/2}   \left(\frac{ A^{1/4} B^{3/4}}{q^{1/4}} + \frac{B}{q^{1/2}}\right)
\end{split}
\end{equation}
since  $AB \leq q^{2+\varepsilon}$.   If $A \ll BX \ll A X^2$, we have the individual bound
\begin{equation}\label{E2}
E(A, B) \ll \frac{q^{1+\varepsilon}X}{\sqrt{AB}}  X^{5/2} B^{1/2 + \theta}.
\end{equation}
Note that, up to powers of $X$, this is comparable to \eqref{shiftedbound} and \eqref{auxbound} with $N$ and $M$ replaced with $B$ and $A$ respectively. 

Alternatively, we write 
\[ E(A, B) =  \frac{q}{\sqrt{AB}}\sum_{d \in \{q, q/p\}}  \sum_{\substack{a_2, b_2 \ll X\\(a_2b_2, p) = 1}} \frac{1}{(a_2b_2)^{1/2 - \varepsilon}} \Bigl|\sum_{p \nmid a_1} \lambda_1(a_1) V_1\left(\frac{a_1}{A}\right) \sum_{b_1 \equiv a_2\overline{b_2} a_1 \bmod{d}} \lambda_2(b_1) V_2\left(\frac{b_1}{B}\right)\Bigr|. \]
Arguing as in Section \ref{HeckeResidueClasses}, the innermost sum equals
\[ \frac{1}{d} \sum_{r \mid d} \frac{B}{r} \sum_{b} S(a_2 \overline{b_2} a_1, b, r) \lambda_2(b) \mathring{V}_2\left(\frac{bB}{r^2}\right), \]
 and hence, by an application of the Cauchy-Schwarz inequality and \eqref{deligne},
 \[ E(A, B) \ll q \sum_{d \in \{q, q/p\}}  \sum_{\substack{a_2, b_2 \ll X\\(a_2b_2, p) = 1}} \frac{1}{(a_2b_2)^{1/2 - \varepsilon}} \sum_{r \mid d} \frac{B^{1/2}}{dr} \Bigl( \sum_{n_1, n_2 \ll r^2q^{\varepsilon}/B} |\mathcal{S}_A(a_2 \overline{b_2}  n_1, a_2 \overline{b_2} n_2, r)|\Bigr)^{1/2}, \]
 where, as in \eqref{DefMathcalSM}, we write
\[ \mathcal{S}_A(a_2 \overline{b_2}  n_1, a_2 \overline{b_2} n_2, r) = \sum_{\substack{m \asymp A\\ (m, p) = 1}} S(m, a_2 \overline{b_2}  n_1,  r) S(m, a_2 \overline{b_2} n_2, r). \]
By Theorem \ref{klooster-short}, we conclude as in the proof of Proposition \ref{bound1} that
\begin{equation}\label{E3}
E(A, B) \ll \frac{q^{1+\varepsilon}}{B^{1/2}}X\left(A^{1/4} q^{7/12} + A^{1/2} q^{5/12} + q^{2/3}\right).
\end{equation}
Combining \eqref{E1}, \eqref{E2}, and \eqref{E3}, we conclude as in Section \ref{opti} that
\[ E(A, B) \ll q^{65/66 + \varepsilon} X^{5/2}. \]
Together with \eqref{M1}, this estimate shows that
\begin{equation}
\label{EstimateS1}
S_1 \gg \psi(q) (\log q)^2.
\end{equation}
Combining \eqref{CS}, \eqref{ss2}, and \eqref{EstimateS1}, we complete the proof of Theorem~\ref{theorem-rudnicksound}. \qed

\end{document}